\documentclass[12pt]{amsart}

\usepackage{calligra,mathrsfs}
\usepackage[all]{xy}
\usepackage{float, comment}
\usepackage{mathtools}
\usepackage{amsmath}
\usepackage{amsthm}
\usepackage{amssymb}
\usepackage{amsbsy}
\usepackage{amstext}
\usepackage{amsopn}
\usepackage{mathrsfs} % allows \mathscr
\usepackage{enumerate}
\usepackage{xcolor}
\usepackage{graphicx} % allows \includegraphics{}s
\usepackage{microtype} % improves formattings
\usepackage[margin=1in,marginparwidth=0.8in, marginparsep=0.1in]{geometry}
\usepackage[pagebackref, bookmarks=true, bookmarksopen=true, bookmarksdepth=3,bookmarksopenlevel=2, colorlinks=true, linkcolor=blue, citecolor=blue, filecolor=blue, menucolor=blue, urlcolor=blue]{hyperref}
\usepackage{tikz}
\usepackage{bbm}
\usepackage[all]{xy}
\usetikzlibrary{arrows,calc,positioning,decorations.pathreplacing} %,arrows.meta}

\DeclareMathOperator{\cHom}{\mathscr{H}\text{\kern -3pt {\calligra\large om}}\,}

\numberwithin{equation}{section}
\newtheorem{Theorem}[equation]{Theorem}
\newtheorem{Proposition}[equation]{Proposition} 
\newtheorem{Lemma}[equation]{Lemma}

\newtheorem{Corollary}[equation]{Corollary}
\newtheorem{Conjecture}[equation]{Conjecture}

\theoremstyle{definition}
\newtheorem{Remark}[equation]{Remark}
\newtheorem{Example}[equation]{Example}
\newtheorem{Definition}[equation]{Definition}

\numberwithin{figure}{section}

\def\uk{\underline{k}}
\def\n{{\bf{n}}}
\def\bL{{\mathbb{L}}}
\def\Spec{{\rm{Spec}}}

\def\ue{{\underline{e}}}
\def\uj{{\underline{j}}}
\def\ux{{\underline{x}}}

\def\la{\langle}
\def\ra{\rangle}
\def\tD{\widetilde{D}}
\def\om{\overline{m}}
\def\l{\lambda}
\def\ttimes{\widetilde{\times}}
\def\tbox{\widetilde{\boxtimes}}
\def\bbox{{\boxtimes}}
\def\O{\mathcal{O}}

\newcommand{\gr}{\mathrm{gr}}
\newcommand{\mb}[1]{\mathbf{#1}}
\newcommand{\fsl}{\mathfrak{sl}}
\newcommand{\fg}{\mathfrak{g}}
\newcommand{\fn}{\mathfrak{n}}

\newcommand{\A}{\mathbb{A}}
\newcommand{\C}{\mathbb{C}}
\newcommand{\D}{\mathbb{D}}
\newcommand{\E}{\mathbb{E}}
\newcommand{\G}{\mathbb{G}}
\newcommand{\N}{\mathbb{N}}
\renewcommand{\P}{\mathbb{P}}
\newcommand{\Q}{\mathbb{Q}}
\newcommand{\R}{\mathbb{R}}
\newcommand{\Z}{\mathbb{Z}}

\newcommand{\cA}{\mathcal{A}}
\newcommand{\cB}{\mathcal{B}}
\newcommand{\cC}{\mathcal{C}}
\newcommand{\cCh}{\mathcal{C}^{ch}}

\newcommand{\cF}{\mathcal{F}}
\newcommand{\cG}{\mathcal{G}}
\newcommand{\cK}{\mathcal{K}}
\newcommand{\cL}{\mathcal{L}}
\newcommand{\cM}{\mathcal{M}}
\newcommand{\cN}{\mathcal{N}}
\newcommand{\cO}{\mathcal{O}}
\newcommand{\cP}{\mathcal{P}}

\newcommand{\cR}{\mathcal{R}}

\newcommand{\cT}{\mathcal{T}}
\newcommand{\cU}{\mathcal{U}}
\newcommand{\cV}{\mathcal{V}}

\newcommand{\shA}{\mathcal{A}}
\newcommand{\shB}{\mathcal{B}}
\newcommand{\shF}{\mathcal{F}}
\newcommand{\shG}{\mathcal{G}}

\newcommand{\coeffs}{\C}
\newcommand{\La}{\Lambda}
\newcommand{\lmat}{\lambda}

\newcommand{\wt}{\widetilde}
\newcommand{\grl}{\widehat} %\grl = graded lift
\newcommand{\Gm}{\G_m}
\newcommand{\conv}{{\mathbin{\scalebox{1.1}{$\mspace{1.5mu}*\mspace{1.5mu}$}}}}
\newcommand{\hconv}{\mathbin{\scalebox{.9}{$\nabla$}}}
\newcommand{\smallhconv}{\mathbin{\scalebox{.5}{$\nabla$}}}
\newcommand{\sconv}{\mathbin{\scalebox{.9}{$\Delta$}}}

\newcommand{\bigconv}{\mathop{\mbox{\huge$*$}}\limits}
\newcommand{\into}{\hookrightarrow}
\newcommand{\onto}{\twoheadrightarrow}
\newcommand{\im}{\mathrm{im}}
\newcommand{\id}{{id}}
\newcommand{\Id}{\mathrm{Id}}
\newcommand{\Gr}{\mathrm{Gr}}
\newcommand{\oGr}{\overline{\mathrm{Gr}}}
\newcommand{\Ran}{\mathrm{Ran}}
\newcommand{\QCoh}{D_{qcoh}}
\newcommand{\IndCoh}{\mathrm{IndCoh}}
\newcommand{\Rep}{\mathrm{Rep}}

\newcommand{\sodot}{\mathop{\mbox{\normalsize$\bigodot$}}\limits}

\DeclareRobustCommand{\SkipTocEntry}[5]{}

\DeclareMathOperator{\colim}{colim}
\DeclareMathOperator{\Hom}{Hom}
\DeclareMathOperator{\End}{End}

\DeclareMathOperator{\udim}{\underline{dim}}

\newcommand{\rmat}[1]{\mathbf{r}_{#1}}

\newcommand{\newword}[1]{\textbf{\emph{#1}}}

\newcommand{\arrtip}{latex'}

\begin{document}
\title{Cluster Theory of the Coherent Satake Category}

\author[Sabin Cautis]{Sabin Cautis}
\address[Sabin Cautis]{University of British Columbia \\ Vancouver BC, Canada}
\email{cautis@math.ubc.ca}

\author[Harold Williams]{Harold Williams}
\address[Harold Williams]{University of California \\ Davis CA, USA}
\email{hwilliams@math.ucdavis.edu}

\begin{abstract}
We study the category of $G(\cO)$-equivariant perverse coherent sheaves on the affine Grassmannian $\Gr_G$. This coherent Satake category is not semisimple and its convolution product is not symmetric, in contrast with the usual constructible Satake category. Instead, we use the Beilinson-Drinfeld Grassmannian to construct renormalized $r$-matrices. These are canonical nonzero maps between convolution products which satisfy axioms weaker than those of a braiding.

We also show that the coherent Satake category is rigid, and that together these results strongly constrain its convolution structure. In particular, they can be used to deduce the existence of (categorified) cluster structures. We study the case $G = GL_n$ in detail and prove that the $\Gm$-equivariant coherent Satake category of $GL_n$ is a monoidal categorification of an explicit quantum cluster algebra. 

More generally, we construct renormalized $r$-matrices in any monoidal category whose product is compatible with an auxiliary chiral category, and explain how the appearance of cluster algebras in 4d $\cN=2$ field theory may be understood from this point of view.
%%%%%%%%%%%%%%%%%%%%%%%%%%
\end{abstract}

\maketitle

\setcounter{tocdepth}{1}

\tableofcontents

\section{Introduction}
\thispagestyle{empty}

Let $G$ be a complex reductive group. The Satake category $\cP^{G(\cO)}(\Gr_G)$ of $G(\cO)$-equivariant perverse sheaves on the affine Grassmannian $\Gr_G$ plays a fundamental role in geometric representation theory and, in particular, the geometric Langlands program. Its structure is well understood via the geometric Satake equivalence of \cite{Lus83,Gin95,MV}, which states that $\cP^{G(\cO)}(\Gr_G)$ is monoidally equivalent to the representation category of the Langlands dual group $G^\vee$. 

Like its constructible counterpart, the derived category of $G(\cO)$-equivariant coherent sheaves on $\Gr_G$ has a perverse $t$-structure \cite{AB10, BFM} which is stable under convolution. This gives us the {\it coherent} Satake category $\cP_{coh}^{G(\cO)}(\Gr_G)$. In contrast with $\cP^{G(\cO)}(\Gr_G)$ this monoidal category is not semi-simple and is poorly understood.

The purpose of this paper is to establish the existence of duals and renormalized $r$-matrices in $\cP_{coh}^{G(\cO)}(\Gr_G)$, and to explain how these lead to its structure being governed by the theory of cluster algebras. This in turn makes visible the physical meaning of $\cP_{coh}^{G(\cO)}(\Gr_G)$: its simple objects correspond to Wilson-'t Hooft line operators in 4d $\cN=2$ gauge theory, just as those of $\cP^{G(\cO)}(\Gr_G)$ correspond to 't Hooft operators in 4d $\cN=4$ gauge theory \cite{KW07}. 
%This in turn paints a compelling picture of $\cP_{coh}^{G(\cO)}(\Gr_G)$ as a fundamental object in 4d $\cN=2$ gauge theory, much as $\cP^{G(\cO)}(\Gr_G)$ plays a key role in 4d $\cN=4$ gauge theory \cite{KW07}.

\addtocontents{toc}{\SkipTocEntry}
\subsection*{The coherent Satake category}
%Recall that $G(\O)$-orbits in $\Gr_G$ are indexed by coweights $\l^\vee$ of $G$. Our starting point is that the closure $\oGr_G^{\l^\vee} \subset \Gr_G$ of such an orbit is a finite union of orbits whose dimensions are either all odd or all even. As observed in \cite{AB10}, this implies that $\cP_{coh}^{G(\cO)}(\Gr_G)$ is finite length, as is the $\Gm$-equivariant coherent Satake category $\cP_{coh}^{G(\cO) \rtimes \Gm}(\Gr_G)$ in which we further consider the action of loop rotation on $\Gr_G$.

Our starting point is the fact that $\cP_{coh}^{G(\cO)}(\Gr_G)$ is finite length, as is the $\Gm$-equivariant coherent Satake category $\cP_{coh}^{G(\cO) \rtimes \Gm}(\Gr_G)$ in which we further consider the action of loop rotation on $\Gr_G$ \cite{AB10}. In either setting the simple objects can be explicitly enumerated: they are in bijection with isomorphism classes of simple equivariant vector bundles on $G(\cO)$-orbits. Recall that $G(\O)$-orbits $\Gr_G^{\l^\vee}$ in $\Gr_G$ are indexed by coweights $\l^\vee$ of $G$. The most basic simple objects are the restrictions of line bundles on $\Gr_G$ to some $\oGr_G^{\l^\vee}$, shifted to lie in cohomological degree $- \frac12 \dim \Gr_G^{\l^\vee}$ \cite{Mir} (as in \cite{BFM} we allow formal half-integer degrees). 

%In either setting the simple objects can be explicitly enumerated: they are in bijection with isomorphism classes of simple equivariant vector bundles on $G(\cO)$-orbits. The simplest example is the restriction of a line bundle on $\Gr_G$ to $\oGr_G^{\l^\vee}$, shifted so that it lies in cohomological degree $- \frac12 \dim \Gr_G^{\l^\vee}$ \cite{Mir} (as in \cite{BFM} we allow formal half-integer degrees). 

In general, a simple object with support $\oGr_G^{\l^\vee}$ restricts to a shifted equivariant vector bundle on $\Gr_G^{\l^\vee}$ (one direction of the above bijection) but a priori lacks an elementary description itself.
However, we will be able to deduce many nontrivial properties of convolution products of simple objects using the following structure.

\begin{Theorem}[c.f. Theorem \ref{thm:satakermats}]\label{thm:intrormats}
Given nonzero objects $\cF, \cG \in D^{G(\O)}_{coh}(\Gr_G)$ there is a canonical nonzero map $\rmat{\cF,\cG}: \cF \conv \cG \to \cG \conv \cF$, which we call a renormalized $r$-matrix. These maps are not isomorphisms in general and are not natural in the categorical sense, but satisfy a collection of axioms which control their failure to constitute a braiding.
\end{Theorem}

The construction and characteristic properties of the maps $\rmat{\cF,\cG}$ are similar to those of the renormalized $r$-matrices which appear in the representation theory of quantum loop algebras \cite{Kas02} and KLR (or quiver Hecke) algebras \cite{KKK13}, hence our terminology. Their existence is closely related to the commutativity constraint which appears in the proof of the geometric Satake equivalence. 

Adding in loop rotations, there exist $\La(\cF,\cG) \in \Z$ so that these $r$-matrices lift to maps 
$$\rmat{\cF,\cG}: \cF \conv \cG \to \cG \conv \cF \{ \La(\cF,\cG) \}$$
in $D^{G(\O) \rtimes \Gm}_{coh}(\Gr_G)$, where $\{ \cdot\, \}$ denotes a $\Gm$-equivariance shift. The axioms alluded to above can be summarized as saying that renormalized $r$-matrices deviate from satisfying the axioms of a braiding in a precise way controlled by certain inequalities among these degrees (see Definition \ref{def:rmat}). Note that it would be impossible to define a true braiding on $D^{G(\O)}_{coh}(\Gr_G)$ since $\cF \conv \cG$ and $\cG \conv \cF$ are generally not isomorphic. 

\begin{Theorem}[c.f. Theorem \ref{thm:adjoints}]\label{thm:introduals}
The monoidal category $D^{G(\O) \rtimes \Gm}_{coh}(\Gr_G)$ is rigid. The left and right duals of $\cF \in D_{coh}^{G(\O) \rtimes \Gm}(\Gr_G)$ are $\cF^L = \D(\cF^*)$ and $\cF^R = (\D\cF)^*$, where $\cF \mapsto \cF^*$ is an involution and $\D\cF = \cHom(\cF, \omega)$, $\omega$ denoting the dualizing sheaf. The duals of a perverse coherent sheaf are again perverse, so $\cP_{coh}^{G(\cO) \rtimes \Gm}(\Gr_G)$ is also rigid. 
\end{Theorem}

This result turns out to be more complicated than its counterpart in the constructible case, where the existence of a fiber functor and commutativity constraint reduce rigidity to an elementary check \cite{DM82,BR17}. 

As in the setting of quantum loop algebras and KLR algebras \cite{KKKO15a}, renormalized $r$-matrices together with certain properties implied by rigidity strongly constrain the behavior of real simple objects -- that is, objects whose convolution square is again simple (see Theorem~\ref{thm:consequences}). %For example, if one of a pair of simple perverse coherent sheaves is real then their convolution product must have a simple head and a simple socle.
On the other hand, the coherent Satake category has an abundance of real objects, namely the shifted restrictions of equivariant line bundles discussed above. 

Since $\cP_{coh}^{G(\cO) \rtimes \Gm}(\Gr_G)$ is finite length, the classes of simple objects form a canonical basis in the Grothendieck ring  $K^{G(\cO) \rtimes \Gm}(\Gr_{G})$. Following the framework developed in \cite{KKKO18} to study the upper global or dual canonical bases of quantum groups \cite{Lus90, Kas91}, Theorems~\ref{thm:intrormats} and \ref{thm:introduals} provide a mechanism by which this basis may come to include the structure of a quantum cluster algebra.

\addtocontents{toc}{\SkipTocEntry}
\subsection*{Cluster algebras} A cluster algebra $\cA_{\wt{B}}$ is a ring with a partial basis of a certain form: it contains special elements (cluster variables) grouped into overlapping subsets (clusters) such that the monomials in any cluster are again basis elements \cite{FZ02}.  Any two clusters are connected by a sequence of mutations, an operation which creates a new cluster by exchanging a single cluster variable for a new one. The form of these exchanges, hence the entire structure, is encoded by a skew-symmetrizable matrix $\wt{B}$. 
In \cite{FKRD15} it was shown that $\C \otimes K^{GL_n(\cO)}(\Gr_{GL_n})$ is isomorphic with the coordinate ring of a unipotent cell in $LSL_2$, hence possesses a cluster structure following \cite{GLS11}. 

Monoidal categorifications of cluster algebras were considered in \cite{HL10} and its successors \cite{HL13, Nak11, KQ14, Qin17} in the context of quantum loop algebras. The terminology is meant to distinguish from descriptions of cluster algebras via homological (not monoidal) properties of quiver representations \cite{CC06,BMRRT06,GLS06,CK06,Pla11,Ami09}. %The work \cite{HL10} and its successors \cite{HL13, Nak11, KQ14, Qin17} identified certain cluster algebras as Grothendieck rings of subcategories of $\cU_qL\fg \textrm{-mod}$, the representation category of a quantum loop algebra. They further showed the partial basis of cluster monomials is composed of classes of real simple representations.

%Monoidal categorifications of cluster algebras were first considered in \cite{HL10}, the terminology being meant to distinguish from descriptions of cluster algebras in terms of homological (rather than monoidal) properties of quiver representations \cite{CC06,BMRRT06,GLS06,CK06,Pla11,Ami09}. The work \cite{HL10} and its successors \cite{HL13, Nak11, KQ14, Qin17} identified certain cluster algebras as Grothendieck rings of subcategories of $\cU_qL\fg \textrm{-mod}$, the representation category of a quantum loop algebra. They further showed the partial basis of cluster monomials is composed of classes of real simple representations.

The algebra $\cA_{\wt{B}}$ has a family of $q$-deformations $\cA_{(\wt{B},L)}$ (quantum cluster algebras) depending on an auxiliary matrix $L$, and in which the elements of any cluster have distinguished pairwise $q$-commuting lifts called quantum cluster variables \cite{BZ05}. The proof in \cite{KKKO18} of the existence of a quantum cluster structure on the upper global basis of a quantum unipotent coordinate algebra proceeds by establishing that one obtains monoidal quantum cluster categorifications from module categories of KLR algebras \cite{KL09,Rou08}.

\addtocontents{toc}{\SkipTocEntry}
\subsection*{Coherent Satake as a monoidal cluster categorification}

For any particular $G$, to deduce a categorified cluster structure on $\cP_{coh}^{G(\cO) \rtimes \Gm}(\Gr_{G})$ from Theorems~\ref{thm:intrormats} and \ref{thm:introduals} one must further show, by hand, that a finite subset of would-be cluster variables and mutations in $K^{G(\cO) \rtimes \Gm}(\Gr_{G})$ lift to suitable real simple objects and exact sequences. In this paper we carry this out fully in the case $G = GL_n$, discussing the expected extension to other $G$ later in the introduction.

%Our main application of Theorems \ref{thm:intrormats} and \ref{thm:introduals} is the following result. 

\begin{Theorem}[c.f. Theorem \ref{thm:mainfinal}]\label{thm:mainthmintro}
The coherent Satake category $\cP_{coh}^{GL_n(\cO) \rtimes \Gm}(\Gr_{GL_n})$ is a monoidal categorification of a quantum cluster algebra $\cA^{\,loc}_{(L_n,\wt{B}_n)}$ with localized frozen variables. That is, there is a $\Z[q^{\pm 1/2}]$-algebra isomorphism of $\cA^{\,loc}_{(L_n,\wt{B}_n)}$ with $K^{GL_n(\cO) \rtimes \Gm}(\Gr_{GL_n})$ such that
\begin{enumerate}
\item quantum cluster variables are classes of certain real simple perverse coherent sheaves,
\item the convolution of two such sheaves in the same cluster is again simple and its isomorphism class is independent of the order up to a grading shift, and
\item the convolution in either order of such a sheaf and its mutation is the middle term of an exact sequence categorifying the relevant exchange relation.
\end{enumerate}
Here $\wt{B}_n$ is the signed adjacency matrix of the quiver in Figure \ref{fig1}. We may take as the initial cluster the classes of certain perverse coherent sheaves $\cP_{k,\ell}$ for $1 \leq k \leq n$ and $0 \leq \ell \leq 1$. These sheaves are the restrictions of the line bundles $\cO(\ell)$ to the closed orbits $\Gr_{GL_n}^{\omega_k^\vee}$ with a cohomological/equivariant shift of $[\frac12 \dim \Gr_{GL_n}^{\omega_k^\vee}]\{-\frac12 \dim \Gr_{GL_n}^{\omega_k^\vee}-k\ell\}$. 
\end{Theorem}

\begin{figure}
\begin{tikzpicture}[thick,>=\arrtip]
\newcommand*{\Ddotsdist}{2}
\newcommand*{\shft}{1}
\newcommand*{\DrawDots}[1]{
  \fill ($(#1) + .25*(\Ddotsdist,0)$) circle (.03);
  \fill ($(#1) + .5*(\Ddotsdist,0)$) circle (.03);
  \fill ($(#1) + .75*(\Ddotsdist,0)$) circle (.03);
}
\node [matrix] (Q) at (0,0)
{
\coordinate [label={below:$\cP_{1,0}$}] (2) at (0,0);
\coordinate [label={below:$\cP_{2,0}$}] (4) at (2,0);
\coordinate (6) at (4,0);
\coordinate (2n-4) at (6-\shft,0);
\coordinate [label={below:$\cP_{n-1,0}$}] (2n-2) at (8-\shft,0);
\coordinate [label={below:$\cP_{n,0}$}] (2n) at (10-\shft,0);

\coordinate [label={$\cP_{1,1}$}] (1) at (0,2);
\coordinate [label={$\cP_{2,1}$}] (3) at (2,2);
\coordinate (5) at (4,2);
\coordinate (2n-5) at (6-\shft,2);
\coordinate [label={$\cP_{n-1,1}$}] (2n-3) at (8-\shft,2);
\coordinate [label={$\cP_{n,1}$}] (2n-1) at (10-\shft,2);

\foreach \v in {1,2,3,4,2n-3,2n-2,2n-1,2n} {\fill (\v) circle (.06);};
\foreach \s/\t in {1/2,3/4,2n-3/2n-2} {
  \draw [->,shorten <=1.7mm,shorten >=1.7mm] ($(\s)+(0.06,0)$) to ($(\t)+(0.06,0)$);
  \draw [->,shorten <=1.7mm,shorten >=1.7mm] ($(\s)-(0.06,0)$) to ($(\t)-(0.06,0)$);
};
\foreach \s/\t in {2/3,4/5,2n-4/2n-3,2n-2/2n-1} {
  \draw [->,shorten <=1.7mm,shorten >=1.7mm] ($(\s)$) to ($(\t)$);
};
\foreach \s/\t in {4/1,6/3,2n-2/2n-5,2n/2n-3} {
  \draw [shorten <=1.7mm,shorten >=1.7mm] ($(\s)$) to ($(\s)!.5!(\t)$);
  \draw [->,shorten <=1.7mm,shorten >=1.7mm] ($(\s)!.5!(\t)$) to ($(\t)$);
};
\DrawDots{4-\shft*.5,1};\\
};
\node (equals) [left=0mm of Q] {$\wt{Q}_n = $};
\node (equalsr) [right=0mm of Q] {\quad\quad\quad};
\end{tikzpicture}
\caption{}\label{fig1}
\end{figure}

%We discuss the expected extension to other $G$ later in the introduction.
The $GL_n$ case is of particular interest because $K^{GL_n(\cO) \rtimes \Gm}(\Gr_{GL_n})$ is also the simplest ($A_1$) example of a quantized Coulomb branch of a 4d $\cN=2$ quiver gauge theory \cite{Nak16,BFN}. The complexified classical limits of such spaces are Poisson isomorphic to coordinate rings of open Richardson varieties in the affine flag variety \cite{FKRD15}, or equivalently unipotent cells of $LSL_2$ in the $A_1$ case. Given results of \cite{GLS13, KO17}, Theorem \ref{thm:mainthmintro} implies that moreover $K^{GL_n(\cO) \rtimes \Gm}(\Gr_{GL_n})$ and the quantized unipotent cell $A_q(N^{(s_0s_1)^n})$ are isomorphic after tensoring both with $\Q(q^{1/2})$. %Given the main result of \cite{GLS13}, Theorem \ref{thm:mainthmintro} implies that $K^{GL_n(\cO) \rtimes \Gm}(\Gr_{GL_n})$ becomes isomorphic with a localization of a subalgebra of the negative half of the quantum group $\cU_q(\widehat{L\fsl_2})$ after tensoring with $\Q(q^{1/2})$. Results of \cite{KO17} identify this localization with the relevant quantized unipotent cell $A_q(N^{(s_0s_1)^n})$.

It follows that a subalgebra of $\Q(q^{1/2}) \otimes_{\Z[q^{\pm 1/2}]} K^{GL_n(\cO) \rtimes \Gm}(\Gr_{GL_n})$ also admits a categorification via KLR algebras of type $A_1^{(1)}$. At present we do not know a precise relationship between this categorification and the category of perverse coherent sheaves.

From another point of view, Theorem \ref{thm:mainthmintro} is significant in that it asserts a precise way in which the geometry of the affine Grassmannian, viewed through the lens of the coherent Satake category, reflects enumerative aspects of an a priori unrelated 3-Calabi-Yau category. Indeed, the cluster algebra of any skew-symmetric matrix encodes (and is entirely determined by) invariants of moduli spaces related to the representation theory of the associated quiver, equipped with a suitable potential \cite{CC06,CK06,Ami09,Pla11}. 

In our case the relevant quiver with potential (without frozen variables) describes coherent sheaves on a particular toric Calabi-Yau 3-fold \cite{CDMMS}, or in principle the Fukaya category of its mirror, a 3-fold related to the periodic Toda system \cite{KS14,Smi15}. In this way Theorem~\ref{thm:mainthmintro} provides a concrete mathematical manifestation of geometric engineering \cite{KKV97}: the affine Grassmannian encodes various structures in pure 4d $\cN=2$ gauge theory \cite{Tel14,Cos14,BFN}, hence should in turn encode properties of the Calabi-Yau that engineers the theory.  

\addtocontents{toc}{\SkipTocEntry}
\subsection*{Example: $G = GL_2$}\label{ex:GL_2}
In this case Theorem \ref{thm:mainthmintro} appears implicitly in \cite{BFM}. The closed orbit $\Gr_{GL_2}^{\omega_1^\vee} \cong \P^1$ parametrizes $\C[[t]]$-lattices $L_1 \subset L_0:= \C^2[[t]]$ with one-dimensional quotients. There is a natural line bundle on this space whose fiber over the point $[L_1 \subset L_0]$ is $L_0/L_1$. We abuse notation a bit and denote this line bundle by $L_0/L_1$. 

Convolutions are defined via the variety $\Gr_{GL_2}^{\omega_1^\vee} \widetilde{\times} \Gr_{GL_2}^{\omega_1^\vee}$ of nested subspaces $L_2 \subset L_1 \subset L_0:= \C^2[[t]]$ in which both inclusions have one-dimensional cokernels. Forgetting $L_1$ defines the convolution map
$$\overline{m}: \Gr_{GL_2}^{\omega_1^\vee} \widetilde{\times} \Gr_{GL_2}^{\omega_1^\vee} \to \oGr_{GL_2}^{2\omega_1^\vee},$$
where $\oGr_{GL_2}^{2\omega_1^\vee}$ parametrizes lattices $L_2 \subset L_0$ with two-dimensional quotient. Note that 
$$\oGr_{GL_2}^{2\omega_1^\vee} = \Gr_{GL_2}^{2\omega_1^\vee} \sqcup \Gr_{GL_2}^{\omega_2^\vee}$$
where $\Gr_{GL_2}^{\omega_2^\vee}$ is the point orbit $[tL_0 \subset L_0]$. 

The sheaves appearing in Theorem \ref{thm:mainthmintro} are the $\ell = 0,1$ cases of 
$$\cP_{1,\ell} := \O_{\Gr_{GL_2}^{\omega_1^\vee}} \otimes (L_0/L_1)^{\ell}\, [\frac 12] \{- \frac12-\ell \} \in \cP_{coh}^{GL_2(\cO) \rtimes \Gm}(\Gr_{GL_2}).$$
The convolution $\cP_{1,1} \conv \cP_{1,-1}$ is the pushforward along $\overline{m}$ of the twisted outer product 
$$\cP_{1,1}\, \tbox\, \cP_{1,-1} \cong \O_{\Gr_{GL_2}^{\omega_1^\vee} \widetilde{\times} \Gr_{GL_2}^{\omega_1^\vee}} \otimes (L_0/L_1) \otimes (L_1/L_2)^{-1} [1]\{-1\}.$$ 
Now, let $D \subset \Gr_{GL_2}^{\omega_1^\vee} \widetilde{\times} \Gr_{GL_2}^{\omega_1^\vee}$ denote the locus where $L_2=tL_0$. This is a divisor and we have the standard short exact sequence 
$$0 \rightarrow \O_{\Gr_{GL_2}^{\omega_1^\vee} \widetilde{\times} \Gr_{GL_2}^{\omega_1^\vee}}(-D) \rightarrow \O_{\Gr_{GL_2}^{\omega_1^\vee} \widetilde{\times} \Gr_{GL_2}^{\omega_1^\vee}} \rightarrow \O_D \rightarrow 0.$$
We can rewrite the first term as
$$0 \rightarrow \O_{\Gr_{GL_2}^{\omega_1^\vee} \widetilde{\times} \Gr_{GL_2}^{\omega_1^\vee}} \otimes (L_0/L_1) \otimes (L_1/L_2)^{-1} \{-2\} \rightarrow \O_{\Gr_{GL_2}^{\omega_1^\vee} \widetilde{\times} \Gr_{GL_2}^{\omega_1^\vee}} \rightarrow \O_D \rightarrow 0,$$
using the convention that $t$ has weight 2 (our $\Gm$ will act by a cover of ordinary loop rotation). 
Shifting by $[1]\{-1\}$ and applying $\om_*$ we get an exact triangle
$$\cP_{1,1} \conv \cP_{1,-1} \{-2\} \to \cP_{1,0} \conv \cP_{1,0} \to \om_*(\O_D) [1]\{-1\}.$$
But $\om_*(\O_D)\cong \O_{\Gr_{GL_2}^{\omega_2^\vee}}$, which is the sheaf $\cP_{2,0}$ appearing in Theorem \ref{thm:mainthmintro}. This gives us the exact sequence 
\begin{equation}\label{eq:ses}
0 \to \cP_{2,0} \{-1\} \to \cP_{1,1} \conv \cP_{1,-1} \{-2\} \to \cP_{1,0} \conv \cP_{1,0} \to 0
\end{equation}
in $\cP_{coh}^{GL_2(\cO) \rtimes \Gm}(\Gr_{GL_2})$.
Up to convolution with invertible objects, the duality functor $\D$ (or the bar involution) takes this sequence to
\begin{equation}\label{eq:ses2}
0 \to \cP_{1,0} \conv \cP_{1,0} \to \cP_{1,-1} \conv \cP_{1,1} \{2\} \to  \cP_{2,0}\{1\} \to 0.
\end{equation}
As these sequences are not split, $\cP_{1,1} \conv \cP_{1,-1}$ and $\cP_{1,-1} \conv \cP_{1,1}$ are not isomorphic, even up to an equivariant shift.
However, rearranging the four maps above we obtain nonzero compositions
\begin{align*}
\rmat{\cP_{1,1}, \cP_{1,-1}}&: \cP_{1,1} \conv \cP_{1,-1} \to \cP_{1,0} \conv\cP_{1,0} \{2\}  \to \cP_{1,-1} \conv \cP_{1,1} \{4\},\\
\rmat{\cP_{1,-1}, \cP_{1,1}}&: \cP_{1,-1} \conv \cP_{1,1} \to \cP_{2,0}\{-1\} \to \cP_{1,1} \conv \cP_{1,-1} \{-2\}.
\end{align*}
These are basic examples of renormalized $r$-matrices (though this does not reflect the way they are defined in general).

Either of (\ref{eq:ses}) or (\ref{eq:ses2}) implies that the classes of the relevant sheaves in $K^{GL_2(\cO) \rtimes \Gm}(\Gr_{GL_2})$ satisfy the exchange relation associated to the mutation of the quiver
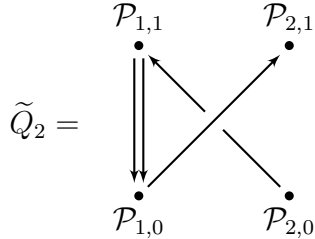
\[
\begin{tikzpicture}[thick,>=\arrtip]
\newcommand*{\Ddotsdist}{2}
\newcommand*{\shft}{1}
\newcommand*{\DrawDots}[1]{
  \fill ($(#1) + .25*(\Ddotsdist,0)$) circle (.03);
  \fill ($(#1) + .5*(\Ddotsdist,0)$) circle (.03);
  \fill ($(#1) + .75*(\Ddotsdist,0)$) circle (.03);
}
\node [matrix] (Q) at (0,0)
{
\coordinate [label={below:$\cP_{1,0}$}] (2) at (0,0);
\coordinate [label={below:$\cP_{2,0}$}] (4) at (2,0);
\coordinate (6) at (4,0);

\coordinate [label={$\cP_{1,1}$}] (1) at (0,2);
\coordinate [label={$\cP_{2,1}$}] (3) at (2,2);
\coordinate (5) at (4,2);

\foreach \v in {1,2,3,4} {\fill (\v) circle (.06);};
\foreach \s/\t in {1/2} {
  \draw [->,shorten <=1.7mm,shorten >=1.7mm] ($(\s)+(0.06,0)$) to ($(\t)+(0.06,0)$);
  \draw [->,shorten <=1.7mm,shorten >=1.7mm] ($(\s)-(0.06,0)$) to ($(\t)-(0.06,0)$);
};
\foreach \s/\t in {2/3} {
  \draw [->,shorten <=1.7mm,shorten >=1.7mm] ($(\s)$) to ($(\t)$);
};
\foreach \s/\t in {4/1} {
  \draw [shorten <=1.7mm,shorten >=1.7mm] ($(\s)$) to ($(\s)!.5!(\t)$);
  \draw [->,shorten <=1.7mm,shorten >=1.7mm] ($(\s)!.5!(\t)$) to ($(\t)$);
}; \\
};

\node (equals) [left=0mm of Q] {$\wt{Q}_2 = $};
\node (equalsr) [right=0mm of Q] {\quad\quad\quad};
\end{tikzpicture}
\]
at the vertex labeled by $\cP_{1,1}$, together with a suitable coefficient matrix. 
For $GL_2$ all unfrozen cluster variables turn out to be classes of sheaves supported on closed orbits (i.e. of the form $\cP_{1,\ell}$ for some $\ell \in \Z$). However, for $GL_n$ with $n > 2$ this is far from being the case. 

\addtocontents{toc}{\SkipTocEntry}
\subsection*{Chiral categories and renormalized $r$-matrices} 
The origin of renormalized $r$-matrices in $\cP_{coh}^{G(\cO) \rtimes \Gm}(\Gr_G)$ is the Beilinson-Drinfeld (BD) Grassmannian \cite{BD}. The BD Grassmannians of $\A^1$ and its powers are ind-schemes $\Gr_{G,\A^n}$ over $\A^n$ for $n > 0$. Collectively they form the prototypical example of a factorization space, meaning they satisfy certain compatibility conditions on their restrictions to diagonals and disjoint loci. 

To construct $r$-matrices we consider $\Gr_{G,\A^2}$. The factorization property in this case is that the fibers of $\Gr_{G,\A^2} \to \A^2$ are isomorphic to $\Gr_G$ over the diagonal and to $\Gr_G \times \Gr_G$ away from it. Given $\cF$, $\cG \in D_{coh}^{G(\cO) \rtimes \Gm}(\Gr_G)$, the factorization structure allows one to produce a pair of sheaves on $\Gr_{G,\A^2}$ which
\begin{enumerate}
\item over the diagonal of $\A^2$ have fibers isomorphic to $\cF \conv \cG$ and $\cG \conv \cF$, respectively, and
\item over the complement of the diagonal each have fibers isomorphic to $\cF \boxtimes \cG$.
\end{enumerate}
Restricting to $\Spec\, \C[t] \cong \A^1 \times \{0\} \subset \A^2$ one can consider the identity endomorphism of $\cF \boxtimes \cG$ over $\A^1 \smallsetminus \{0\}$. This defines a rational section of $\Hom$ between these sheaves. After multiplying by a suitable power of $t$ and restricting to the fiber over $0 \in \A^1$ we obtain a canonical homomorphism $\rmat{\cF, \cG}: \cF \conv \cG \to \cG \conv \cF$. This construction can be thought of as a meromorphic analogue of the Eckmann-Hilton argument.

We emphasize that it is only the factorization structure of the BD Grassmannian which is essential for defining these $r$-matrices. Abstractly, we only use the fact that we have a monoidal category $\cC$ which can be recovered from a suitably compatible chiral category on $\A^1$ as a subcategory of its fiber at $0$. A chiral category consists of a collection of sheaves of categories over the powers of $\A^1$ which satisfy analogous compatibilities to those of factorization spaces or chiral/factorization algebras \cite{Ras14}. 

In Section~\ref{sec:chiral} we describe the construction of renormalized $r$-matrices in this level of generality (c.f. Theorem \ref{thm:rmat}). In Section~\ref{sec:r-matrix} we distill from \cite{KKKO15a,KKKO18} a list of axioms (Definition \ref{def:rmat}) required of such a collection of morphisms for the proofs of the main results of loc. cited to carry over -- we call such a collection a system of renormalized $r$-matrices. We then recall following \cite{KKKO18} how renormalized $r$-matrices reduce the existence of iterated mutations to the existence of one-step mutations given the presence of a quantum cluster structure on the Grothendieck ring. After showing that the chiral $r$-matrices of Section \ref{sec:chiral} satisfy the conditions in Definition \ref{def:rmat} we thus obtain the following result.

\begin{Theorem}[c.f. Theorem \ref{thm:K3Omain}, Theorem \ref{thm:rmat}]\label{thm:chiralintro}
Let $\grl{\cC}$ be a graded monoidal abelian category contained in the subcategory of compact objects in the fiber at $0$ of a compatible $\Gm$-equivariant chiral category on $\A^1$. Let $(\{\cF_i\}_{i \in I}, \wt{B})$ be a quantum monoidal seed in $\grl{\cC}$ with coefficient matrix $L$. Suppose that 
\begin{enumerate}
\item $\grl{\cC}$ is rigid,
\item $(\{{\cF_i}\}_{i \in I}, \wt{B})$ admits a mutation in every unfrozen direction, and
\item there is an isomorphism $K_0(\grl{\cC}) \cong \cA_{(L,\wt{B})}$ taking the classes $\{[\cF_i]\}_{i \in I}$ to the initial quantum cluster variables.
\end{enumerate}
Then $\grl{\cC}$ is a monoidal categorification of the quantum cluster algebra $\cA_{(L,\wt{B})}$.
\end{Theorem}

Here a quantum monoidal seed is a collection of $q$-commuting real simple objects (that is, the products of any two in either order are isomorphic up to a grading shift) together with an exchange matrix (see Section \ref{sec:monoidalseeds}). Compatibility means in particular that the sections over $\A^1$ and $\A^2$ of the relevant chiral category possess monoidal structures compatible with the factorization structure and the monoidal structure on $\grl{\cC}$ (in natural examples one will in fact have a chiral $\E_1$-category, hence this condition will hold over $\A^n$ for all $n$).

While we believe formulating our construction at this level of abstraction (rather than in the specific context of the BD Grassmannian) provides more conceptual clarity, we also do so with an eye towards laying a foundation for future applications. In particular, one expects analogues of Theorem \ref{thm:mainthmintro} involving the $K$-theoretic Coulomb branches of \cite{BFN} and suitable abelian categories when the associated 4d $\cN=2$ theory possesses a BPS quiver \cite{ACCERV14}. While it is not clear how to define the correct $t$-structures, it is clear that a factorization structure is present and will be crucial for generalizing Theorem \ref{thm:mainthmintro}.

\addtocontents{toc}{\SkipTocEntry}
\subsection*{Computations for $GL_n$} 
Given Theorems \ref{thm:intrormats}, \ref{thm:introduals}, and \ref{thm:chiralintro}, the proof of Theorem \ref{thm:mainthmintro} reduces to two remaining steps. First, one needs to verify that the stated initial cluster consists of $q$-commuting real objects and to produce its mutations in all unfrozen directions.
Given $k \in [1,n], \ell \in \Z$ we let $\cP_{k,\ell} \in \cP_{coh}^{GL_n(\cO) \rtimes \Gm}(\Gr_{GL_n})$ denote the restriction of $\cO(\ell)$ to $\Gr_{GL_n}^{\omega_k^\vee}$, shifted by $\la \frac12 \dim \Gr_{GL_n}^{\omega_k^\vee} \ra\{-k\ell\}$ (where $\la 1 \ra$ is short for $[1]\{-1\}$). 
%In Sections~\ref{sec:commrel} and \ref{sec:mutrels} we establish the following properties of these sheaves.

\begin{Theorem}[c.f. Propositions \ref{prop:comm} and \ref{prop:mutation}]\label{thm:introGL_ncomps}
The perverse coherent sheaves $\cP_{k,\ell}$ are simple and real for all $k \in [1,n], \ell \in \Z$. If $|\ell_1-\ell_2| \le 1$ then $\cP_{k_1,\ell_1}$ and $\cP_{k_2,\ell_2}$ $q$-commute. Moreover, there exist exact sequences
\begin{gather*}
0 \rightarrow \cP_{k-1,\ell} \conv \cP_{k+1,\ell} \{-1\} \rightarrow \cP_{k,\ell+1} \conv \cP_{k,\ell-1} \{-2k\} \rightarrow \cP_{k,\ell} \conv \cP_{k,\ell} \rightarrow 0 \\
0 \rightarrow \cP_{k,\ell} \conv \cP_{k,\ell} \rightarrow \cP_{k,\ell-1} \conv \cP_{k,\ell+1} \{2k\} \rightarrow \cP_{k+1,\ell} \conv \cP_{k-1,\ell} \{1\} \rightarrow 0.
\end{gather*}
In particular, letting $\wt{B}_n$ denote the signed adjacency matrix of the quiver in Figure \ref{fig1}, the pair $(\{\cP_{k,\ell}\}_{k \in [1,n], \ell \in \{0,1\}}, \wt{B}_n)$ is a quantum monoidal seed admitting a mutation in all unfrozen directions.
\end{Theorem}

Next, we must establish the third property listed in Theorem \ref{thm:chiralintro}. We do this in Section~\ref{sec:K0asclusteralgebra} as part of the following result.

\begin{Theorem}[c.f. Proposition \ref{prop:K0containment}, Lemma \ref{lem:XtoP}]\label{thm:introiso}
There is a $\Z[q^{\pm 1/2}]$-algebra isomorphism between $K^{GL_n(\cO) \rtimes \Gm}(\Gr_{GL_n})$ and $\cA^{\,loc}_{(L_n,\wt{B}_n)}$. The initial quantum cluster variables are the images of the classes $\{[\cP_{k,\ell}]\}_{k \in [1,n], \ell \in \{0,1\}}$. Moreover, $[\cP_{k,\ell}]$ is a quantum cluster variable for all $k \in [1,n], \ell \in \Z$ (up to multiplication by powers of $q$ and frozen variables). 
\end{Theorem}

Here we use the aforementioned construction in \cite{GLS13} of quantum 
cluster structures on quantum unipotent coordinate rings. Since $\cA_{(L_n,\wt{B}_n)}$ also appears in this context, we obtain a finite set of generators for $\cA^{loc}_{(L_n,\wt{B}_n)}$ from known generating sets of quantum coordinate rings. Using these generators we establish the isomorphism in the Theorem by showing in Section \ref{sec:K0generators} that a certain subset of the $[\cP_{k,\ell}]$ generate $K^{GL_n(\cO) \rtimes \Gm}(\Gr_{GL_n})$ (a~similar result was recently obtained in \cite{FT17}). 

\addtocontents{toc}{\SkipTocEntry}
\subsection*{Categorification of the quantum twist}

Many cluster varieties appearing in Lie theory have distinguished automorphisms called twists \cite{BZ97,FZ99,GY16,MaSc16,MuSp16}. These were introduced for applications in total positivity \cite{BFZ96}, but have since been understood as significant from various categorical viewpoints \cite{GLS12,KO17}. In particular, after erasing frozen variables they become special cases \cite{GLS12,Wen16a,Wen16b} of the operators introduced in \cite{KS08} to encode the wall-crossing of generalized Donaldson-Thomas invariants -- these operators have been referred to as spectrum generators \cite{GMN13a,GMN13b}, (half) monodromy operators \cite{ACCERV14}, DT transformations \cite{GS16}, or total DT invariants \cite{Kel11}.

In our case we obtain the following interpretation in terms of convolution. 

\begin{Theorem}[c.f. Theorem \ref{thm:twist}, Corollary \ref{cor:twist}]\label{thm:introtwist}
The isomorphism 
$$\Q(q^{1/2}) \otimes_{\Z[q^{\pm 1/2}]} K^{GL_n(\cO) \rtimes \Gm}(\Gr_{GL_n}) \cong \Q(q^{1/2}) \otimes_{\Z[q^{\pm 1}]} A_q(N^{(s_0s_1)^n})$$
identifies the anti-automorphism $[\cF] \mapsto [\cF^L]$ sending a sheaf to its left dual with the composition $X \mapsto \overline{\eta_{(s_0s_1)^n}(X)}$ of the quantum twist and the bar involution of $\cA^{loc}_{(L_n,\wt{B}_n)}$. Specializing frozen variables (hence necessarily $q$) to 1 identifies $[\cF] \mapsto [\cF^L]$ with the total DT invariant of $\cA_{B_n}$.
\end{Theorem}

Here $B_n$ is the matrix obtained by erasing the frozen rows of $\wt{B}_n$.
The bar involution of a quantum cluster algebra is an anti-automorphism which inverts $q$ and fixes the quantum cluster variables \cite{BZ05}. In Section \ref{sec:barinvolution} we lift the bar involution of $\cA^{loc}_{(L_n,\wt{B}_n)}$ to an autoequivalence of $D_{coh}^{GL_n(\cO) \rtimes \Gm}(\Gr_{GL_n})$, the result being similar in flavor to its analogues in the setting of the Steinberg variety and nilpotent cone \cite{Lus98,Ost00}. 

Theorem \ref{thm:introtwist} implies the squared twist is integrable: the double dual $[\cF] \mapsto [\cF^{LL}]$ preserves the symmetric monoidal subcategory $\Rep (GL_n(\cO) \rtimes \Gm)$, embedded as sheaves supported on the identity. In fact, the double dual is very close to being a power of the autoequivalence given by tensoring with $\cO(1)$, which itself manifestly leaves invariant objects supported on the identity. This provides a categorical interpretation of the discrete integrability of the $A_{n-1}^{(1)}$ $Q$-system \cite{Ked} (see also \cite{Wil16,CDZ14}), which we now recall.

\addtocontents{toc}{\SkipTocEntry}
\subsection*{Relation to $Q$-systems}
$Q$-systems are recursion relations that describe the characters of Kirillov-Reshetikhin (KR) modules, special Yangian modules labeled by multiples of fundamental weights \cite{KR87,HKOTY99}. In \cite{Ked} it was shown that the relations of the $Q$-system can be normalized to coincide with exchange relations in the cluster algebra $\cA_{B_n}$ (or another suitable cluster algebra in other affine types \cite{DFK,Wil15}).

This connection emerges as follows in our setup. Recall that $\Gr_{GL_n}^{\omega_k^\vee}$ is isomorphic to the finite Grassmannian $\Gr_{k,n}$ of subspaces in $\C^n$ of codimension $k$, and that $\cP_{k,\ell}$ is an equivariant shift of $\O_{\Gr_{k,n}}(\ell)[\frac12 \dim \Gr_{k,n}]$. Thus, for $\ell \geq 0$, the global sections functor recovers (up to the cohomological shift) the irreducible $GL_n$-representation of highest weight $\ell \omega_k$. These are exactly the classical limits of the KR modules in type $A_{n-1}^{(1)}$.

On the other hand, the mutations described in Section \ref{sec:mutrels} can be used for any $\ell \in \Z$ to produce the sheaf $\cP_{k,\ell}$ from the initial cluster $\{\cP_{k,\ell}\}_{k \in [1,n], \ell \in \{0,1\}}$. Given a suitable compatibility between convolution of sheaves and tensor products of global sections (see Proposition \ref{prop:cohomology}), it follows that the associated exchange relations describe characters of KR modules normalized by shifts dictated by the perverse $t$-structure. These exchange relations are exactly those that appear in \cite{Ked}. Since we work with sheaves in formal half-integer degrees, the degree normalizations will in general be 4th roots of unity, as found in \cite{DFK}. 

We also note that through the coherent Satake category we find a conceptual explanation for the appearance of identical quivers in the context of $Q$-systems and separately in the context of BPS quivers \cite{ACCERV14,CDZ12}. We turn to the latter topic next.

\addtocontents{toc}{\SkipTocEntry}
\subsection*{Perspective from gauge theory}
Coulomb branches of gauge theories with eight supercharges in 3 and 4 dimensions have recently been the subject of intense study in both physics and mathematics \cite{CHZ14,Nak16,BFN,BDG17,BDGH16,FT17}. In this language, the spectrum of $K^{G(\cO)}(\Gr_{G})$ is the Coulomb branch of pure 4d $\cN=2$ gauge theory on $\R^3 \times S^1$ with compact gauge group $G_c$. 

Of relevance to us is the monoidal category of line operators in the holomorphic-topological twist of \cite{Kap06} on $\C \times \R^2$, with monoidal structure given by collision of lines in $\R^2$ \cite{KS09} -- expectation values of such operators wrapped on $S^1$ give functions on the Coulomb branch. It was proposed by Costello that this category should be $D_{coh}^{G(\O)}(\Gr_G)$ -- this is the higher geometric quantization of the shifted cotangent bundle of $G(\O) \backslash \Gr_G$, which appears as a derived space of solutions to the equations of motion of the twisted theory \cite{Cos14,EY15}.

It is expected that the Grothendieck ring of the category of line operators has a basis consisting of Wilson-'t Hooft loop operators \cite{KS09}. Mathematically, the most natural origin of a basis in a Grothendieck ring would be a finite-length $t$-structure. It was suggested in \cite{KS09} that Wilson-'t Hooft line operators in 4d $\cN=4$ gauge theory correspond to suitable analogues of perverse coherent sheaves on the affine Grassmannian Steinberg variety.

Theorem~\ref{thm:mainthmintro} implicitly claims that Wilson-'t Hooft line operators in pure 4d $\cN=2$ gauge theory are simple perverse coherent sheaves on the affine Grassmannian itself. Indeed, irreducible line operators in 4d $\cN=2$ theories satisfying certain finiteness conditions are partially described by the cluster algebras of their BPS quivers \cite{GMN13,CN14,CiDZ17}. The quiver $\wt{Q}_n$ from Figure~\ref{fig1} is exactly the BPS quiver of pure $SU(n)$ gauge theory \cite{ACCERV14} (up to frozen vertices present because we focus on $GL_n$ not $SL_n$). Thus Theorem~\ref{thm:mainthmintro} asserts that simple perverse coherent sheaves have the cluster-theoretic properties expected of Wilson-'t Hooft operators.

Theorem~\ref{thm:chiralintro}, on the other hand, provides an illuminating conceptual perspective on why cluster algebras should appear in the context of line operators of $\mathcal{N} = 2$ theories at all, complementary to any considerations involving (framed) BPS states or specific properties of class $S$ theories. Indeed, the language of chiral categories succinctly captures the holomorphic dependence of categories of line operators in 4d holomorphic-topological field theories. In this sense Theorem~\ref{thm:chiralintro} says that the phenomenon of iterated mutation of clusters is simply a structural consequence of the basic formal properties of holomorphic-topological field theory. Furthermore, we expect the least satisfying aspect of the Theorem (the need to know in advance a cluster structure on the algebra of loop operators) can ultimately be removed, see Conjecture~\ref{con:potentials}.

We note that relations between $r$-matrices and holomorphic-topological field theory are also explored in \cite{Cos13,Cos14b,CWY17,CWY18}, which establish a direct connection between affine quantum groups and deformed $\cN=1$ gauge theory. Of course, it is quantum loop algebras which provided the examples of monoidal cluster categorifications originally considered in \cite{HL10}.

\addtocontents{toc}{\SkipTocEntry}
\subsection*{Beyond $GL_n$} We expect the following extension of Theorem \ref{thm:mainthmintro} to other types.

\begin{Conjecture}\label{conj:othertypes}
Let $C$ be a finite-type Cartan matrix, $G$ the associated simply-connected simple algebraic group, and $G_{Ad}$ its adjoint form. Then $\cP_{coh}^{G(\cO) \rtimes \Gm}(\Gr_{G_{Ad}})$ is a monoidal categorification of the quantum cluster algebra $\cA^{\,loc}_{(L_C,\wt{B}_C)}$, where
\[
\wt{B}_C = \begin{pmatrix} C^T - C & -C^T \\ C & 0 \end{pmatrix}
\]
and $L_C$ is a suitable coefficient matrix. The elements of the initial monoidal cluster are the sheaves $\cP_{\omega_i^\vee\!,\,0}$ and $\cP_{\omega_i^\vee\!, \,\omega_i}$ for $\omega_i^\vee$ a fundamental coweight.
\end{Conjecture}

Here $\cP_{\l^\vee\!,\,\mu}$ denotes the simple perverse coherent sheaf supported on $\oGr_{G_{Ad}}^{\l^\vee}$ and associated to a weight $\mu$ (see Section \ref{sec:cohsatdefs}). On one hand, the matrix $\wt{B}_C$ is the adjacency matrix of the BPS quiver of 4d $\cN=2$ gauge theory with compact gauge group $G_c$ \cite{ACCERV14,CDZ12}, so the conjecture is consistent with a correspondence between simple perverse coherent sheaves and simple line operators. 

On the other, the normalized $Q$-system associated to $G$ is described by relations among specific cluster variables in $\cA_{\wt{B}_C}$ attached to the weights $\ell \omega_i$ \cite{DFK}. Under the conjectural isomorphism $K^{G(\cO)}(\Gr_{G_{Ad}}) \cong \cA_{\wt{B}_C}$ these cluster variables should coincide with the classes of the simple objects $\cP_{\omega_i^\vee\!,\, \ell \omega_i}$. Indeed, when $\cP_{\omega_i^\vee\!,\, \ell \omega_i}$ is the restriction of a line bundle on $\Gr_{G_{Ad}}$ (in the appropriate cohomological degree) its global sections are known to be isomorphic as a $G(\cO)$-module to the classical limit of the KR module of weight $\ell \omega_i$ \cite{FL07}. This is the case when $\ell$ satisfies a simple divisibility condition (which is trivial if $G$ is simply laced). We also anticipate an analogue of Conjecture \ref{conj:othertypes} for twisted types involving the skew-symmetrizable matrices of \cite{Wil15}.

%At present, we do not know a tractable generating set for the quantum cluster algebra $\cA^{\,loc}_{(L_C,\wt{B}_C)}$ in general. Recall that in type $A_n$ such a generating set is provided by the relation with quantum unipotent coordinate rings, but this does not generalize to other types. In the absence of such generators, we do not know how to establish the isomorphism $K^{G(\cO) \rtimes \Gm}(\Gr_{G_{Ad}}) \cong \cA^{\,loc}_{(L_C,\wt{B}_C)}$ that would be required to apply Theorem \ref{thm:chiralintro}. However, we expect this algebraic obstacle can be removed (and Theorem \ref{thm:chiralintro} greatly extended) by the development of a theory of potentials for monoidal seeds. 

\addtocontents{toc}{\SkipTocEntry}
\subsection*{Potentials and monoidal categorification}

Recall that a potential on a quiver is a formal sum of oriented cycles \cite{Gin06,DWZ08}. Mutation of quivers with potential is similar to ordinary quiver mutation, except that arbitrary sequences of mutations are not always allowed. Instead, a possible obstruction arises when one mutates at a vertex of an oriented 3-cycle. % -- an oriented 2-cycle is created, and whether or not it can be erased to allow further mutations depends on the potential.

Precisely this situation arises in the proof of Theorem \ref{thm:chiralintro}: one can show that a monoidal seed admits mutation sequences that avoid vertices of oriented 3-cycles without assuming in advance the existence of a cluster structure on $K_0(\grl{\cC})$ (and without needing to work in a graded setting).

\begin{Conjecture}\label{con:potentials}
Let $\cC$ be a rigid monoidal abelian category with a system of renormalized $r$-matrices, and let $(\{\cF_i\}_{i \in I}, \wt{B})$ be a monoidal seed where $\wt{B}$ is the signed adjacency matrix of a quiver $Q$. To this data one can associate a potential $W$ on $Q$ such that a mutation sequence can be performed on $(\{\cF_i\}_{i \in I}, \wt{B})$ in $\cC$ if and only if it can be performed on $(Q,W)$. This potential should encode monoidal factorization properties of objects constructed from the~$\cF_i$.
\end{Conjecture}

From the point of view of the Conjecture, the hypothesis $K_0(\grl{\cC}) \cong \cA_{(L,\wt{B})}$ in Theorem~\ref{thm:chiralintro} should be regarded as a stand-in for knowing the potential associated to $(\{\cF_i\}_{i \in I}, \wt{B})$ is nondegenerate (i.e. allows arbitrary mutation sequences).

Conjecture \ref{con:potentials} is also consistent with considerations from geometric engineering. Line operators in the 4d $\mathcal{N}=2$ theory arising by compactifying type $IIA$/$IIB$ string theory on a noncompact Calabi-Yau 3-fold $X$ may be produced from noncompact $B$/$A$-branes in $X$. When the relevant category of branes is described by a quiver with potential $(Q,W)$ the mutation of suitable collections of branes in $X$ will be controlled by $W$, hence so should mutation of monoidal seeds in the resulting monoidal category of line operators.

\addtocontents{toc}{\SkipTocEntry}
\subsection*{Acknowledgements}

We would like to thank Roman Bezrukavnikov, Tudor Dimofte, Joel Kamnitzer, Andy Neitzke, Simon Riche, Pavel Safronov and Milen Yakimov for useful discussions. We especially thank Michael Finkelberg for his comments and interest, Ivan Mirkovic for sharing his unpublished notes \cite{Mir}, and Amnon Neeman for illuminating some confusing issues that appeared in the proof of Theorem \ref{thm:adjoints}. Finally, we thank Kevin Costello for helping inspire the present work, which originated as an attempt to understand the relation between his interpretation of line operators as coherent sheaves on $\Gr_G$ and the appearance of cluster algebras in \cite{GMN13}. S.C. was supported by an NSERC Discovery/accelerator grant and H.W. by NSF Postdoctoral Fellowship DMS-1502845 and NSF grant DMS-1702489.  

\section{The coherent Satake category}\label{sec:coherentSatake}

We begin by discussing the category of perverse coherent sheaves on the affine Grassmannian following \cite{AB10, BFM}. We then describe some structural results in the case $G=GL_n$.

\subsection{Definitions and basic properties}\label{sec:cohsatdefs}

Let $G$ be a reductive complex algebraic group and
$$\Gr_G := G(\cK)/G(\cO)$$
its affine Grassmannian, where $\cK := \C((t))$ and $\cO := \C[[t]]$.
We fix Cartan and Borel subgroups $H \subset B \subset G$, and let $P^\vee_+$ denote the associated set of dominant elements of the coweight lattice $P$ of $H$. The assignment $\l^\vee \mapsto G(\cO) \cdot [t^{\l^\vee}]$ yields a bijection between $P^\vee_+$ and the set of left $G(\cO)$-orbits in $\Gr_G$. 

The (spherical) Schubert variety $\oGr_G^{\l^\vee}$ is the closure of this orbit endowed with its reduced scheme structure. %These varieties are normal, Cohen-Macaulay, and have rational singularities.
We have $\oGr_G^{\mu^\vee} \subset \oGr_G^{\l^\vee}$ exactly when $\mu^\vee \le \l^\vee$. The complement of $\cup_{\mu^\vee < \l^\vee} \oGr_G^{\mu^\vee}$ inside $\oGr_G^{\l^\vee}$ is the Schubert cell $\Gr_G^{\l^\vee}$, which set-theoretically is the $G(\cO)$-orbit of $t^{\l^\vee}$ and is open inside $\oGr_G^{\l^\vee}$. In general $\Gr_G$ is not reduced, and as an ind-scheme its reduced locus $(\Gr_G)_{red}$ is the colimit $\varinjlim \oGr_G^{\le \l^\vee}$.

The left $G(\cO)$-action extends to an action of ${G(\O) \rtimes \Gm}$, where by convention our $\Gm$ will denote the 4-fold cover of the standard loop rotation (see Remark \ref{rem:Gmandq}).
We will write $D_{coh}^{G(\O) \rtimes \Gm}(\Gr_G)$ for the bounded derived category of ${G(\O) \rtimes \Gm}$-equivariant coherent sheaves on $(\Gr_G)_{red}$. The reader may refer to \cite{GR14} for generalities on coherent sheaves on ind-schemes. We note in particular that any $\cF \in D_{coh}^{G(\O) \rtimes \Gm}(\Gr_G)$ is a pushforward from $D_{coh}^{G(\O) \rtimes \Gm}(\oGr_G^{{\l^\vee}})$ for a sufficiently large $\oGr_G^{{\l^\vee}}$. %Thus it has finite-dimensional support, and essentially all computations we will need can be performed on these finite-dimensional varieties. 

Except in Section \ref{sec:chiral}, we will not really need to distinguish between the classical derived category and its DG/$\infty$-enhancement -- the reader may interpret $D_{coh}$ at whichever level they prefer. In the same spirit we write $D^{G(\O) \rtimes \Gm}_{qcoh}(\oGr_G^{\l^\vee})$ for the (possibly enhanced) unbounded derived category of ${G(\O) \rtimes \Gm}$-equivariant quasicoherent sheaves on $\oGr_G^{\l^\vee}$. 

%To discuss perverse coherent sheaves it will, however, be convenient to formally modify the notion of cohomological degree as follows.
%Recall that the connected components of $\Gr_G$ are in bijection with $\pi_1(G)$, and correspondingly $D_{coh}^{G(\O) \rtimes \Gm}(\Gr_G)$ has a block decomposition indexed by components.
Note that $D_{coh}^{G(\O) \rtimes \Gm}(\Gr_G)$ has a block decomposition indexed by $\pi_0(\Gr_G) \cong \pi_1(G)$. Crucially for what follows, the dimensions of the $G(\cO)$-orbits within a fixed component are of equal parity. With this in mind, we will always consider the formal variant of $D_{coh}^{G(\O) \rtimes \Gm}(\Gr_G)$ where the block attached to a component with odd-dimensional orbits consists of complexes supported in formal degrees $\Z + \frac12$. 

\begin{Definition}[\cite{AB10}] A $G(\cO) \rtimes \Gm$-equivariant \newword{perverse coherent sheaf} on $\Gr_G$ is an object $\cF \in D_{coh}^{G(\O) \rtimes \Gm}(\Gr_G)$ such that for every orbit $i_{\l^\vee}: \Gr_G^{\l^\vee} \hookrightarrow \Gr_G$
\begin{enumerate}
\item $i_{\l^\vee}^*(\cF) \in D^{G(\O) \rtimes \Gm}_{qcoh}(\Gr_G^{\l^\vee})$ is supported in degrees $\leq - \frac12 \dim \Gr_G^{\l^\vee}$, 
\item $i_{\l^\vee}^!(\cF) \in D^{G(\O) \rtimes \Gm}_{qcoh}(\Gr_G^{\l^\vee})$ is supported in degrees $\geq - \frac12 \dim \Gr_G^{\l^\vee}$.
\end{enumerate}
We write $\cP_{coh}^{G(\cO) \rtimes \Gm}(\Gr_G) \subset D_{coh}^{G(\O) \rtimes \Gm}(\Gr_G)$ for the full subcategory of perverse coherent sheaves. 
\end{Definition}

Inherent in the definition is the choice of middle perversity function. By virtue of the above parity condition on orbits, this notion of perverse coherent sheaves enjoys comparable formal properties to its constructible counterpart.% (the usual notion of perverse sheaves with respect to the middle perversity). 

\begin{Theorem}\cite{AB10}
The category $\cP_{coh}^{G(\cO) \rtimes \Gm}(\Gr_G)$ is the heart of a finite-length $t$-structure on $D_{coh}^{G(\O) \rtimes \Gm}(\Gr_G)$ which is preserved by the duality functor~$\D$. 
The restriction to $\Gr_G^{{\l^\vee}}$ of a simple perverse coherent sheaf whose support is $\oGr_G^{{\l^\vee}}$ is a simple equivariant vector bundle concentrated in cohomological degree $- \frac12 \dim \Gr_G^{\l^\vee}$. This determines a bijection between isomorphism classes of simple objects in $\cP_{coh}^{G(\cO) \rtimes \Gm}(\Gr_G)$ and pairs of an orbit together with a simple equivariant vector bundle on it. 
\end{Theorem}

Recall the convolution diagram
$$ \overline{m}: (G(\cK) \rtimes \Gm) \times_{G(\cO) \rtimes \Gm} \Gr_G \onto \Gr_G $$
induced from multiplication in $G(\cK) \rtimes \Gm$. Given $\cF, \cG \in D_{coh}^{G(\O) \rtimes \Gm}(\Gr_G)$, we write $\cF\, \tbox\, \cG \in D_{coh}^{G(\O) \rtimes \Gm}((G(\cK) \rtimes \Gm) \times_{G(\cO) \rtimes \Gm} \Gr_G)$ for their twisted product. This sheaf is characterized by having pullback to $(G(\cK) \rtimes \Gm) \times \Gr_G$ isomorphic to the pullback of  $\cF \boxtimes \cG \in D_{coh}^{G(\O) \rtimes \Gm}(\Gr_G \times \Gr_G)$ along projection in the first factor. If $\cF$ and $\cG$ are pushforwards from $\oGr_G^{{\l_1^\vee}}$ and $\oGr_G^{{\l_2^\vee}}$, respectively, then $\cF\, \tbox\, \cG$ is a pushforward from the finite-dimensional convolution variety
$$ \oGr_G^{\l_1^\vee} \ttimes \oGr_G^{\l_2^\vee} \subset (G(\cK) \rtimes \Gm) \times_{G(\cO) \rtimes \Gm} \Gr_G.$$ 
The convolution map determines a monoidal structure on $D_{coh}^{G(\O) \rtimes \Gm}(\Gr_G)$ via
$$\cF \conv \cG := \overline{m}_*(\cF\, \tbox\, \cG).$$
%As in the constructible case, we obtain an induced monoidal structure on $\cP_{coh}^{G(\cO) \rtimes \Gm}(\Gr_G)$ by the following result.

\begin{Theorem}\cite{BFM}
If $\cF$, $\cG \in D_{coh}^{G(\O) \rtimes \Gm}(\Gr_G)$ are perverse then so is $\cF \conv \cG$. 
\end{Theorem}

If $X$ is a variety equipped with an action of $\Gm$ we denote by $\O_X \{k\}$ the structure sheaf with the $\Gm$ action of weight $2k$. In other words, given a local function $f \in \O_{X}(U)$, the corresponding section $f' \in \O_X\{k\}(U)$ transforms as $a \cdot f' = a^{-2k}(a \cdot f)$, where $a \cdot f$ denotes the pullback of $f$ along the action of $a$ on $U$. For example, under the standard $\Gm$-action on $X=\Spec\, \C[t]$  multiplication by $t$ induces a morphism $\O_X \rightarrow \O_X \{\frac12\}$. %Perhaps the easiest way to remember this convention is that for the standard $\Gm$-action on $X=\Spec\, \C[t]$  multiplication by $t$ induces a morphism $\O_X \rightarrow \O_X \{\frac12\}$. 

Tensoring with $\O_{X} \{k\}$ gives us an autoequivalence $\cF \mapsto \cF \{k\}$ of $D_{coh}^{\Gm}(X)$. At the level of K-theory, the equivariant Grothendieck ring $K^{\Gm}(X)$ is a module over $K^{\Gm}(pt) = \Z[q^{\pm1/2}]$. Here $q$ acts as shift by $\{-1\}$ (tensoring with $\O_X \{-1\}$). 

We write $K^{G(\O) \rtimes \Gm}(\Gr_G)$ for the Grothendieck ring of $D_{coh}^{G(\O) \rtimes \Gm}(\Gr_G)$ (and hence of $\cP_{coh}^{G(\cO) \rtimes \Gm}(\Gr_G)$). In this case we can also think of multiplication by $q$ as convolution with $[\cO_e\{-1\}]$ where $e$ denotes the point $\Gr_G^0$. It will turn out to be useful to define $\la 1 \ra := [1]\{-1\}$. 

\begin{Remark}\label{rem:Gmandq}
Our conventions are arranged so that the ring $K^{G(\O) \rtimes \Gm}(\Gr_G)$ for which $\Gm$ refers to the double cover of loop rotation is naturally embedded in the $\Z[q^{\pm 1/2}]$-algebra $K^{G(\O) \rtimes \Gm}(\Gr_G)$ defined above as a $\Z[q^{\pm 1}]$-subalgebra. The convention that multiplication by $q$ corresponds to shifting by weight 1 with respect to the double cover of loop rotation agrees with a similar choice made in \cite{CK1} (and related subsequent papers) as well as in \cite{BFN,FT17}. On the other hand, we are motivated to consider the action of the 4-fold cover (hence extend scalars to $\Z[q^{\pm1/2}]$) by the fact that the equivariant structure on the bar-invariant simple objects $\cP_{\l^\vee\!,\,\mu}$ defined below does not factor through the double cover. In particular, the naive $\Z[q^{\pm 1}]$-form of $K^{G(\O) \rtimes \Gm}(\Gr_G)$ will turn out not to contain all quantum cluster variables. They will, however, be contained in the $\Z[q^{\pm 1}]$-form obtained by multiplying all classes supported on odd-dimensional Schubert varieties by $q^{1/2}$.
\end{Remark}

Since $\cP_{coh}^{G(\cO) \rtimes \Gm}(\Gr_G)$ has finite length the classes of simple objects provide a basis of $K^{G(\O) \rtimes \Gm}(\Gr_G)$. To a pair $({\l^\vee}, \mu) \in P^\vee \times P$ we can attach a simple perverse coherent sheaf~$\cP_{\l^\vee\!,\,\mu}$ as follows. The orbit $\Gr_G^{{\l^\vee}}$ is an affine space bundle over its $\Gm$-fixed locus, which is isomorphic to $G/P_{w{\l^\vee}}$ for $P_{w{\l^\vee}}$ the parabolic associated to the unique dominant conjugate ${w{\l^\vee}}$ of ${\l^\vee}$. Acting by the stabilizer of ${\l^\vee}$ in $W$ we can conjugate $\mu$ to become dominant for the Levi factor of $P_{w\l^\vee}$ in a unique way. This defines a simple $G$-equivariant vector bundle on $G/P_{w\l^\vee}$. We extend this to a $G(\cO) \rtimes \G_m$-equivariant bundle by letting $t$ act by zero and letting $\G_m$ act on fibers by the character
$$a \mapsto a^{2\la \l^\vee\!,\,\mu \ra + \dim \Gr_G^{\l^\vee}}.$$ The pullback of this bundle to $\Gr_G^{\l^\vee}$, shifted to lie in cohomological degree $-\frac12 \dim \Gr_G^{\l^\vee}$, is the restriction of a unique simple perverse coherent sheaf $\cP_{\l^\vee\!,\,\mu}$ whose support is $\oGr_G^{\l^\vee}$.

\begin{Remark}
The choice of $\G_m$-equivariant structure of $\cP_{\l^\vee\!,\,\mu}$ will be explained by the discussion of the bar involution in Section~\ref{sec:barinvolution}. It also results in the mutation exact sequences of Section~\ref{sec:mutrels} aligning correctly with the conventions of quantum cluster theory.
\end{Remark}

We refer to pairs $({\l^\vee}, \mu) \in P^\vee \times P$ for which ${\l^\vee}$ is dominant and $\mu$ dominant for the Levi factor of $P_{\lambda^\vee}$ as \newword{dominant pairs}. From the above construction it follows that simple objects (up to $\Gm$-equivariant shifts) are labeled by $(P^\vee \times P)/W$ (equivalently, by dominant pairs).

\begin{Proposition}\label{prop:leadingterms}
Given dominant pairs $(\l_1^\vee, \mu_1)$ and $(\l_2^\vee, \mu_2)$ the product $\cP_{\l_1^\vee\!,\,\mu_1} \conv \cP_{\l_2^\vee\!,\,\mu_2}$ is supported on $\oGr^{\l_1^\vee+\l_2^\vee}$. Moreover, in $K^{G(\cO) \rtimes \Gm}(\Gr_G)$ we have 
$$[\cP_{\l_1^\vee\!,\, \mu_1} \conv \cP_{\l_2^\vee\!,\,\mu_2}] = q^{\la \l_1^\vee\!,\,\mu_2 \ra - \la \l_2^\vee\!,\,\mu_1 \ra} [\cP_{\l_1^\vee+\l_2^\vee\!,\,\mu_1+\mu_2}] + \sum_{(\l^\vee\!,\,\mu) \in S} p_{\l^\vee\!,\,\mu} [\cP_{\l^\vee\!,\,\mu}]$$
where $p_{\l^\vee\!,\,\mu} \in \Z[q^{\pm 1/2}]$ and $S$ is a finite collection of dominant pairs $(\l^\vee,\mu)$ with either $\l^\vee < \l_1^\vee+\l_2^\vee$, or $\l^\vee = \l_1^\vee+\l_2^\vee$ and $\lVert \mu \rVert^2 \leq \lVert \mu_1 \rVert^2 + \lVert \mu_2 \rVert^2$ for any $W$-invariant quadratic form $\lVert \cdot \rVert^2$. 
\end{Proposition}
\begin{proof}
The map $\om: \oGr_G^{\l_1^\vee} \ttimes \oGr_G^{\l_2^\vee} \to \oGr_G^{\l_1^\vee+\l_2^\vee}$ restricts to an isomorphism over $\Gr_G^{\l_1^\vee+\l_2^\vee}$. In particular, the restriction of $\cP_{\l_1^\vee\!,\,\mu_1} \conv \cP_{\l_2^\vee\!,\,\mu_2}$ to $\Gr_G^{\l_1^\vee+\l_2^\vee}$ is a vector bundle whose fiber over the $P_{\l_1^\vee+\l_2^\vee}$-fixed point $[t^{\l_1^\vee+\l_2^\vee}]$ is $V_{\mu_1} \otimes V_{\mu_2}$ placed in cohomological degree $-\frac12 \dim \Gr_G^{\l_1^\vee+\l_2^\vee}$. Here we write $V_{\mu_i}$ for the irreducible $P_{\l_i^\vee}$-representation of highest weight $\mu_i$. This bundle need not be equivariantly simple as $V_{\mu_1} \otimes V_{\mu_2}$ need not be simple as a $P_{\l_1^\vee + \l_2^\vee}$ representation, but any weight $\mu$ of $V_{\mu_1} \otimes V_{\mu_2}$ lies in the convex hull of the sums of the $W$-orbits of $\mu_1$ and $\mu_2$. Thus it satisfies $\lVert \mu \rVert^2 \leq \lVert \mu_1 \rVert^2 + \lVert \mu_2 \rVert^2$ for any $W$-invariant form.

The coefficient of $[\cP_{\l_1^\vee+\l_2^\vee\!,\, \mu_1+\mu_2}]$ is determined by our normalizations of the $\cP_{\l^\vee\!,\,\mu}$ and the fact that $(1,a) \in G(\cO) \rtimes \Gm$ acts on $(G(\cK) \rtimes \Gm) \times_{G(\cO) \rtimes \Gm} \Gr_G$ by
$$ (1,a)[t^{\l_1^\vee}, 1, [t^{\l_2^\vee}]] = [t^{\l_1^\vee}a^{\l_1^\vee}, a, [t^{\l_2^\vee}] = [t^{\l_1^\vee}, 1, (a^{4\l_1^\vee},a)[t^{\l_2^\vee}]].$$
This contributes a factor of $q^{2\la \l_1^\vee\!,\,\mu_2 \ra}$ to the final expression, which is partially canceled by the normalization of $\cP_{\l_1^\vee+\l_2^\vee\!,\,\mu_1+\mu_2}$ to yield the leading coefficient $q^{\la \l_1^\vee\!,\,\mu_2 \ra - \la \l_2^\vee\!,\,\mu_1 \ra}$.
\end{proof}

Recall that a simple object in a monoidal abelian category is \newword{real} if its product with itself is again simple \cite{HL10}.

\begin{Lemma}\cite{Mir}\label{lem:real}
Let $\cL$ be a $G(\cO) \rtimes \G_m$-equivariant line bundle on $\Gr_G$. Then the restriction of $\cL\,[\frac 12 \dim \Gr^{\l^\vee}_G]$ to $\oGr^{\l^\vee}$ is a real simple perverse coherent sheaf for any $\l^\vee \in P^\vee$.
\end{Lemma}
\begin{proof}
Perversity and simplicity follow from the trivial case of $\cL = \cO_{\Gr_G}$ (c.f. \cite[Section~4.1]{AB10}) since tensoring with $\cL$ is an autoequivalence which preserves perversity. Letting $\cP$ denote the indicated restriction, we further have
\begin{align*}
\cP \conv\cP 
&\cong \om_* \left( \O_{\oGr_G^{\l^\vee} \ttimes \oGr_G^{\l^\vee}} \otimes \om^*\cL \right) [\dim \Gr^{\l^\vee}_G] \\
&\cong \O_{\oGr^{2 \l^\vee}_G} \otimes \cL\, [\frac 12 \dim \Gr^{2 \l^\vee}_G], 
\end{align*}
where the second isomorphism follows from the projection formula and the fact that (spherical) Schubert varieties have rational singularities. Reality of $\cP$ then follows since as we have just seen the restriction of $\cL\,[\frac 12 \dim \Gr^{2 \l^\vee}_G]$ to ${\oGr^{2 \l^\vee}_G}$ is simple.
\end{proof}

\subsection{Geometry of $\Gr_{GL_n}$}\label{sec:geometryGL}

In the remainder of Section \ref{sec:coherentSatake} we study $\cP_{coh}^{G(\cO) \rtimes \Gm}(\Gr_G)$ in greater detail for $G = GL_n$. In this case the geometry of $\Gr_G$ can be described in terms of $\C[[t]]$-lattices inside $\C((t)) \otimes_\C V$, where $V = \C^n$. Letting $L_0$ denote the standard lattice $\C[[t]] \otimes_\C V$ we have
$$\Gr^{\l^\vee}_{GL_n} =\{L \subset L_0: t|_{L_0/L} \text{ is nilpotent of type } {\l^\vee} \}.$$ 
The closure $\oGr^{\l^\vee}_{GL_n}$ is similar but with $t|_{L_0/L}$ nilpotent of type $\le {\l^\vee}$. In particular, we have
$$\Gr^{\omega^\vee_k}_{GL_n} = \{L \overset{k}\subset L_0: tL_0 \subset L\}$$
where $L \overset{k}\subset L_0$ indicates that $\dim(L_0/L)=k$. 
This space is isomorphic to a (finite) Grassmannian of quotient spaces of $V = L_0/tL_0$ and we denote it by $\Gr^k_{GL_n}$ for short. 

The convolution spaces of such varieties can similarly be described as 
$$\Gr_{GL_n}^{k_1} \ttimes \dots \ttimes \Gr_{GL_n}^{k_m} = \{L_m \overset{k_m}\subset \dots \overset{k_2}\subset L_1 \overset{k_1}\subset L_0: tL_{i-1} \subset L_i \}.$$
This is an iterated bundle of (finite) Grassmannians which, for convenience, we denote $\Gr^{(k_1,\dots,k_m)}_{GL_n}$. The multiplication map $\om$ to $\Gr_{GL_n}$ is then given by the forgetful map
\begin{align*}
\om: \Gr^{(k_1,\dots,k_m)}_{GL_n} & \rightarrow \Gr_{GL_n} \\
(L_m \subset \dots \subset L_1 \subset L_0) &\mapsto (L_m \subset L_0).
\end{align*}

These spaces carry some obvious tautological vector bundles. For example, one has the vector bundle of rank $k_i$ whose fiber over a point $(L_m \subset \dots \subset L_0)$ is $(L_{i-1}/L_i)$. We abuse notation and denote this bundle as $L_{i-1}/L_i$.  

If $i_k: \Gr_{GL_n}^k \hookrightarrow \Gr_{GL_n}$ is the natural inclusion we will write
\begin{equation}\label{eq:Ps}
\cP_{k,\ell} := i_{k*} \left( \O_{\Gr_{GL_n}^k} \otimes \det(L_0/L)^{\ell} \right) \la \frac12 \dim \Gr_{GL_n}^k \ra \{-k\ell\}
\end{equation}
where $\la 1 \ra := [1]\{-1\}$. In the notation of Section \ref{sec:cohsatdefs} we have $\cP_{k,\ell} = \cP_{\omega^\vee_k\!,\, \ell\omega_k}$. We sometimes simply write $\cP_k$ for $\cP_{k,0}$. Note that $\cP_{0,\ell} \cong \O_e$ for all $\ell \in \Z$, where $e = \Gr_{GL_n}^0$. 

We also have the analogous varieties
$$\Gr^{-k}_{GL_n} := \Gr^{{\omega_k^\vee}^*}_{GL_n} = \{L_0 \overset{k}\subset L: tL \subset L_0 \}.$$ 
Here ${\omega_k^\vee}^* = - w_0\, \omega_k^\vee$, where $w_0$ is the longest element of $W$. We then define
\begin{equation}\label{eq:-Ps}
\cP_{-k,\ell} := i_{-k*} \left( \O_{\Gr^{-k}_{GL_n}} \otimes \det(L/L_0)^{\ell} \right) \la \frac12 \dim \Gr^{-k}_{GL_n} \ra \{-k \ell\},
\end{equation}
where $i_{-k}: \Gr^{-k}_{GL_n} \hookrightarrow \Gr_{GL_n}$ is the inclusion. By Lemma \ref{lem:real} it follows that $\cP_{k,\ell}$ and $\cP_{-k,\ell}$ are real. Moreover, it is not difficult to check that the $\cP_{n,\ell}$ are invertible and that $\cP_{n,\ell}^{-1} \cong \cP_{-n,-\ell}$. 

\subsection{Commutativity relations}\label{sec:commrel}

To prove Theorem \ref{thm:introGL_ncomps} we require the following $q$-commutativity result. It implies in particular that the collection $\{\cP_{k,\ell}\}_{k \in [1,n], \ell \in \{0,1\}}$ is a quantum monoidal cluster in $\cP_{coh}^{GL_n(\O) \rtimes \Gm}(\Gr_{GL_n})$. 

\begin{Proposition}\label{prop:comm}
We have 
$$ \cP_{k_1,\ell_1} \conv \cP_{k_2,\ell_2} \cong \cP_{k_2,\ell_2} \conv \cP_{k_1,\ell_1} \{ 2(\ell_1-\ell_2){\rm{min}}\{k_1,k_2\} \} $$
whenever $|\ell_1 - \ell_2| \le 1$. 
\end{Proposition}
\begin{proof}
Consider the convolution diagram 
$$\{L_2 \overset{k_2}\subset L_1 \overset{k_1}\subset L_0: tL_0 \subset L_1, tL_1 \subset L_2\} = \Gr_{GL_n}^{(k_1,k_2)} \xrightarrow{\om} \oGr^{\omega^\vee_{k_1}+\omega^\vee_{k_2}}_{GL_n}.$$
Then 
$$\cP_{k_1,\ell_1} * \cP_{k_2,\ell_2} \{k_1\ell_1+k_2\ell_2\} \cong \om_*(\O_{\Gr^{(k_1,k_2)}_{GL_n}} \otimes \det(L_0/L_1)^{\ell_1} \otimes \det(L_1/L_2)^{\ell_2}) \la \frac12 \dim \Gr^{(k_1,k_2)}_{GL_n} \ra.$$
Now $\om^*(\det(L_0/L_2)) \cong \det(L_0/L_1) \otimes \det(L_1/L_2)$. So, if $\ell_1=\ell_2=\ell$, then 
\begin{align*}
\cP_{k_1,\ell} * \cP_{k_2,\ell} \{(k_1+k_2)\ell\}
&\cong \om_*(\O_{\Gr^{(k_1,k_2)}_{GL_n}} \otimes \om^*(\det(L_0/L_2)^{\ell})) \la \frac12 \dim \Gr^{(k_1,k_2)}_{GL_n} \ra \\
&\cong \om_*(\O_{\Gr^{(k_1,k_2)}_{GL_n}}) \otimes \det(L_0/L_2)^{\ell}  \la \frac12 \dim \Gr^{(k_1,k_2)}_{GL_n} \ra \\
&\cong \O_{\oGr^{\omega^\vee_{k_1}+\omega^\vee_{k_2}}_{GL_n}} \otimes \det(L_0/L_2)^{\ell}  \la \frac12 \dim \Gr^{(k_1,k_2)}_{GL_n} \ra
\end{align*}
where $\om_*(\O_{\Gr^{(k_1,k_2)}_{GL_n}}) \cong \O_{\oGr^{\omega^\vee_{k_1}+\omega^\vee_{k_2}}_{GL_n}}$ since any $\oGr^{\l^\vee}_{GL_n}$ has rational singularities. By symmetry this equals $\cP_{k_2,\ell} * \cP_{k_1,\ell} \{(k_1+k_2)\ell\}$. 

For the second, more complicated case, it suffices by symmetry to let $\ell := \ell_1$ and assume $\ell_2=\ell+1$. We will assume $k_1 \le k_2$, the case $k_1 \ge k_2$ being similar. Then
\begin{gather*}
\cP_{k_1,\ell} * \cP_{k_2,\ell+1} \{k_1\ell+k_2(\ell+1)\} \cong \om_{1*}(\O_{\Gr^{(k_1,k_2)}_{GL_n}} \otimes \det(L_1/L_2)) \otimes \det(L_0/L_2)^{\ell} \la \frac12 \dim \Gr^{(k_1,k_2)}_{GL_n} \ra\\
\cP_{k_2,\ell+1} * \cP_{k_1,\ell} \{k_1\ell+k_2(\ell+1)\} \cong \om_{2*}(\O_{\Gr^{(k_2,k_1)}_{GL_n}} \otimes \det(L_0/L_1)) \otimes \det(L_0/L_2)^{\ell} \la \frac12 \dim \Gr^{(k_2,k_1)}_{GL_n} \ra
\end{gather*}
where $\om_1$ and $\om_2$ are the multiplication maps with domains $\Gr^{(k_1,k_2)}_{GL_n}$ and $\Gr^{(k_2,k_1)}_{GL_n}$, respectively.
It remains to show that
\begin{equation}\label{eq:needtoshow1}
\om_{1*}(\O_{\Gr^{(k_1,k_2)}_{GL_n}} \otimes \det(L_1/L_2) \{2k_1\}) \cong \om_{2*}(\O_{\Gr^{(k_2,k_1)}_{GL_n}} \otimes \det(L_0/L_1)),
\end{equation}

To check this we consider the space 
$$W := \{ L_2 \overset{k_1}\subset L_1 \overset{k_2-k_1}\subset L_1' \overset{k_1}\subset L_0: tL_0 \subset L_1, tL_1' \subset L_2 \}.$$
It fits into a commutative diagram
\begin{equation*}%\label{eq:29}
\begin{gathered}\xymatrix{
& W \ar[dl]_{\pi_1} \ar[dr]^{\pi_2} \ar[dd]^{\pi} & \\
\Gr^{(k_1,k_2)}_{GL_n} \ar[dr]^{\om_1} & & \Gr^{(k_2,k_1)}_{GL_n} \ar[dl]_{\om_2} \\
& \oGr^{\omega^\vee_{k_1}+\omega^\vee_{k_2}}_{GL_n} &
}\end{gathered}\end{equation*}
where the maps $\pi_1$ and $\pi_2$ are given by forgetting $L_1'$ and $L_1$ respectively. Since 
\begin{align*}
\pi_{1*}(\O_W \otimes \det(L_1/L_2)) &\cong \O_{\Gr_{GL_n}^{(k_1,k_2)}} \otimes \det(L_1/L_2) \ \ \text{ and } \\
\pi_{2*}(\O_W \otimes \det(L_0/L_1')) &\cong \O_{\Gr^{(k_2,k_1)}_{GL_n}} \otimes \det(L_0/L_1)
\end{align*}
it follows that (\ref{eq:needtoshow1}) is equivalent to
\begin{equation}\label{eq:needtoshow2}
\pi_*(\O_W \otimes \det(L_0/L_1')) \cong \pi_*(\O_W \otimes \det(L_1/L_2)) \{2k_1\}. 
\end{equation}

Now consider the map of rank $k_1$ vector bundles on $W$ given by 
$$t: L_0/L_1' \rightarrow L_1/L_2 \{2\}.$$
This map is degenerate (i.e. not an isomorphism) along a divisor $D \subset W$ where 
$$\O_W(D) \cong \O_W \otimes \det(L_0/L_1')^{-1} \otimes \det(L_1/L_2 \{2\}).$$ 
The standard triangle $\O_W(-D) \rightarrow \O_W \rightarrow \O_D$ then gives us the exact triangle
$$\O_W \otimes \det(L_0/L_1') \otimes \det(L_1/L_2)^{-1} \{-2k_1\} \rightarrow \O_W \rightarrow \O_D.$$
After tensoring with $\det(L_1/L_2)\{2k_1\}$ we arrive at 
$$\O_W \otimes \det(L_0/L_1') \rightarrow \O_W \otimes \det(L_1/L_2) \{2k_1\} \rightarrow \O_D \otimes \det(L_1/L_2)\{2k_1\}.$$
Thus, to show (\ref{eq:needtoshow2}), it remains to show that 
\begin{equation}\label{eq:needtoshow3}
\pi_* (\O_D \otimes \det(L_1/L_2)) = 0.
\end{equation}

To show this, we note that $D$ has a partial resolution $\tD \xrightarrow{f} D$, where
$$\tD := \{L_2 \overset{k_1-1}\subset L_1'' \overset{1}\subset L_1 \overset{k_2-k_1}\subset L_1' \overset{k_1}\subset L_0: tL_0 \subset L_1'', tL_1' \subset L_2\}$$
and $f$ forgets $L_1''$. It suffices to show $\pi_* f_* (\O_{\tD} \otimes \det(L_1/L_2))=0$. But we can factor $\pi \circ f$ instead as $\tD \xrightarrow{f'} D' \rightarrow \oGr^{\omega^\vee_{k_1}+\omega^\vee_{k_2}}_{GL_n}$ where 
$$D' := \{L_2 \overset{k_1-1}\subset L_1'' \overset{k_2-k_1+1}\subset L_1' \overset{k_1}\subset L_0: tL_0 \subset L_1'', tL_1' \subset L_2\}$$
and $f'$ forgets $L_1$. Since $tL_0 \subset L_1''$ the morphism $f'$ is a $\P^{k_2-k_1}$-bundle, because $L_1$ can be chosen arbitrarily inside $L_1'/L_1''$. The restriction of $\det(L_1/L_2)$ to the fibers of $f'$ is isomorphic to $\O_{\P^{k_2-k_1}}(-1)$. If $k_1 < k_2$, it follows that 
$$f'_*(\O_{\tD} \otimes \det(L_1/L_2)) = 0.$$
If $k_1=k_2$, then we can repeat the argument by also forgetting $L_1'$ to obtain this vanishing. This proves (\ref{eq:needtoshow3}) and completes the proof.
\end{proof}
\begin{Remark}\label{rem:comm}
A simple version of the argument used in the proof of the second case of Proposition \ref{prop:comm} also shows that $\cP_{k,\ell_1}$ and $\cP_{n,\ell_2}$ $q$-commute for all $\ell_1,\ell_2$ with
$$ \cP_{k,\ell_1} \conv \cP_{n,\ell_2} \cong \cP_{n,\ell_2} \conv \cP_{k,\ell_1} \{2k(\ell_1-\ell_2)\}.$$
\end{Remark}

\begin{Lemma}\label{lem:frozenidentity}
For any $k_1,k_2,k_1',k_2' \in \Z$ with $k_1+k_2=k_1'+k_2'$ we have 
$$\cP_{n,k_1} * \cP_{n,k_2} \{2nk_2\} \cong \cP_{n,k_1'} * \cP_{n,k_2'} \{2nk_2'\}.$$
\end{Lemma}
\begin{proof}
On $\Gr_{GL_n}^{(n,n)}$ we have $t: L_0/L_1 \xrightarrow{\sim} L_1/L_2 \{2\}$. This implies that 
\begin{align*}
\cP_{n,k_1} * \cP_{n,k_2} \{(k_1+k_2)n\} 
&\cong \om_*(\O_{\Gr_{GL_n}^{(n,n)}} \otimes \det(L_0/L_1)^{k_1+k_2} \{-2nk_2\}) \\
&\cong \cP_{n,k_1+k_2} * \cP_{n,0} \{-2nk_2\}.
\end{align*}
Similarly we get $\cP_{n,k'_1} * \cP_{n,k'_2} \{(k_1'+k_2')n\} \cong \cP_{n,k_1+k_2} * \cP_{n,0} \{-2nk_2'\}$. 
\end{proof}
\begin{Remark}\label{rem:frozenidentity}
Lemma \ref{lem:frozenidentity} applied repeatedly implies that for any $r \in \N$ we have 
$$ (\cP_{n,1})^r \cong \cP_{n,r} * (\cP_{n,0})^{r-1} \{-nr(r-1)\}.$$
\end{Remark}

\subsection{Mutation exact sequences}\label{sec:mutrels}

The following family of exact sequences will be used to show that the monoidal cluster $\{\cP_{k,\ell}\}_{k \in [1,n], \ell \in \{0,1\}}$ admits mutations in every unfrozen direction. 

\begin{Proposition}\label{prop:mutation}
For $k \in [0,n], \ell \in \Z$ we have exact sequences
\begin{gather*}
0 \rightarrow \cP_{k-1,\ell} \conv \cP_{k+1,\ell} \{-1\} \rightarrow \cP_{k,\ell+1} \conv \cP_{k,\ell-1} \{-2k\} \rightarrow \cP_{k,\ell} \conv \cP_{k,\ell} \rightarrow 0 \\
0 \rightarrow \cP_{k,\ell} \conv \cP_{k,\ell} \rightarrow \cP_{k,\ell-1} \conv \cP_{k,\ell+1} \{2k\} \rightarrow \cP_{k+1,\ell} \conv \cP_{k-1,\ell} \{1\} \rightarrow 0.
\end{gather*}
in $\cP_{coh}^{GL_n(\cO) \rtimes \Gm}(\Gr_{GL_n})$, where in these formulas $\cP_{-1,\ell}$ and $\cP_{n+1,\ell}$ are interpreted as the zero object. 
\end{Proposition}
\begin{proof}
We prove the first exact sequence. The second sequence can be proved similarly. Alternatively, the second can be derived from the first by applying the bar involution (see Section \ref{sec:barinvolution}) or, in a more elementary fashion, by applying the functor $\D$. % (see Example \ref{ex:appD}). 

We prove the case $\ell=0$. The general case is proved similarly or by tensoring with the global line bundle $\O(\ell)$. Following the notation from the proof of Proposition \ref{prop:comm} we have 
\begin{gather*}
\cP_{k,1} \conv \cP_{k,-1} \cong \om_*(\O_{\Gr^{(k,k)}_{GL_n}} \otimes \det(L_0/L_1) \otimes \det(L_1/L_2)^{-1}) \la \frac12 \dim \Gr_{GL_n}^{(k,k)} \ra \\
\cP_k \conv \cP_k \cong \om_*(\O_{\Gr^{(k,k)}_{GL_n}}) \la \frac12 \dim \Gr_{GL_n}^{(k,k)} \ra.
\end{gather*}
Now, on $\Gr^{(k,k)}_{GL_n}$ we have the map of rank $k$ vector bundles given by 
$$t: L_0/L_1 \rightarrow L_1/L_2 \{2\}.$$
As in the proof of Proposition \ref{prop:comm}, this map is degenerate along a divisor $D$ where 
$$\O_{\Gr^{(k,k)}_{GL_n}}(D) \cong \O_{\Gr^{(k,k)}_{GL_n}} \otimes \det(L_0/L_1)^{-1} \otimes \det(L_1/L_2\{2\}).$$
Substituting into the standard triangle $\O_{\Gr^{(k,k)}_{GL_n}}(-D) \rightarrow \O_{\Gr^{(k,k)}_{GL_n}} \rightarrow \O_D$ we obtain
$$\O_{\Gr^{(k,k)}_{GL_n}} \otimes \det(L_0/L_1) \otimes \det(L_1/L_2)^{-1} \{-2k\} \rightarrow \O_{\Gr^{(k,k)}_{GL_n}} \rightarrow \O_D.$$
Shifting by $\la \frac12 \dim \Gr_{GL_n}^{(k,k)} \ra$ and applying $\om_*$ we arrive at the exact triangle
$$\om_*(\O_D \la \frac12 \dim \Gr_{GL_n}^{(k,k)} \ra [-1]) \rightarrow \cP_{k,1} \conv \cP_{k,-1} \{-2k\} \rightarrow \cP_k \conv \cP_k.$$
Since $\frac12 \dim \Gr_{GL_n}^{(k,k)} = k(n-k) = \frac12 \dim \Gr_{GL_n}^{(k-1,k+1)}+1$ it remains to show that 
\begin{equation*}\label{eq:needtoshow4}
\om_*(\O_D) \cong \om_*(\O_{\Gr^{(k-1,k+1)}_{GL_n}}).
\end{equation*}
This follows since $D$ has a resolution 
$$\tD = \{L_2 \overset{k-1}\subset L_1' \overset{1}\subset L_1 \overset{k}\subset L_0: tL_0 \subset L_1', tL_1 \subset L_2\}$$
and we can then factor $\om_*(\O_D)$ as $\tD \rightarrow \Gr_{GL_n}^{(k-1,k+1)} \xrightarrow{\om} \Gr^{\omega^\vee_{k-1}+\omega^\vee_{k+1}}_{GL_n}$ where the first map forgets $L_1$ and the second forgets $L_1'$. 
\end{proof}

\subsection{Generating the Grothendieck ring}\label{sec:K0generators}

In addition to the $\cP_{k,\ell}$ one can also consider the sheaves 
$$\cP_{k,\ell}^\uj := i_{k*} \left( \O_{\Gr_{GL_n}^k} \otimes 
\bigwedge\nolimits^{\!j_1} (L_0/L) \otimes \dots \otimes \bigwedge\nolimits^{\!j_r}(L_0/L) \otimes \det(L_0/L)^\ell \right) \la \frac12 k(n-k) \ra \{ -|\uj| - k\ell\}$$
where $\uj = (j_1, \dots, j_r)$ is a sequence with $j_i \in [0,k]$, $|\uj|=\sum_i j_i$, and $\ell \in \Z$. Note that allowing $j_i > k$ would add no generality since $L_0/L$ having rank $k$ on $\Gr_{GL_n}^k$ would then imply $\cP_{k,\ell}^\uj = 0$.

On the other hand, if $j_i=0$ then the vector bundle $\bigwedge\nolimits^{\!j_i} (L_0/L)$ is a trivial line bundle so without losing generality we could assume $j_i \neq 0$. Similarly, if $j_i=k$ then the contribution is just the line bundle $\det(L_0/L)$ which can be absorbed into $\det(L_0/L)^\ell$, so we would also lose no generality if we assumed $j_i < k$. However, allowing $j_i$ to be $0$ or $k$ is convenient for stating and proving the following result. 

\begin{Lemma}\label{lem:recursion}
Let $\uj = (j_1, \dots, j_r)$ be a sequence with $j_i \in [1,k]$. Then in $K^{GL_n(\cO) \rtimes \Gm}(\Gr_{GL_n})$ we have 
$$[\cP_{k,\ell}^\uj] = \sum_{s=0}^{k-1} \sum_{\ue \in \{0,1\}^r} (-1)^s (-q)^{-k+2s+1} [\cP_{k-1,\ell}^{(\uj-\ue,s)}] \cdot [\cP_{1,-s+|\ue|+\ell}].$$
\end{Lemma}
\begin{proof}
Consider the varieties
$$\{L_2 \overset{1}\subset L_1 \overset{k-1}\subset L_0: tL_0 \subset L_2\} =: W \subset \Gr^{(k-1,1)}_{GL_n} = \{L_2 \overset{1}\subset L_1 \overset{k-1}\subset L_0: tL_0 \subset L_1, tL_1 \subset L_2 \}.$$
Then
\begin{equation}\label{eq:Pjkl}
\cP_{k,\ell}^\uj \{|\uj|+k\ell\} \cong \om_* \left( \O_{\Gr^{(k-1,1)}_{GL_n}} \otimes \bigwedge\nolimits^{\uj}(L_0/L_2) \otimes \det(L_0/L_2)^\ell \otimes \O_W \right) \la \frac12 k(n-k) \ra
\end{equation}
where $\om: \Gr^{(k-1,1)}_{GL_n} \rightarrow \Gr_{GL_n}$ is the usual multiplication map given by forgetting $L_1$ and $\bigwedge^{\raisebox{-0.4ex}{\scriptsize $\uj$}} (L_0/L_2)$ is short for $\bigwedge^{\raisebox{-0.4ex}{\scriptsize $j_1$}}(L_0/L_2) \otimes \dots \otimes \bigwedge^{\raisebox{-0.4ex}{\scriptsize $j_r$}}(L_0/L_2)$. 

To deal with $\O_W$ note that $W \subset \Gr^{(k-1,1)}_{GL_n}$ is carved out by the zero section of the map $t: L_0/L_1 \rightarrow L_1/L_2 \{2\}$. Thus, if we denote by $\cV$ the vector bundle $(L_0/L_1) \otimes (L_1/L_2)^\vee \{-2\}$, we have a Koszul resolution 
\begin{equation}\label{eq:koszul}
\bigwedge\nolimits^{\!k-1} \cV \rightarrow \dots \rightarrow \bigwedge\nolimits^{\!2} \cV \rightarrow \cV \rightarrow \O_{\Gr^{(k-1,1)}_{GL_n}} \rightarrow \O_W
\end{equation}
of $\O_W$ on $\Gr^{(k-1,1)}_{GL_n}$. Moreover, we have 
$$\bigwedge\nolimits^{\!s} \cV \cong \bigwedge\nolimits^{\!s}(L_0/L_1) \otimes \det(L_1/L_2)^{-s} \{-2s\}$$ 
while, at the level of Grothendieck rings, $[(L_0/L_2)] = [(L_0/L_1)]+[(L_1/L_2)]$ which implies that
$$[\bigwedge\nolimits^{\!j_i}(L_0/L_2)] = [\bigwedge\nolimits^{\!j_i}(L_0/L_1)] + [\bigwedge\nolimits^{\!j_i-1}(L_0/L_1)] \cdot [(L_1/L_2)].$$
Putting this together gives us 
\begin{align*}
& \left[\om_* \left( \O_{\Gr^{(k-1,1)}_{GL_n}} \otimes \bigwedge\nolimits^{\!\uj}(L_0/L_2) \otimes \det(L_0/L_2)^\ell \otimes \bigwedge\nolimits^s \cV \right) \right] \\
= & \sum_{\ue \in \{0,1\}^r} q^{2s} \left[ \om_* \left( \O_{\Gr^{(k-1,1)}_{GL_n}} \otimes \bigwedge\nolimits^{\!\uj-\ue}(L_0/L_1) \otimes \bigwedge\nolimits^{\!s}(L_0/L_1) \otimes \det(L_0/L_1)^\ell \otimes \det(L_1/L_2)^{-s+|\ue|+\ell} \right) \right] \\
= & \sum_{\ue \in \{0,1\}^r} q^{2s} q^{-|\uj|-k\ell} [\cP_{k-1,\ell}^{(\uj-\ue,s)} \la -\frac12 (k-1)(n-k+1) \ra] \cdot [\cP_{1,-s+|\ue|+\ell} \la -\frac12 (n-1) \ra ] 
\end{align*}
where $|\ue| := \sum_i e_i$. The result follows by combining this identity with (\ref{eq:Pjkl}) and the Koszul resolution (\ref{eq:koszul}) (while keeping in mind that the class of $\la 1 \ra$ in K-theory is $-q$). 
\end{proof}

\begin{Corollary}\label{cor:K0generators}
The Grothendieck ring of $\cP_{coh}^{GL_n(\O) \rtimes \Gm}(\Gr_{GL_n})$ is generated as a $\Z[q^{\pm 1/2}]$-algebra by $[\cP_{1,\ell}]$ for $\ell \in \Z$ and $[\cP_n]^{-1}$. 
\end{Corollary}
\begin{proof}
For any dominant $\l = \sum_{i=1}^r \omega_{k_i}$ we have the resolution 
$$\om: \Gr_{GL_n}^{k_1} \ttimes \dots \ttimes \Gr_{GL_n}^{k_r} \rightarrow \oGr_{GL_n}^{\l}.$$ 
By the projection formula and the fact $\oGr_{GL_n}^{\l}$ has rational singularities the map $\om_*$ is surjective on K-theory. Moreover, $K^{GL_n(\O) \rtimes \Gm}(\Gr_{GL_n}^{k_1} \ttimes \dots \ttimes \Gr_{GL_n}^{k_r})$ is isomorphic to 
$$K^{GL_n(\O) \rtimes \Gm}(\Gr_{GL_n}^{k_1}) \otimes_{K^{GL_n(\O) \rtimes \Gm}(pt)} \dots \otimes_{K^{GL_n(\O) \rtimes \Gm}(pt)}  K^{GL_n(\O) \rtimes \Gm}(\Gr_{GL_n}^{k_r}).$$
Thus it suffices to show generation of the (equivariant) K-theory of the fixed point $e \in \Gr_{GL_n}$ and of $\Gr^k_{GL_n}$ for $k=1, \dots, n$. 

The former is generated (as an algebra) by
$$\O_e \otimes [\bigwedge\nolimits^i(L_0/tL_0)] \cong [\cP_n^i] \cdot [\cP_n^{-1}]$$
where, following the notation introduced above and used in Lemma \ref{lem:recursion}, $\cP_n^i = \O_{\Gr^n_{GL_n}} \otimes \bigwedge^{\raisebox{-0.4ex}{\scriptsize $i$}}(L_0/L)$. On the other hand, applying Lemma \ref{lem:recursion} repeatedly shows that $[\cP_n^i]$ is generated by the classes $\{[\cP_{1,\ell}]\}_{\ell \in \Z}$.

The latter is generated, as an algebra, by the equivariant K-theory of $e \in \Gr_{GL_n}$ and by $[\cP_k^i]$ for $i=0,\dots,k$. Once again, Lemma \ref{lem:recursion} shows that the $[\cP_k^i]$ are generated by the classes $\{[\cP_{1,\ell}]\}_{\ell \in \Z}$.
\end{proof}

\begin{Remark}
A similar argument as in the proof of Corollary \ref{cor:K0generators} can be used to show that $\cP_{coh}^{GL_n(\O) \rtimes \Gm}(\Gr_{GL_n})$ is generated by $[\cP_{1,\ell}]$ and $[\cP_{-1,\ell}]$ (without the need for $[\cP_n]^{-1}$). This gives another (more geometric) proof of \cite[Theorem 9.3]{FT17}. 
\end{Remark}

\section{Rigidity}\label{sec:rigidity}

In this section we construct the left and right adjoints for convolution in the coherent Satake category, and more generally in $D_{coh}^{G(\O) \rtimes \Gm}(\Gr_G)$. We illustrate the construction with some explicit examples in the case $G=GL_n$.

Let us briefly contrast our setting with that of the constructible Satake category $\cP^{G(\cO)}(\Gr_G)$. This category is symmetric monoidal and its global sections functor is a fiber functor. Thus following \cite[Proposition 1.20]{DM82} to establish its rigidity it suffices to show that objects with one-dimensional image under the fiber functor are invertible. This condition is moreover sufficient to apply Tannakian reconstruction, leading in the case of $\cP^{G(\cO)}(\Gr_G)$ to the geometric Satake equivalence (c.f. \cite[Theorem 2.7 (3)]{BR17}).

However, the global sections functor on the coherent Satake category is not a fiber functor (it does not respect the monoidal structure), and its convolution product is not symmetric. This forces us to prove rigidity by directly checking the usual properties of adjoints. An adaptation of our proof should also apply to the constructible Satake category to recover rigidity in that case. 

\subsection{Left and right duals}

We begin with the diagram 
$$\Gr_G \xleftarrow{\pi} (G(\cK) \rtimes \Gm) \times_{G(\O) \rtimes \Gm} \Gr_G \xrightarrow{\om} \Gr_G$$
where $\pi(g,x) = [g]$ and $\om(g,x) = [gx]$. The projection $\pi$ has a section 
$$s: \Gr_G \rightarrow (G(\cK) \rtimes \Gm) \times_{G(\O) \rtimes \Gm} \Gr_G$$
given by $s([g]) = (g,[g^{-1}])$.

For $\shF \in D_{coh}^{G(\O) \rtimes \Gm}(\Gr_G)$ we define $\shF^* := s^* (\O\, \tbox\, \shF)$. More concretely, suppose $\shF$ is supported on $\oGr_G^{\l^\vee}$. Let  ${\l^\vee}^* := - w_0 \, {\l^\vee}$ denote the dominant conjugate of $-\l^\vee$. Then $s$ restricts to a section 
$$s: \oGr_G^{{\l^\vee}^*} \rightarrow \oGr_G^{{\l^\vee}^*} \ttimes \oGr_G^{{\l^\vee}}$$
and $\shF^* \cong s^*(\O_{\oGr_G^{{\l^\vee}^*}} \tbox \shF)$. Since $\oGr_G^{{\l^\vee}^*} \ttimes \oGr_G^{\l^\vee}$ is a locally trivial fibration it follows that $\shF^*$ is also coherent on $\oGr_G^{{\l^\vee}^*}$. It is easy to see that $\cF \mapsto \cF^*$ is an involution of $D_{coh}^{G(\O) \rtimes \Gm}(\Gr_G)$. 

\begin{Theorem}\label{thm:adjoints}
The monoidal category $D^{G(\O) \rtimes \Gm}_{coh}(\Gr_G)$ is rigid. The left and right duals of $\shF \in D_{coh}^{G(\O) \rtimes \Gm}(\Gr_G)$ are $\shF^L = \D(\shF^*)$ and $\shF^R = (\D\shF)^*$. 
\end{Theorem}

We prove this result starting in Section \ref{sec:proofoutline}. It yields the following corollary.

\begin{Corollary}\label{cor:adjoints}
The left and right duals of a perverse sheaf are again perverse, hence $\cP_{coh}^{G(\O) \rtimes \Gm}(\Gr_G)$ is rigid.
\end{Corollary}
\begin{proof}
Since $\D$ preserves $\cP_{coh}^{G(\O) \rtimes \Gm}(\Gr_G)$ it remains to show that $\cF \mapsto \cF^*$ also preserves it. 

Consider the finite-dimensional convolution variety 
$$i_{({\l^\vee}, {\l^\vee}^*)}: \Gr_G^{{\l^\vee}} \ttimes \Gr_G^{{\l^\vee}^*} \into (G(\cK) \rtimes \Gm) \times_{G(\O) \rtimes \Gm} \Gr_G.$$ 
Since $\shF$ is perverse and $\Gr_G^{{\l^\vee}} \ttimes \Gr_G^{{\l^\vee}^*}$ is a Zariski-locally trivial fibration, $i_{({\l^\vee}, {{\l^\vee}^*})}^* (\cO\, \tbox \, \shF)$ is supported in degrees $\leq \frac12 \dim \Gr_{G}^{{\l^\vee}^*}$. From the commutativity of 
$$\xymatrix{
\Gr_G^{{\l^\vee}} \ar[rrr]^{i_{{\l^\vee}}} \ar[d]^s & & & \Gr_G \ar[d]^s \\
\Gr_G^{{\l^\vee}} \ttimes \Gr_G^{{\l^\vee}^*} \ar[rrr]^{i_{({\l^\vee}, {\l^\vee}^*)}} & & & (G(\cK) \rtimes \Gm) \times_{G(\O) \rtimes \Gm} \Gr_G }$$
it follows that $i_{{\l^\vee}}^* \shF^*$ is supported in degrees $\leq \frac12 \dim \Gr_G^{{\l^\vee}}$ (where we used that $\Gr_G^{\l^\vee} = \dim \Gr_G^{{\l^\vee}^*}$). The analogous statement for $i_{{\l^\vee}}^! \shF^*$ follows similarly. 
\end{proof}

\begin{Remark}
In general $\shF^L$ and $\shF^R$ are not isomorphic. This is in contrast to the constructible case where left and right duals in $\cP^{G(\cO)}(\Gr_G) \cong \Rep\, G^\vee$ agree. %(indeed the left and right adjoints of a $G^\vee$-representation are the same). 

The difference between $\shF^L$ and $\shF^R$ in our case can be traced back to the fact that $(\omega_{\oGr^{\l^\vee}_G})^*$ (defined using the pullback $s^*$) is {\em not} isomorphic to $\omega_{\oGr^{{\l^\vee}^*}_G}$. However, they are isomorphic in the constructible case (replacing the coherent dualizing sheaf with the constructible one). 
\end{Remark}

\subsection{The adjunction maps}

We now describe the maps relating $\shF$, $\shF^L$, and $\shF^R$. From hereon fix $Y = \oGr_G^{\l^\vee}$ and $Y^* = \oGr_G^{{\l^\vee}^*}\!\!$. We will write $D_{coh}(Y)$ for $D^{G(\O) \rtimes \Gm}_{coh}(\Gr_G^{\l^\vee})$ and likewise for $Y^*$. We have the map $s: Y^* \rightarrow Y^* \ttimes Y$ and, for $\shF \in D_{coh}(Y)$, we can consider the composition
\begin{equation}\label{eq:adj1}
\shF^L \tbox \shF = \D(\shF^*) \tbox \shF \xrightarrow{adj} s_* {s}^* (\shF^L \tbox \shF) \cong s_* (\D(\shF^*) \otimes \shF^*) \xrightarrow{m} s_* \omega_{Y^*},
\end{equation}
where $m$ is the natural evaluation map $\cHom(\shF^*, \omega_{Y^*}) \otimes \shF^* \to \omega_{Y^*}$. 
Since $p := \om \circ s$ projects $Y^*$ to the point $e \in \Gr_G$ we have an adjunction map
$$p_* \omega_{Y^*} = p_* p^! \O_e \rightarrow \O_e.$$ 
Composing with the pushforward of (\ref{eq:adj1}) to $\Gr_G$ we thus obtain
\begin{equation}\label{eq:adj2}
\shF^L * \shF \rightarrow \om_* s_* \omega_{Y^*} \xrightarrow{adj} \O_e.
\end{equation}

To define the map $\O_e \rightarrow \shF * \shF^L$ we now let $s: Y \rightarrow Y \ttimes Y^*$ and note that 
\begin{equation}\label{eq:adj3}
s^! (\shF \tbox \shF^L) = s^! (\shF \tbox \D(\shF^*)) \cong \D s^* \D (\shF \tbox \D(\shF^*)) \cong \D(\D(\shF) \otimes \shF),
\end{equation}
where in the last isomorphism we use the fact that $\shF^{**} \cong \shF \cong \D \D (\shF)$. Moreover, we have 
$$\O_Y \xrightarrow{\iota} \cHom(\D(\shF),\D(\shF)) \cong \D(\D(\shF) \otimes \shF)$$  
which gives us maps 
$$\O_Y \rightarrow s^!(\shF \tbox \shF^L) \ \ \Leftrightarrow \ \ s_* \O_Y \rightarrow \shF \tbox \shF^L.$$
We now push forward $\om$ to $\Gr_G$. As before, $p := \om \circ s$ projects $Y$ to the point $e \in \Gr_G$ and we have a composition
\begin{equation}\label{eq:adj4}
\O_e \xrightarrow{adj} \om_* s_* \O_Y \rightarrow \shF * \shF^L,
\end{equation}
where the first map is the natural adjunction $\O_e \rightarrow p_* p^* \O_e = p_* \O_Y$. 

The maps involving $\shF^R$ are defined similarly. Equivalently they can be obtained from those above. For example, replacing $\shF$ with $\shF^R$ in (\ref{eq:adj2}) we get a map $(\shF^R)^L * \shF^R \rightarrow \O_e$. But $(\shF^R)^L \cong \D(\D(\shF)^{**}) \cong \shF$ so this gives the map $\shF * \shF^R \rightarrow \O_e$.

\subsection{Some examples}

Before proceeding to the proof of Theorem \ref{thm:adjoints} we illustrate the process of applying $(\cdot)^*$ and taking adjoints in the case $G=GL_n$. 

\begin{Lemma}\label{lem:example1}
We have $\cP_{k,\ell}^* \cong \cP_{-k,\ell} \cong \cP_{n,-\ell}^{-1} * \cP_{n-k,-\ell}$. 
\end{Lemma}
\begin{proof}
Consider $\Gr^{-k}_{GL_n} \ttimes \Gr^k_{GL_n} = \{L_0 \overset{k}\subset L_1 \overset{k}\supset L_2: tL_1 \subset L_0, tL_1 \subset L_2\}$. We have the section 
\begin{align*}
s: \Gr^{-k}_{GL_n} &\rightarrow \Gr^{-k}_{GL_n} \ttimes \Gr^k_{GL_n}, \quad
(L_0 \overset{k}\subset L) \mapsto (L_0 \overset{k}\subset L \overset{k}\supset L_0),
\end{align*}
from which we see that
\begin{align*}
\cP_{k,\ell}^* &= s^*(\O_{\Gr^{-k}_{GL_n}} \tbox \cP_{k,\ell}) \cong s^*(\O_{\Gr^{-k}_{GL_n} \ttimes \Gr^k_{GL_n}} \otimes \det(L_1/L_2)^{\ell} \la \frac12 k(n-k) \ra \{-k\ell\}) \cong \cP_{-k,\ell}.
\end{align*}
To rewrite $\cP_{-k,\ell}$ consider the convolution
$$\cP_{-k,\ell} * \cP_n \cong \O_{\Gr^{n-k}_{GL_n}} \otimes \det(t^{-1}L/L_0)^\ell \la \frac12 k(n-k) \ra \{-k\ell\}.$$
Now on $\Gr^{n-k}_{GL_n}$ the isomorphism $t^{-1}L/L_0 \xrightarrow{\sim} L/tL_0 \{2\}$ implies that 
\begin{align*}
\det(t^{-1}L/L_0)^{\ell} 
\cong \det(L/tL_0)^{\ell} \{2k\ell\} \cong \det(L_0/tL_0)^{\ell} \otimes \det(L_0/L)^{-\ell} \{2k\ell\}. 
\end{align*}
It follows that 
$$\cP_{-k,\ell} * \cP_n \cong \cP_{n,\ell} * \cP_n^{-1} * \cP_{n-k,-\ell} \{2k\ell\}$$
where we use that tensoring with $\det(L_0/tL_0)^{\ell}$ is equivalent to convolving on the left with $\cP_{n,\ell} * \cP_{-n}$. The result now follows since by Remark \ref{rem:comm} we have
\begin{align*}
\cP_{n-k,-\ell} * \cP_n^{-1} &\cong \cP_n^{-1} * \cP_{n-k,-\ell} \{2(n-k)\ell\} \\
\cP_{n,\ell} * \cP_n^{-2} &\cong \cP_{n,-\ell}^{-1} \{-2n\ell\}.\qedhere
\end{align*}
\end{proof}

\begin{Lemma}\label{lem:example2}
We have $\D \cP_{k,\ell} \cong \cP_{n,k} * \cP_n^{-1} * \cP_{k,-\ell-n} \{k(n-k)\}$. 
\end{Lemma}
\begin{proof}
By definition $\D \cP_{k,\ell}$ is the pushforward to $\Gr_{GL_n}$ of 
$$\omega_{\Gr^k_{GL_n}} \otimes \det(L_0/L)^{-\ell} \la -\frac12 k(n-k) \ra \{k\ell\}.$$
Now, the usual calculation of the dualizing sheaf of the finite Grassmannian shows that
\begin{equation}\label{eq:omega}
\omega_{\Gr^k_{GL_n}} \cong \O_{\Gr^k_{GL_n}} \otimes \det(L_0/L)^{-n} \otimes \det(L_0/tL_0)^k [k(n-k)].
\end{equation}
Simplifying we obtain
$$\D \cP_{k,\ell} \cong \cP_{k,-\ell-n} \otimes \det(L_0/tL_0)^k \{-kn\} \{k(n-k)\}.$$
The result now follows since tensoring with $\det(L_0/tL_0)^k \{-kn\}$ is equivalent to convolving on the left with $\cP_{n,k} * \cP_{n,0}^{-1}$. 
\end{proof}

In light of Theorem \ref{thm:adjoints}, Lemmas \ref{lem:example1} and \ref{lem:example2} together with some simplification imply the following.

\begin{Proposition}\label{prop:example3}
The left and right adjoints of $\cP_{k,\ell}$ are given by 
\begin{align*}
\cP_{k,\ell}^L &= \D(\cP_{k,\ell}^*) \cong \cP_{n,k+\ell-n}^{-1} * \cP_{n-k,\ell-n} \{k(n-k)\} \\
\cP_{k,\ell}^R &= (\D \cP_{k,\ell})^* \cong \cP_{n,-k+\ell+n}^{-1} * \cP_{n-k,\ell+n} \{-k(n-k)\}.
\end{align*}
\end{Proposition}
\begin{Remark}\label{rem:adj}
In fact, if one does not insist on having everything written in terms of convolutions, the expressions above can be written more simply as 
\begin{align*}
\cP_{k,\ell}^L &\cong \cP_{-k,-\ell} \otimes \omega_{\Gr^{-k}_{GL_n}} \la -k(n-k) \ra \\
\cP_{k,\ell}^R &\cong \cP_{-k,-\ell} \otimes \omega_{\Gr^{-k}_{GL_n}}^{-1} \la k(n-k) \ra
\end{align*}
where $\omega_{\Gr^{-k}_{GL_n}} \cong \O_{\Gr^{-k}_{GL_n}} \otimes \det(L/L_0)^n \otimes \det(t^{-1}L_0/L_0)^{-k} [k(n-k)]$ is the dualizing sheaf. 
\end{Remark}

\subsection{Outline of the proof}\label{sec:proofoutline}

We begin the proof of Theorem \ref{thm:adjoints} by outlining the argument.

It suffices to show that the composition 
\begin{equation}\label{eq:adj5}
\shF = \O_e * \shF \rightarrow \shF * \shF^L * \shF \rightarrow \shF * \O_e = \shF
\end{equation}
is an isomorphism (the one involving $\shF^R$ follows by replacing $\shF$ with $\shF^R$).
To be more precise, showing this composition is the identity would establish Theorem \ref{thm:adjoints} and that the unit and counit maps are exactly as indicated. However, suppose we know that the composition in~(\ref{eq:adj5}) is some isomorphism $f$. Then replacing the map $\O_e \rightarrow \shF * \shF^L$ with 
$$\O_e \rightarrow \shF * \shF^L \xrightarrow{f^{-1} * id} \shF * \shF^L$$
makes (\ref{eq:adj5}) equal to the identity on the nose. This is sufficient to prove Theorem \ref{thm:adjoints}, and showing that (\ref{eq:adj5}) is an isomorphism (rather than the identity) will simplify our arguments. 

To understand the composition in (\ref{eq:adj5}) the following diagram plays a large role.
\begin{equation}\label{eq:main}
\begin{gathered}\xymatrix{
Y \ar[rr]^{j_1} \ar[d]_{j_2} & & Y \ttimes Y^* \ar[d]^{i_2} \\
Y \times Y \ar[rr]^{i_1} & & Y \ttimes Y^* \ttimes Y }
\end{gathered}
\end{equation}
Here $j_1, i_1$ and $i_2$ are induced by $s$ while $j_2$ is the diagonal. The map in (\ref{eq:adj5}) is obtained by pushing forward to $\Gr_G$ the composition
\begin{equation}\label{eq:adj6}
i_{1*} (\O_Y \bbox \shF) \rightarrow \shF \tbox \D(\shF^*) \tbox \shF \rightarrow i_{2*}(\shF \tbox \omega_{Y^*}).
\end{equation}
This map is adjoint to the composition
\begin{equation}\label{eq:adj10}
\O_Y \bbox \shF \rightarrow i_1^! (\shF \tbox \D(\shF^*) \tbox \shF) \rightarrow i_1^! i_{2*} (\shF \tbox \omega_{Y^*}) \cong j_{2*} j_1^! (\shF \tbox \omega_{Y^*}) \cong j_{2*}(\shF).
\end{equation}
The first isomorphism is a consequence of the fact that the square in (\ref{eq:main}) is independent (see Lemma \ref{lem:*}). The difficult part is to show that the composition in (\ref{eq:adj10}) is the obvious restriction morphism (see Section \ref{sec:rest}). The rest of the argument, which is easier, uses this fact to show that the composition in (\ref{eq:adj5}) is an isomorphism (see Section \ref{sec:proofadj}). 

Of the maps in (\ref{eq:adj10}) it is the second one which is harder to understand. This is done in Section \ref{sec:hardhalf}. We have $i_1^! (\shF \tbox \D(\shF^*) \tbox \shF) \cong \cHom(\D(\shF), \D(\shF)) \tbox \shF$ and the goal is to show that the second map in (\ref{eq:adj10}) is equivalent to the composition 
$$\cHom(\D(\shF), \D(\shF)) \tbox \shF \rightarrow j_{2*} \cHom(\D(\shF), \D(\shF) \otimes \shF) \xrightarrow{m} j_{2*} \cHom(\D(\shF), \omega_Y) \cong j_{2*} \shF$$ 
where the first map above is restriction composed with the natural map 
$$\cHom(\D(\shF), \D(\shF)) \otimes \shF \rightarrow \cHom(\D(\shF), \D(\shF) \otimes \shF).$$
This is the content of the diagram in Proposition \ref{prop:adjoints} on taking $\shF = \shG$ so that $\shA = \shF \tbox \D(\shF^*) \tbox \shF$. 

In carrying out this analysis we will be composing operations such as tensor product which need not preserve $D_{coh}$ on the varieties involved, as they are generally singular. Thus our arguments necessarily take place in the unbounded derived category $D_{qcoh}$ of quasicoherent sheaves, even though the statement of Theorem \ref{thm:adjoints} refers only to $D_{coh}$. 
% A number of standard results that hold for $D_{coh}(X)$ with $X$ smooth no longer hold in this setting. For example, $\cHom(\shF,\shG)$ is not (in general) isomorphic to $\cHom(\shF, \O_X) \otimes \shG$. With this in mind, w
We collect in an appendix various results about $D_{qcoh}$ that we will need. Since on a technical level the rest of Section \ref{sec:rigidity} is independent from those that follow, we emphasize that the reader who wishes to skip to Section \ref{sec:r-matrix} may safely do so. 

We make frequent use of the natural maps $\iota: \O_X \rightarrow \cHom(\shF,\shF)$ and $m: \cHom(\shF,\shG) \otimes \shF \rightarrow \shG$. A particular instance of $m$ which will appear frequently is the map
$$\D(\shF) \otimes \shF = \cHom(\shF,\omega_X) \otimes \shF \xrightarrow{m} \omega_X$$
where $\D(\shF) = \cHom(\shF,\omega_X)$ is the standard duality functor. For $\shF \in D_{coh}(X)$ we have $\D \D(\shF) \cong \shF$, but for a general $\shF \in D_{qcoh}(X)$ we just have a map $\shF \rightarrow \D \D(\shF)$ induced by a natural transformation $\tau: id \rightarrow \D \D$. 

\begin{Lemma}\label{lem:*}
The square in (\ref{eq:main}) is independent in the sense of Definition \ref{def:independent}.
\end{Lemma}
\begin{proof}
We need to show that $i_2^* i_{1*} (\cM) \xrightarrow{\sim} j_{1*}j_2^*(\cM)$ for $\cM \in D_{qcoh}(Y \times Y)$. Using the commutativity of \cite[Eq. 3.10.2.3]{Lip} it suffices to consider the case $\cM = \O_{Y \times Y}$.

Consider the following commutative diagram which extends (\ref{eq:main}).
\begin{equation}\label{eq:21}
\begin{gathered}
\xymatrix{
Y \ar[rr]^{j_1} \ar[d]_{j_2} & & Y \ttimes Y^* \ar[d]^{i_2} \\
Y \times Y \ar[d]_{p_2} \ar[rr]^{i_1} & & Y \ttimes Y^* \ttimes Y \ar[d]^{p_1} \\
Y \ar[rr]^{j_1} & & Y \ttimes Y^* }
\end{gathered}
\end{equation}
Here $p_1$ and $p_2$ forget the last factor $Y$ and hence both vertical compositions are the identity. Now consider the composition 
\begin{equation}\label{eq:25}
i_2^* p_1^* j_{1*} (\O_Y) \rightarrow i_2^* i_{1*} p_2^* (\O_Y) \rightarrow j_{1*} j_2^* p_2^* (\O_Y).
\end{equation}
Since $p_2 \circ j_2$ and $p_1 \circ i_2$ are both the identity it follows from Lemma \ref{lem:2} that the composition is an isomorphism. Moreover, the first map is an isomorphism since $p_1$ is a fiber bundle and hence flat. It follows that the second map in (\ref{eq:25}) is also an isomorphism, which is equivalent to what we needed to show.
\end{proof}

\subsection{Hard half of the composition}\label{sec:hardhalf}

There are several natural transformations that we will use, including $\eta,\theta$ and perhaps most interestingly $\phi$. We refer to the appendix for their definition and properties. 

Fix $\shF, \shG \in D_{coh}(Y)$ and write $\shA := \shF \tbox \D(\shG^*) \tbox \shG$. In this section we prove the following key result which will be used to understand the second map in (\ref{eq:adj10}). 

\begin{Proposition}\label{prop:adjoints}
There exists a commutative diagram
\begin{equation*}%\label{eq:commute7}
\begin{gathered}
\xymatrix{
i_1^!(\shA) \ar[r]^{adj} \ar[d]^\sim & i_1^! i_{2*} i_2^* (\shA) \ar[r]^m & i_1^! i_{2*} (\shF \tbox \omega_{Y^*}) \ar[d]^\sim \\
\cHom(\D(\shF),\D(\shG)) \tbox \shG \ar[r] & j_{2*} \cHom (\D(\shF), \D(\shG) \otimes \shG) \ar[r]^m & j_{2*} \cHom(\D(\shF),\omega_Y) \cong j_{2*} \shF.}
\end{gathered}
\end{equation*}
\end{Proposition} 

The rest of the section will contain the proof of Proposition \ref{prop:adjoints}. We will use the following commutative diagram .
\begin{equation}\label{eq:commute3}
\begin{gathered}\xymatrix{
i_1^! (\shA) \ar[d]^{adj} \ar[r]^{adj} & i_1^! i_{2*} i_2^* (\shA) \ar[r] \ar[d]^{\sim}  & i_1^! i_{2*} (\shF \tbox \omega_{Y^*}) \ar[d]^{\sim} \\
j_{2*} j_2^* i_1^! (\shA) \ar[r]^{\phi} & j_{2*} j_1^! i_2^* (\shA) \ar[r] \ar[d]^{\theta} & j_{2*} j_1^! (\shF \tbox \omega_{Y^*}) \ar[d]^{\theta} \\
& j_{2*} \D j_1^* \D i_2^* (\shA) \ar[r] & j_{2*} \D j_1^* \D (\shF \tbox \omega_{Y^*})
}
\end{gathered}
\end{equation}
To check that the top left rectangle commutes one writes out $\phi \circ adj$ as 
$$i_1^! \xrightarrow{adj} j_{2*} j_2^* i_1^! \xrightarrow{adj} j_{2*} j_2^* i_1^! i_{2*} i_2^* \xrightarrow{\sim} j_{2*} j_2^* j_{2*} j_1^! i_2^* \xrightarrow{adj} j_{2*} j_1^! i_2^*$$
which is easily rewritten as 
$$i_1^! \xrightarrow{adj} i_1^! i_{2*} i_2^* \xrightarrow{\sim} j_{2*} j_1^! i_2^*.$$
The other rectangles commute for more obvious reasons. Note that the map we are interested in is the composition in the top row of (\ref{eq:commute3}).

We will first focus on the case $\shF = \omega_Y$. Let us denote $\shB = \omega_Y \tbox \D(\shG^*) \tbox \shG$. 

\begin{Lemma}\label{lem:5}
There exists an isomorphism $\gamma: \D(\D(\shG) \otimes \shG) \xrightarrow{\sim} j_1^* \D i_2^* (\shB)$ which fits into a commutative diagram 
$$\xymatrix{
\D j_1^* \D i_2^* (\shB) \ar[rr]^{\sim} \ar[d]^{\D \gamma} & & \D j_1^* \D (\omega_Y \tbox (\D(\shG^*) \otimes \shG^*)) \ar[rr]^{m} & & \D j_1^* \D (\omega_Y \tbox \omega_{Y^*}) \ar[d]^{\sim} \\
\D \D (\D(\shG) \otimes \shG) \ar[rr]^{m} & & \D \D (\omega_Y) \ar[rr]^{\sim} & & \omega_Y.
} $$
\end{Lemma}
\begin{proof}
We use the following two general facts. First, for any morphism $f: S \rightarrow T$ and $\cM \in D_{qcoh}(T)$ the following diagram commutes.
$$\xymatrix{
f^* \O_T \ar[d]^{\sim} \ar[rrr]^{f^* \iota} & & & f^* \cHom(\cM,\cM) \ar[d] \\
\O_S \ar[rrr]^{\iota} & & & \cHom(f^* \cM, f^* \cM)
}$$
Second, identifying $\D(\D(\cM) \otimes \cM) \cong \cHom(\cM,\D\D(\cM))$ the map 
$$\D \cHom(\cM,\D\D(\cM)) \xrightarrow{\D \tau} \D \cHom(\cM, \cM) \xrightarrow{\D \iota} \D \O_T \cong \omega_T$$
can be identified with the composition
$$\D \cHom(\D(\cM),\D(\cM)) \cong \D \D (\D(\cM) \otimes \cM) \xrightarrow{m} \D \D (\omega_T) \cong \omega_T.$$
Now take $f := j_1: Y \rightarrow Y \ttimes Y^*$ and $\cM := \O_Y \bbox \shG^*$ so that 
$$i_2^*(\shB) \cong \omega_Y \tbox (\D(\shG^*) \otimes \shG^*) \cong \D(\cM) \otimes \cM.$$ 
Consider now the map 
$$\D j_1^* \D i_2^* (\shB) \cong \D j_1^* \D (\D(\cM) \otimes \cM) \xrightarrow{\eta} j_1^! \D \D (\D(\cM) \otimes \cM) \xrightarrow{m} j_1^! \D \D (\omega_{Y \ttimes Y^*}) \cong \omega_Y.$$
By the second fact above this map is equivalent to 
$$\D j_1^* \D i_2^* (\shB) \cong \D j_1^* \cHom(\cM, \D \D(\cM)) \xrightarrow{\eta} j_1^! \D \cHom(\cM, \D \D(\cM)) \xrightarrow{\D(\tau \circ \iota)} j_1^! \omega_{Y \ttimes Y^*} \cong \omega_Y.$$
Since $\D \D \cM \cong \cM$ we can rewrite this as 
$$\D j_1^* \D i_2^* (\shB) \cong \D j_1^* \cHom(\cM,\cM) \xrightarrow{\eta} j_1^! \D \cHom(\cM,\cM) \xrightarrow{\D \iota} j_1^! \omega_{Y \ttimes Y^*} \cong \omega_Y$$
and finally as 
$$\D j_1^* \D i_2^* (\shB) \cong \D j_1^* \cHom(\cM,\cM) \xrightarrow{\D \iota} \D j_1^* \O_{Y \ttimes Y^*} \cong \omega_Y.$$ 
Now the map $j_1^* \cHom(\cM,\cM) \rightarrow \cHom(j_1^* \cM, j_1^* \cM) \cong \cHom(\shG,\shG)$ is an isomorphism. So, using the first fact above, we can rewrite this last composition as 
$$\D j_1^* \D i_2^* (\shB) \cong \D j_1^* \cHom(\cM,\cM) \xrightarrow{\sim} \D \cHom (\shG,\shG) \xrightarrow{\D \iota} \omega_Y.$$
The result follows, where $\gamma$ is the isomorphism $j_1^* \D i_2^* (\shB) \cong \cHom(\shG,\shG) \cong \D(\D(\shG) \otimes \shG)$. 
\end{proof}

\begin{Lemma}\label{lem:6}
There exists an isomorphism $\gamma': \D(\D(\shG) \otimes \shG) \xrightarrow{\sim} j_1^* \D i_2^* (\shB)$ which fits into a commutative diagram
$$\xymatrix{
j_2^* i_1^!(\shB) \ar[r]^{\phi} \ar[d]^{\sim} & j_1^! i_2^*(\shB) \ar[r]^{\theta} & \D j_1^* \D i_2^* (\shB) \ar[d]^{\D \gamma'} \\
\D(\shG) \otimes \shG \ar[rr]^{\tau} & & \D\D (\D(\shG) \otimes \shG).
}$$
\end{Lemma}
\begin{proof}
We do this in several steps. First consider the simpler case of $\shB_1 := \cO_Y \tbox \shG^* \tbox \omega_Y$ and the map $j_1^* i_2^! (\shB_1) \xrightarrow{\phi} j_2^! i_1^* (\shB_1)$. We claim this map is an isomorphism. 

To see this consider again the following commutative diagram from (\ref{eq:21}). Now $\shB_1 = p_1^!(\cO_Y \tbox \shG^*)$ and we can look at the composition 
$$j_1^* i_2^! p_1^! (\cO_Y \tbox \shG^*) \xrightarrow{\phi} j_2^! i_1^* p_1^! (\cO_Y \tbox \shG^*) \xrightarrow{\phi} j_2^! p_2^! j_1^* (\cO_Y \tbox \shG^*).$$
By Lemma \ref{lem:2} this composition is just $j_1^* (p_1 \circ i_2)^! \xrightarrow{\phi} (p_2 \circ j_2)^! j_1^*$, which is an isomorphism since $p_1 \circ i_2$ and $p_2 \circ j_2$ are both the identity maps. Since 
$$j_1^* i_2^! p_1^! (\cO_Y \tbox \shG^*) \cong j_2^! i_1^* p_1^! (\cO_Y \tbox \shG^*) \cong j_2^! p_2^! j_1^* (\cO_Y \tbox \shG^*) \cong \shG$$
it follows that $j_1^* i_2^! (\shB_1) \xrightarrow{\phi} j_2^! i_1^* (\shB_1)$ is an isomorphism. 

Next take $\shB_2 := \cO_Y \tbox \shG^* \tbox \D(\shG) \cong \cHom(\cO_Y \tbox \cO_{Y^*} \tbox \shG, \shB_1)$. By Proposition \ref{prop:phi} we have the following commutative diagram.
\begin{align*}
\xymatrix{
j_1^* i_2^! \cHom(\cO_Y \tbox \cO_{Y^*} \tbox \shG, \shB_1) \ar[d]^{\phi} \ar[r]^{\sim} & j_1^* \cHom(\cO_Y \tbox \shG^*, i_2^! \shB_1) \ar[r] & \cHom(\shG, j_1^* i_2^! \shB_1) \ar[d]^{\phi} \\
j_2^! i_1^* \cHom(\cO_Y \tbox \cO_{Y^*} \tbox \shG, \shB_1) \ar[r] & j_2^! \cHom(\cO_Y \tbox \shG, i_1^* \shB_1) \ar[r]^{\sim} & \cHom(\shG, j_2^! i_1^* \shB_1) 
}
\end{align*}
It is not hard to check that the two unmarked horizontal morphisms are isomorphisms. The right hand vertical map $\phi$ is an isomorphism by the argument above. It follows that $j_1^* i_2^! (\shB_2) \xrightarrow{\phi} j_2^! i_1^* (\shB_2)$ is also an isomorphism. 

Next consider the composition 
$$j_1^* \D i_2^* \D (\shB_2) \xrightarrow{\theta^{-1}} j_1^* i_2^! (\shB_2) \xrightarrow{\phi} j_2^! i_1^* (\shB_2) \xrightarrow{\theta} \D j_2^* \D i_1^* (\shB_2).$$
In the notation of Proposition \ref{prop:dual} this composition is denoted $\rho$. Notice that, in this case, the map $\theta$ is an isomorphism and $\phi$ is an isomorphism by the argument above. This means $\rho$ is an isomorphism.

Finally, note that $\D(\shB) \cong \shB_2$ and consider
\begin{equation}\label{eq:23}
j_2^* \D i_1^* \D (\shB) \xrightarrow{\tau} \D \D j_2^* \D i_1^* \D (\shB) \xrightarrow{\D \rho \D} \D j_1^* \D i_2^* \D \D (\shB) \xrightarrow{\tau^{-1}} \D j_1^* \D i_2^* (\shB).
\end{equation}
By Proposition \ref{prop:dual} this composition is equal to 
\begin{equation}\label{eq:24}
j_2^* \D i_1^* \D (\shB) \xrightarrow{\theta^{-1}} j_2^* i_1^! (\shB) \xrightarrow{\phi} j_1^! i_2^* (\shB) \xrightarrow{\theta} \D j_1^* \D i_2^* (\shB).
\end{equation}
Now if we identify $j_2^* \D i_1^* \D (\shB) \cong j_2^* (\D(\shG) \tbox \shG) \cong \D(\shG) \otimes \shG$ then using (\ref{eq:23}) the composition in (\ref{eq:24}) is equivalent to 
$$\D(\shG) \otimes \shG \xrightarrow{\tau} \D \D (\D(\shG) \otimes \shG) \xrightarrow{\D \gamma'} \D \D (\D(\shG) \otimes \shG)$$
where $\gamma'$ is the isomorphism $(\rho \D) \circ \tau^{-1}$. 
\end{proof}

\begin{Corollary}\label{cor:1}
There exists a commutative diagram
$$\xymatrix{
j_2^* i_1^!(\shB) \ar[r]^\phi \ar[d]^\sim & j_1^! i_2^* (\shB) \ar[rr]^{\theta} & & \D j_1^* \D i_2^* (\shB) \ar[r] & \D j_1^* \D (\omega_Y \tbox \omega_{Y^*}) \ar[d]^\sim \\
\D(\shG) \otimes \shG \ar[rrrr]^{m} & & & & \omega_Y. 
}$$
\end{Corollary}
\begin{proof}
Combining Lemmas \ref{lem:5} and \ref{lem:6} we have a commutative diagram of the form
\begin{equation}\label{eq:commute4} 
\begin{gathered}\xymatrix{
j_2^* i_1^!(\shB) \ar[r]^\phi \ar[d]^\sim & j_1^! i_2^* (\shB) \ar[r]^{\theta} & \D j_1^* \D i_2^* (\shB) \ar[r] & \D j_1^* \D (\omega_Y \tbox \omega_{Y^*}) \ar[d]^\sim \\
\D(\shG) \otimes \shG \ar[r]^{\tau} & \D\D(\D(\shG) \otimes \shG) \ar[r]^{\D \gamma''} & \D\D(\D(\shG) \otimes \shG) \ar[r]^{m} & \omega_Y
}
\end{gathered}
\end{equation}
for some isomorphism $\gamma''$ of $\D(\D(\shG) \otimes \shG)$. 

Now if we identify $\D(\D(\shG) \otimes \shG) \cong \cHom(\shG,\shG)$ then, as in the proof of Lemma \ref{lem:5}, the composition $m \circ \D\gamma''$ in the diagram above is equivalent to $\D$ applied to the following composition
\begin{equation}\label{eq:26}
\O_Y \xrightarrow{\iota} \cHom(\shG,\shG) \xrightarrow{\gamma''} \cHom(\shG,\shG).
\end{equation}
Now, by adjunction, $\Hom(\O_Y, \cHom(\shG,\shG)) \cong \Hom(\shG,\shG)$. This identifies the composition in (\ref{eq:26}) with some $\delta \in \Hom(\shG,\shG)$. One can then write the composition in (\ref{eq:26}) as $\O_Y \xrightarrow{\iota} \cHom(\shG,\shG) \xrightarrow{\delta} \cHom(\shG,\shG)$ where $\delta$ acts on the second factor. This allows us to write the composition $m \circ \D\gamma''$ as 
$$\D\D(\D(\shG) \otimes \shG) \xrightarrow{\D\delta} \D\D(\D(\shG) \otimes \shG) \xrightarrow{m} \omega_Y.$$
Finally this allows us to rewrite the bottom row in the commutative diagram (\ref{eq:commute4}) as 
$$\D(\shG) \otimes \shG \xrightarrow{\D \delta} \D(\shG) \otimes \shG \xrightarrow{\tau} \D\D(\D(\shG) \otimes \shG) \xrightarrow{m} \omega_Y.$$
The result follows since the composition $\D(\shG) \otimes \shG \xrightarrow{\tau} \D\D(\D(\shG) \otimes \shG) \xrightarrow{m} \omega_Y$ is equivalent to $m$ and we can include the isomorphism $\D \delta$ into the left vertical isomorphism in (\ref{eq:commute4}). 
\end{proof}

Now we are ready to consider the general case of $\shA \cong \shF \tbox \D(\shG^*) \tbox \shG$. 

\begin{Proposition}\label{prop:1}
There exists a commutative diagram
\begin{equation*}%\label{eq:commute6}
\begin{gathered}\xymatrix{
j_2^* i_1^! (\shA) \ar[r]^\phi \ar[d]^\sim & j_1^! i_2^* (\shA) \ar[r] & j_1^!(\shF \tbox \omega_{Y^*}) \ar[d]^{\sim} \\
\cHom(\D(\shF),\D(\shG)) \otimes \shG \ar[r] & \cHom(\D(\shF),\D(\shG) \otimes \shG) \ar[r]^m  & \cHom(\D(\shF),\omega_Y) \cong \shF.
}
\end{gathered}\end{equation*}
\end{Proposition}
\begin{proof}
We will use the fact that $\shA \cong \cHom(\D(\shF) \tbox \O_{Y^*} \tbox \O_Y, \shB)$. By Proposition \ref{prop:phi} we have a commutative diagram 
\begin{equation}\label{eq:commute5}
\begin{gathered}\xymatrix{
j_2^* i_1^! (\shA) \ar[r]^\phi \ar[d] & j_1^! i_2^* (\shA) \ar[r] \ar[d] & j_1^!(\shF \tbox \omega_{Y^*}) \ar[d]^{\sim} \\
\cHom(\D(\shF), j_2^* i_1^!(\shB)) \ar[r]^\phi & \cHom(\D(\shF),  j_1^! i_2^*(\shB)) \ar[r] & \cHom(\D(\shF), j_1^!(\omega_Y \tbox \omega_{Y^*})) ,
}
\end{gathered}
\end{equation}
where we have used the fact that $j_2^* i_1^* (\D(\shF) \tbox \O_{Y^*} \tbox \O_Y) \cong \D(\shF)$. The left vertical map is the composition 
\begin{equation}\label{eq:27}
j_2^* i_1^! \cHom(\D(\shF) \tbox \O_{Y^*} \tbox \O_Y, \shB) \xrightarrow{\sim} j_2^* \cHom(\D(\shF) \tbox \O_Y, i_1^!(\shB)) \rightarrow \cHom(\D(\shF), j_2^* i_1^!(\shB)).
\end{equation}
Now we have $i_1^!(\shB) \cong \D(\shG) \tbox \shG$, $j_2^* i_1^!(\shB) \cong \D(\shG) \otimes \shG$ and 
$$\cHom(\D(\shF) \tbox \O_Y, i_1^!(\shB)) \cong \cHom(\D(\shF), \D(\shG)) \tbox \shG.$$ 
Subsequently, the second map in (\ref{eq:27}) can be identified with the natural morphism 
$$j_2^*(\cHom(\D(\shF), \D(\shG)) \tbox \shG) \cong \cHom(\D(\shF),\D(\shG)) \otimes \shG \rightarrow \cHom(\D(\shF), \D(\shG) \otimes \shG).$$
Combining this  with Corollary \ref{cor:1} the commutative diagram (\ref{eq:commute5}) gives us the commutative diagram in Proposition \ref{prop:1}.%(\ref{eq:commute6}).  
\end{proof}

Finally, applying $j_{2*}$ to the commutative diagram in Proposition \ref{prop:1} and comparing with the bottom two rows in (\ref{eq:commute3}) we arrive at the commutative diagram
\begin{equation}\label{eq:commute8}
\begin{gathered}\xymatrix{
i_1^!(\shA) \ar[r]^{adj} \ar[d] & i_1^! i_{2*} i_2^* (\shA) \ar[r] & i_1^! i_{2*} (\shF \tbox \omega_{Y^*}) \ar[d]^\sim \\
j_{2*} \cHom(\D(\shF),\D(\shG) \otimes \shG) \ar[rr]^m & & j_{2*} \cHom(\D(\shF),\omega_Y) \cong j_{2*} \shF,}
\end{gathered}
\end{equation}
where the vertical map on the left is the composition 
$$i_1^!(\shA) \xrightarrow{adj} j_{2*} j_2^* i_1^!(\shA) \xrightarrow{\sim} j_{2*} (\cHom(\D(\shF),\D(\shG)) \otimes \shG) \rightarrow j_{2*} \cHom(\D(\shF),\D(\shG)\otimes \shG).$$
This map can be factored as 
$$i_1^!(\shA) \xrightarrow{\sim} \cHom(\D(\shF),\D(\shG)) \tbox \shG \rightarrow j_{2*} \cHom(\D(\shF),\D(\shG) \otimes \shG).$$
Together with (\ref{eq:commute8}) this now recovers the diagram in Proposition \ref{prop:adjoints}. %(\ref{eq:commute7}). 

\subsection{Other half of the composition} \label{sec:rest}
We are now ready to identify more explicitly the composition in (\ref{eq:adj10}). We start with the following commutative diagram 
\begin{equation}\label{eq:adj7}
\begin{gathered}\xymatrix{
\O_Y \bbox \shF \ar[rr] \ar[d]^= & & i_1^! (\shF \tbox \D(\shF^*) \tbox \shF) \ar[d]^{\sim} \ar[r] & i_1^! i_{2*}(\shF \tbox \omega_{Y^*}) \ar[d]^{\sim} \\
\O_Y \bbox \shF \ar[rr]^{\iota} & & \cHom(\D(\shF),\D(\shF)) \tbox \shF \ar[r] & j_{2*} \shF
}
\end{gathered}
\end{equation}
The right rectangle is the large commutative rectangle from Proposition \ref{prop:adjoints} (we take $\shF = \shG$ in the notation of Section \ref{sec:hardhalf}). The composition in the bottom row is equivalent to the following composition 
\begin{align*}
\O_Y \bbox \shF \rightarrow j_{2*} \shF & \xrightarrow{\iota} j_{2*}(\cHom(\D(\shF),\D(\shF)) \otimes \shF) \\
& \rightarrow j_{2*} \cHom(\D(\shF),\D(\shF) \otimes \shF) \xrightarrow{m} j_{2*} \cHom(\D(\shF),\omega_Y) \cong j_{2*} \shF.
\end{align*}
This is subsequently equivalent to the natural restriction map $\O_Y \bbox \shF \rightarrow j_{2*} \shF$ because the composition 
$$\shF \xrightarrow{\iota} \cHom(\D(\shF),\D(\shF)) \otimes \shF \rightarrow \cHom(\D(\shF),\D(\shF) \otimes \shF) \xrightarrow{m} \cHom(\D(\shF),\omega_Y) \cong \shF$$
is an isomorphism. 

\subsection{Final step in the proof of rigidity}\label{sec:proofadj}

Pushing forward the map in (\ref{eq:adj10}) to $\Gr_G$ we obtain the composition in (\ref{eq:adj5}). The last step is showing that this composition is an isomorphism. 

The composition in the top row of (\ref{eq:adj7}) is by adjunction equivalent to the one in (\ref{eq:adj6}), so we get the following commutative diagram.
\begin{equation}\label{eq:adj8}
\begin{gathered}\xymatrix{
i_{1*} (\O_Y \bbox \shF) \ar[r] \ar[d]^= & i_{1*} i_1^! i_{2*}(\shF \tbox \omega_{Y^*}) \ar[r]^{adj} \ar[d]^\sim & i_{2*}(\shF \tbox \omega_{Y^*}) \\
i_{1*} (\O_Y \bbox \shF) \ar[r] & i_{1*} j_{2*} \shF & }
\end{gathered}
\end{equation}
Here the composition in the top row is the one from (\ref{eq:adj6}) and the map in the bottom row is restriction. Now by Lemma \ref{lem:1} below the composition 
$$i_{1*} i_1^! i_{2*} (\shF \tbox \omega_{Y^*}) \xrightarrow{\sim} i_{1*} j_{2*} j_1^! (\shF \tbox \omega_{Y^*}) = i_{2*} j_{1*} j_1^! (\shF \tbox \omega_{Y^*}) \xrightarrow{adj} i_{2*}(\shF \tbox \omega_{Y^*})$$
is equal to the adjunction map $i_{1*} i_1^! i_{2*} (\shF \tbox \omega_{Y^*}) \rightarrow i_{2*} (\shF \tbox \omega_{Y^*})$. Thus we can complete the diagram in (\ref{eq:adj8}) to get 
\begin{equation}\label{eq:adj9}
\begin{gathered}\xymatrix{
i_{1*} (\O_Y \bbox \shF) \ar[rrrr] \ar[d]^= & & & & i_{2*} (\shF \tbox \omega_{Y^*}) \ar[d]^= \\
i_{1*} (\O_Y \bbox \shF) \ar[r] & i_{1*}j_{2*} \shF \ar[rr]^\sim & & i_{2*} j_{1*} j_1^! (\shF \tbox \omega_{Y^*}) \ar[r]^{adj} & i_{2*} (\shF \tbox \omega_{Y^*}) 
}
\end{gathered}
\end{equation}
where the top row is the map from (\ref{eq:adj6}), while the middle isomorphism is the fact that $j_1^!(\shF \tbox \omega_{Y^*}) \cong \shF$ and $i_1 \circ j_2 = i_2 \circ j_1$.

Now we are ready to push down to $\Gr_G$. Let us denote $p: Y \times Y \rightarrow Y$ the projection onto the second factor and similarly $q: Y \ttimes Y^* \rightarrow Y$ the projection onto the first factor. Then the pushforward of $i_{1*} (\O_Y \bbox \shF)$ to $\Gr_G$ can be identified with $p_{*} (\O_Y \bbox \shF) \cong p_{*} p^* (\shF)$. Similarly, the pushforward of $i_{2*} (\shF \tbox \omega_{Y^*})$ can be identified with $q_*(\shF \tbox \omega_{Y^*}) \cong q_* q^! (\shF)$. Subsequently, using the commutative diagram (\ref{eq:adj9}), the composition 
$$\shF = \O_e * \shF \rightarrow \shF * \shF^L * \shF \rightarrow \shF * \O_e = \shF$$
can be identified with a composition
$$\shF \xrightarrow{adj} p_{*} p^*(\shF) \xrightarrow{adj} p_{*} j_{2*} j_2^* p^*(\shF) \xrightarrow{\sim} q_* j_{1*} j_1^! q^!(\shF) \xrightarrow{adj} q_* q^!(\shF) \xrightarrow{adj} \shF.$$
Notice that the composition of the first two adjunction maps is the adjunction map for $p \circ j_2$. Since $p \circ j_2$ is the identity this composition is an isomorphism. Likewise, since $q \circ j_1$ is the identity map, the composition of the last two adjunctions is also an isomorphism. It follows that the entire composition is an isomorphism. This completes the proof of Theorem \ref{thm:adjoints}.

\begin{Lemma}\label{lem:1}
Using the notation from (\ref{eq:ind}) the composition
$$i_{1*} \xrightarrow{adj} i_{1*} j_{2*} j_2^*  = i_{2*} j_{1*} j_2^* \xrightarrow{\sim} i_{2*} i_2^* i_{1*}$$
is equal to $i_{1*} \xrightarrow{adj} i_{2*} i_2^* i_{1*}$. Similarly, the composition 
$$i_{1*} i_1^! i_{2*} \xrightarrow{\sim} i_{1*} j_{2*} j_1^! = i_{2*} j_{1*} j_1^! \xrightarrow{adj} i_{2*}$$
is equal to $i_{1*} i_1^! i_{2*} \xrightarrow{adj} i_{2*}$. 
\end{Lemma}
\begin{proof}
We prove the first assertion (the second follows similarly). Since the last map in the composition is the inverse of $i_2^* i_{1*} \xrightarrow{\sim} j_{1*} j_2^*$ it suffices to show that the following diagram commutes.
\begin{equation}\label{eq:3}
\begin{gathered}\xymatrix{
i_{1*} \ar[r]^{adj} \ar[d]^{adj} & i_{1*} j_{2*} j_2^* \ar[d]^{=} \\
i_{2*} i_2^* i_{1*} \ar[r]^{\sim} & i_{2*} j_{1*} j_2^* 
}
\end{gathered}
\end{equation}
The bottom isomorphism can be decomposed as 
$$i_{2*} i_2^* i_{1*} \xrightarrow{adj} i_{2*} i_2^* i_{1*} (j_{2*} j_2^*) = i_{2*} (i_2^* i_{2*}) j_{1*} j_2^* \xrightarrow{adj} i_{2*} j_{1*} j_2^*.$$
Thus the composition 
$$i_{1*} \xrightarrow{adj} i_{2*} i_2^* i_{1*} \xrightarrow{\sim} i_{2*} j_{1*} j_2^* $$
is equal to 
$$i_{1*} \xrightarrow{adj} i_{1*} (j_{2*} j_2^*) = i_{2*} j_{1*} j_2^* \xrightarrow{adj} i_{2*} i_2^* i_{2*} j_{1*} j_2^* \xrightarrow{adj} i_{2*} j_{1*} j_2^*.$$
Since the composition $i_{2*} \xrightarrow{adj} i_{2*} i_2^* i_{2*} \xrightarrow{adj} i_{2*}$ is the identity this simplifies to give 
$$i_{1*} \xrightarrow{adj} i_{1*} (j_{2*} j_2^*) = i_{2*} j_{1*} j_2^* $$
and proves the commutativity of (\ref{eq:3}). 
\end{proof}

\section{Renormalized $r$-matrices}\label{sec:r-matrix}

In this section we develop the notion of a renormalized $r$-matrix in an arbitrary monoidal category $\cC$, abstracting the essential features of the $r$-matrices introduced in the setting of KLR algebras in \cite{KKK13}. Following \cite{KKKO18}, we recall how they provide a framework for the monoidal categorification of cluster algebras. 

\subsection{Definitions and examples}\label{sec:def-r-matrix}

\begin{Definition}\label{def:rmat}
Let $(\cC, \conv)$ be a $\coeffs$-linear monoidal category. A \newword{system of renormalized $r$-matrices} in $\cC$ is an assignment to each pair of objects $M,N \in \cC$ a map \[\rmat{M,N}: M \conv N \to N \conv M\] and an element $\La(M,N) \in \Z \cup \{-\infty\}$ satisfying the following properties. 
\begin{enumerate}[(i)]
\item\label{p1} For any $M \in \cC$ the morphisms $\rmat{M,1_\cC}$ and $\rmat{1_\cC,M}$ are given by composing the unit isomorphisms of $M \conv 1_\cC$ and $1_\cC \conv M$ with $M$. 
\item\label{p2} $\rmat{M,N}$ is nonzero if and only if $\La(M,N) \ne -\infty$.
\item\label{p3} For any $M,N_1,N_2 \in \cC$ we have 
$$\La(M, N_1 \conv N_2) \leq \La(M, N_1) + \La(M, N_2).$$
If equality holds then 
$$\rmat{M,N_1 \conv N_2} = (\id_{N_1} \conv \rmat{M,N_2}) \circ (\rmat{M,N_1} \conv \id_{N_2})$$
while the right-hand composition is zero if the inequality is strict. The corresponding statement with $M$ on the other side also holds. 

\item\label{p4} For $M,N \in \cC$ we have $\La(M,N) + \La(N,M) \ge 0$. Moreover
$$\rmat{N,M} \circ \rmat{M,N} \ne 0 \iff \La(M,N) + \La(N,M) = 0.$$
\item\label{p5} For any $M,N_1,N_2 \in \cC$ and morphism $f: N_1 \to N_2$ consider the diagram
\[
\begin{tikzpicture}
[thick,>=\arrtip]
\node (a) at (0,0) {$M \conv N_1$};
\node (b) at (4,0) {$N_1 \conv M$};
\node (c) at (0,-1.5) {$M \conv N_2$};
\node (d) at (4,-1.5) {$N_2 \conv M.$};
\draw[->] (a) to node[above] {$\rmat{M, N_1}$} (b);
\draw[->] (b) to node[right] {$f \conv \id_M$} (d);
\draw[->] (a) to node[left] {$\id_M \conv f$}(c);
\draw[->] (c) to node[above] {$\rmat{M, N_2}$} (d);
\end{tikzpicture}
\]
\begin{itemize}
\item If $\La(M,N_1) = \La(M,N_2)$ the diagram commutes.
\item If $\La(M,N_1) < \La(M,N_2)$ the bottom left composition is zero. 
\item If $\La(M,N_1) > \La(M,N_2)$ the top right composition is zero. 
\end{itemize}
The corresponding statements hold when the product with $M$ is taken on the other side. 
\end{enumerate}
\end{Definition}

As the definition suggests, renormalized $r$-matrices are rarely invertible in motivating examples and need not be natural in the categorical sense. The last property above says that their failure to be natural is controlled by the quantity $\La(M,N)$.

\begin{Example}
Let $\cC$ be a braided monoidal category. Then letting $\rmat{M, N}$ be the braiding isomorphism and setting $\La(M,N)=0$ for all $M$, $N$ we obtain a system of renormalized $r$-matrices in $\cC$. 
\end{Example}

\begin{Example}\label{ex:KLR}
Let $A$ be a symmetric generalized Cartan matrix and $Q_+$ the associated semigroup of positive roots. For each $\beta \in Q_+$, let $R(\beta)$ be the the symmetric KLR (or quiver Hecke) algebra associated to some choice of auxiliary polynomials. Let $R(\beta)\textrm{-mod}$ be the category of gradable finite-dimensional $R(\beta)$-modules. Then the direct sum $\oplus_{\beta \in Q_+} R(\beta)\textrm{-mod}$, equipped with its convolution product, admits a system of renormalized $r$-matrices constructed in \cite{KKK13}. We note that this category is not rigid.
\end{Example}

\begin{Example}\label{ex:chiralrmat}
Our main example in this paper comes by way of a monoidal category $\cC$ which extends in a suitable fashion to a chiral category over $\A^1$. The construction of renormalized $r$-matrices in this case is explained in Section \ref{sec:chiral} (c.f. Theorem \ref{thm:rmat}). It is in light of this construction that we find it more natural not to assume $\cC$ is abelian in the definition.
\end{Example}

\begin{Example}
The constructions of renormalized $r$-matrices in Examples \ref{ex:KLR} and \ref{ex:chiralrmat} parallel the construction of renormalized $r$-matrices for finite-dimensional representations of a quantum loop algebra \cite{Kas02}. We expect the latter fit into the present axiomatic setup or a close relative thereof. Indeed, following \cite{KKKO15a} they share many of the key properties described in the following section. Moreover, the formalism of meromorphic tensor categories considered in \cite{Soi97,Soi99}, of which representations of quantum loop algebras provide central motivating examples, shares many of the key features of the chiral formalism considered in Section \ref{sec:chiral} and should allow a similar general construction.
\end{Example}

Having a system of renormalized $r$-matrices becomes a stronger condition given additional hypotheses on the category $\cC$. 

\begin{Lemma}\label{lem:rmat}
If $M$ or $N$ has a left or right dual then $\rmat{M,N} \ne 0$ unless $M*N=0$ or $N*M=0$. 
\end{Lemma}
\begin{proof}
Suppose $N$ has a right dual $N^R$ (the other cases are the similar). Consider the following diagram.
\[
\begin{tikzpicture}
[thick,>=\arrtip]
\node (a) at (0,0) {$M*(N*N^R)$};
\node (b) at (6,0) {$(N*N^R)*M$};
\node (c) at (0,-1.5) {$M*1_{\cC}$};
\node (d) at (6,-1.5) {$1_{\cC}*M.$};
\draw[->] (a) to node[above] {$\rmat{M, N*N^R}$} (b);
\draw[->] (b) to node[right] {$adj * \id_M$} (d);
\draw[->] (a) to node[left] {$\id_M * adj$}(c);
\draw[->] (c) to node[above] {$\rmat{M,1_{\cC}}$} (d);
\end{tikzpicture}
\]
The composition along the bottom left is just the adjunction map $M * N * N^R \rightarrow M$ which is nonzero (unless $M*N=0$). By Definition \ref{def:rmat}(\ref{p5}) this implies that 
$$\La(M,N*N^R) \ge \La(M,1_{\cC}).$$
It follows from (\ref{p1}) and (\ref{p2}) that $\La(M,1_{\cC}) \neq -\infty$ ((\ref{p3}) then implying $\La(M,1_{\cC}) = 0$), hence $\La(M,N*N^R) \neq -\infty$ as well. 
On the other hand, by (\ref{p3}) we have $\La(M,N)+\La(M,N^R) \ge \La(M,N*N^R)$. Thus $\La(M,N) \ne -\infty$ and the result follows by (\ref{p2}). 
\end{proof}

Since in a rigid abelian monoidal category products of nonzero objects are nonzero, we obtain the following corollary.

\begin{Corollary}\label{cor:rmat}
If $\cC$ is rigid and abelian then $\rmat{M,N} \ne 0$ for all nonzero $M$, $N$.
\end{Corollary}

\subsection{Consequences}\label{sec:rmatimplies}

When $\cC$ is abelian and satisfies some natural conditions, the existence of renormalized $r$-matrices strongly constrains the underlying monoidal structure. In this section we discuss some key consequences. The results here and in the remainder of Section~\ref{sec:r-matrix} are straightforward generalizations from the KLR case treated in \cite{KKKO15a,KKKO18}, hence we mostly refer to their antecedents in loc. cited for detailed proofs. 

We assume now that $\cC$ is a locally finite abelian category whose monoidal product is biexact, and such that that $\End(1_\cC) \cong \coeffs \cdot id$. In other words, $\cC$ is a $\coeffs$-linear ring category in the terminology of \cite{EGNO15}. We further assume that $\cC$ has a system of renormalized $r$-matrices for which $\rmat{M,N}$ is nonzero for all nonzero $M$, $N$, and that $\cC$ has the following property considered in \cite{KKKO15a}.

\begin{Definition}\label{prop:midfactor}
We say $\cC$ has \newword{separable triple products} if the following condition holds. Suppose $M_1$, $M_2$, and $M_3$ are objects with $X \subset M_1 \conv M_2$, $Y \subset M_2 \conv M_3$ such that $X \conv M_3 \subset M_1 \conv Y$. Then there exists a subobject $N \subset M_2$ such that
$X \subset M_1 \conv N$ and $N \conv M_3 \subset Y.$
%$$X \subset M_1 \conv N, \quad N \conv M_3 \subset Y.$$
Likewise if $M_1 \conv Y \subset X \conv M_3$ then there is an $N \subset M_2$ such that
$Y \subset N \conv M_3$ and $M_1 \conv N \subset X.$
%$$Y \subset N \conv M_3, \quad M_1 \conv N \subset X.$$
\end{Definition}

By Corollary \ref{cor:rmat}, the nonvanishing of $\rmat{M,N}$ is guaranteed if $\cC$ is rigid. This is also the case for the condition of having separable triple products.

\begin{Proposition}[c.f. Lemma 3.10 \cite{KKKO15a}]\label{prop:rigid-triple}
If $\cC$ is rigid then it has separable triple products.
\end{Proposition}
\begin{comment}
\begin{proof}
Let $M_1$, $M_2$, $M_3$, $X$, and $Y$ be as in Definition \ref{prop:midfactor}. Take $N$ to be the image of $f: M_1^L * X \rightarrow M_2$ where $f$ is the adjoint of the inclusion map $i: X \hookrightarrow M_1 * M_2$. Since $i$ can be factored as $X \xrightarrow{adj} M_1 * M_1^L * X \xrightarrow{\id * f} M_1 * M_2$ it immediately follows that $X \hookrightarrow M_1 * N$.

On the other hand, the composition
$$M_1^L * X * M_3 \xrightarrow{f * \id} N * M_3 \hookrightarrow M_2 * M_3$$
can also be factored as
$$M_1^L * X * M_3 \xrightarrow{\id * i * \id} M_1^L * M_1 * M_2 * M_3 \xrightarrow{adj * \id * \id} M_2 * M_3.$$
Using the condition that $X * M_3 \hookrightarrow M_1 * Y$ the image of this composition lands in $Y$. Thus $N * M_3 \hookrightarrow Y$.
\end{proof}
\end{comment}

\begin{Lemma}[c.f. Lemma 3.1.5 \cite{KKKO18}]\label{lem:p3}
If $M$ is a nonzero simple object in $\cC$ then for any nonzero $N_1,N_2$ we have
$$\La(M, N_1 \conv N_2) = \La(M, N_1) + \La(M, N_2), \quad \La(N_1 \conv N_2, M) = \La(N_1, M) + \La(N_2, M).$$
\end{Lemma}
\begin{proof}
We show the first equality (the second follows similarly). In light of condition (\ref{p3}) it suffices to show that $(\id_{N_1} \conv \rmat{M,N_2}) \circ (\rmat{M,N_1} \conv \id_{N_2})$ is nonzero. Otherwise ${\im \; \rmat{M,N_1} \conv N_2} \subset N_1 \conv \ker \rmat{M, N_2}$. Since $\cC$ has separable triple products there then exists $L \subset M$ with $\im \; \rmat{M,N_1} \subset N_1 \conv L$ and $L \conv N_2 \subset \ker \rmat{M, N_2}$. Since $\rmat{M,N_1}$ and $\rmat{M, N_2}$ are nonzero $L$ must be proper and nonzero, contradicting the simplicity of $M$.
\end{proof}

\begin{Lemma}\label{lem:subquotientsandLambda}
Let $L$ be a subquotient of $N$. Then for any $M$ we have
$$\Lambda(M, L) \leq \Lambda(M, N), \quad \Lambda(L, M) \leq \Lambda(N, M).$$
\end{Lemma}
\begin{proof}
It suffices to assume that $L$ is either a subobject or a quotient of $N$. In either case one of the compositions in condition (\ref{p5}) must be nonzero, which forces the stated inequality.
\end{proof}

\begin{Proposition}[c.f. Theorem 3.2 \cite{KKKO15a}]\label{prop:KKKO3.2}
Let $M$ and $N$ be nonzero objects such that $\rmat{M,M}$ is a multiple of the identity and $N$ is simple. Then
\begin{enumerate}[(i)]
\item $\im \; \rmat{M,N}$ is simple, is the socle of $N \conv M$, and is the head of $M \conv N$.
\item $\im \; \rmat{N,M}$ is simple, is the socle of $M \conv N$, and is the head of $N \conv M$.
\item $M$ is simple.
\end{enumerate}
\end{Proposition}
\begin{proof}
Let $S \subset M \conv N$ be a simple nonzero subobject, and consider the diagram
\[
\begin{tikzpicture}
[thick,>=\arrtip]
\node (a) at (0,0) {$S \conv M$};
\node (b) at (4.5,0) {$M \conv S$};
\node (c) at (0,-1.5) {$M \conv N \conv M$};
\node (d) at (4.5,-1.5) {$M \conv M \conv N$};
\draw[->] (a) to node[above] {$\rmat{S,M}$} (b);
\draw[right hook->] (b) to (d);
\draw[right hook->] (a) to (c);
\draw[->] (c) to node[above] {$\id_M \conv \rmat{N,M}$} (d);
\end{tikzpicture}
\]
Since $\rmat{M,M}$ is a multiple of the identity, it follows from Definition \ref{def:rmat}(\ref{p3}) that the bottom map is a scalar multiple of $\rmat{M \conv N, M}$.

By Definition \ref{def:rmat}(\ref{p5}) either the diagram commutes or the composition around the bottom-left of the square vanishes. The latter possibility implies
$S \conv M \subset M \conv \ker \rmat{N, M}.$
Since $\cC$ has separable triple products there exists a subobject $K \subset N$ such that
$
S \subset M \conv K$ and $ K \conv M \subset \ker \rmat{N,M}.
$
Since $S$ and $\rmat{N,M}$ are nonzero $K$ is then proper and nonzero. As this contradicts the simplicity of $N$, we conclude that the diagram commutes.

This in turn implies
$S \conv M \subset M \conv \rmat{N,M}^{-1}(S).$
As before this implies that there exists a subobject $K \subset N$ such that
$S \subset M \conv K$ and $K \conv M \subset \rmat{N,M}^{-1}(S).$
Again the first condition implies $K$ is nonzero, so by simplicity of $N$ we have $K = N$. But now observe that the condition $N \conv M \subset \rmat{N,M}^{-1}(S)$ is equivalent to ${\im \;\rmat{N,M}} \subset S$. Since $\rmat{N,M}$ is nonzero and $S$ simple it follows that ${\im \;\rmat{N,M}} = S$, and since $S$ was an arbitrary simple subobject we conclude that ${\im \;\rmat{N,M}}$ is the socle of $M \conv N$.

The remaining properties of $\im \;\rmat{M,N}$ and $\im \;\rmat{N,M}$ can be proved similarly. The simplicity of $M$ follows from the case $N = 1_\cC$ and Definition \ref{def:rmat}(\ref{p1}).
\end{proof}

%Recall again that an object $M$ of a monoidal abelian category is real if its product with itself is simple. 

\begin{Corollary}[c.f. Corollary 3.3 \cite{KKKO15a}]\label{cor:realiff}
An object $M$ is real if and only if $\rmat{M,M}$ is a multiple of the identity. 
\end{Corollary}
%\begin{proof}
%If $M$ is real then since $\cC$ is locally finite and $\coeffs$ algebraically closed, $\rmat{M,M}$ must be a multiple of the identity since it is nonzero. The other direction follows since $M \conv M = \im \; \rmat{M,M}$ is simple by the Proposition. 
%\end{proof}

We say $M$ and $N$ \newword{commute} if $M \conv N$ and $N \conv M$ are isomorphic. When either object is real, the following criteria follow immediately from the preceding results and Definition \ref{def:rmat}.

\begin{Proposition}[c.f. Lemma 3.2.3 \cite{KKKO18}]\label{prop:KKKO2.11}
Let $M$ and $N$ be simple objects and suppose either $M$ or $N$ is real. Then the following conditions are equivalent:
\begin{enumerate}[(i)]
\item $M$ and $N$ commute.
\item $\rmat{M,N} \circ \rmat{N,M}$ is a multiple of the identity.
\item $M \conv N$ or $N \conv M$ is simple.
\item $\La(M,N) + \La(N,M) = 0$.
\end{enumerate}
\end{Proposition}

\begin{Corollary}[c.f. Proposition 3.2.5 \cite{KKKO18}]\label{cor:KKKO2.13}
Let $M_1, \dotsc, M_n$ be pairwise commuting real simple objects. Then $ M_1 \conv \cdots \conv M_n $ is also real and simple.
\end{Corollary}
%\begin{proof}
%To show the given product is simple, induct on $n$. Assuming the simplicity of $M_1 \conv \cdots \conv M_{n-1}$, we have
%\[
%\La(M_1 \conv \cdots \conv M_{n-1}, M_n) + \La(M_n, M_1 \conv \cdots \conv M_{n-1}) = \sum_{i=1}^{n-1} \La(M_i, M_n) + \La(M_n, M_i) = 0
%\]
%by Definition \ref{def:rmat}(\ref{p3}). By Proposition \ref{prop:KKKO2.11} $M_1 \conv \cdots \conv M_n$ is then simple. The reality of $M_1 \conv \cdots \conv M_n$ follows from applying this argument to the square of $M_1 \conv \cdots \conv M_n$.
%\end{proof}

The following result gives a criterion for recognizing real objects.

\begin{Proposition}[c.f. Proposition 3.2.20 \cite{KKKO18}]\label{prop:KKKO2.21}
Suppose there exists an exact sequence
$$0 \to X \to M \conv N \to Y \to 0$$
with $X$, $Y$, $M$, and $N$ simple objects such that $X \conv N$ and $Y \conv N$ are simple and nonisomorphic. Then $N$ is real.
\end{Proposition}
\begin{comment}
\begin{proof}
Suppose there exists a proper nonzero subobject $S \subset N \conv N$. If the intersection $(X \conv N) \cap (M \conv S)$ in $M \conv N \conv N$ is nonzero then since $X \conv N$ is simple we must have $X \conv N \subset M \conv S$. By separability of triple products there exists a subobject $Z \subset N$ such that $X \subset M \conv Z$ and $Z \conv N \subset S$. Since $S$ is proper and $X$ is nonzero, $Z$ is both proper and nonzero. This contradicts the simplicity of $N$, hence $(X \conv N) \cap (M \conv S) = 0$. Since the outer terms in 
\begin{equation*}
0 \to X \conv N \to M \conv N \conv N \to Y \conv N \to 0
\end{equation*}
are simple it now follows that $ M \conv S \cong Y \conv N $. 

Since $M \conv N \conv N$ is of length 2, $N \conv N$ must be as well. Thus $S' := (N \conv N)/S$ is simple, and necessarily $M \conv S' \cong X \conv N$. Note we cannot have a splitting $N \conv N \cong S \oplus S'$; otherwise by the above argument $X \conv N \cong M \conv S' \cong Y \conv N$, a contradiction.

Let $\lambda$ be an eigenvalue of the action of $\rmat{N,N}$ on $\End_\cC(N \conv N)$. Since $\rmat{N,N}$ is not a multiple of $\id_{N\conv N}$, $\im (\rmat{N,N} - \lambda \id_{N \conv N})$ is both a proper nonzero subobject and a proper nonzero quotient of $N \conv N$. In particular
$ S \cong \im (\rmat{N,N} - \lambda \id_{N \conv N}) \cong S'.$
But then we obtain an exact sequence
$$ 0 \to M \conv S \to M \conv N \conv N \to M \conv S \to 0, $$
hence the contradiction $X \conv N \cong M \conv S \cong Y \conv N$.
\end{proof}
\end{comment}

A useful consequence of an object $M$ being real is that we can often control the integers $\Lambda(M, L)$ through the following two results. Here we follow \cite{KKOP17} in writing $M \hconv N$ and $M \sconv N$ for the head and socle of $M \conv N$, respectively. 

\begin{Proposition}[c.f. Proposition 3.2.11 \cite{KKKO18}]\label{prop:KKKO2.20}
Let $L$, $M$, and $N$ be simple objects, and suppose that $L$ is real and commutes with $M$. Then
\begin{gather*}
\La(L, M \hconv N) = \La(L, M \conv N), \quad
\La(N \hconv M, L) = \La(N \conv M, L),
\end{gather*}
and the following diagrams commute:
\[
\begin{tikzpicture}
[thick,>=\arrtip]
\node (a) at (0,0) {$L \conv M \conv N$};
\node (b) at (4,0) {$M \conv N \conv L$};
\node (c) at (0,-1.5) {$L \conv (M \hconv N)$};
\node (d) at (4,-1.5) {$(M \hconv N) \conv L$};
\draw[->] (a) to node[above] {$\rmat{L,M \conv N}$} (b);
\draw[->>] (b) to (d);
\draw[->>] (a) to (c);
\draw[->] (c) to node[above] {$\rmat{L,M \smallhconv N}$} (d);

\node (ar) at (7,0) {$N \conv M \conv L$};
\node (br) at (11,0) {$L \conv N \conv M$};
\node (cr) at (7,-1.5) {$(N \hconv M) \conv L$};
\node (dr) at (11,-1.5) {$L \conv (N \hconv M)$};
\draw[->] (ar) to node[above] {$\rmat{N \conv M, L}$} (br);
\draw[->>] (br) to (dr);
\draw[->>] (ar) to (cr);
\draw[->] (cr) to node[above] {$\rmat{N \smallhconv M, L}$} (dr);
\end{tikzpicture}
\]
\end{Proposition}
\begin{comment}
\begin{proof}
By Definition \ref{def:rmat}(\ref{p5}) and Lemma \ref{lem:p3}, the two equalities follow from the commutativity of the diagrams. 
We consider the left diagram, the right being similar. The diagram fails to commute if and only if composition
\[
\begin{tikzpicture}
[thick,>=\arrtip]
\node (a) at (0,0) {$L \conv M \conv N$};
\node (b) at (4,0) {$M \conv L \conv N$};
\node (c) at (8,0) {$M \conv N \conv L$};
\node (d) at (12,0) {$(M \hconv N) \conv L$};
\draw[->] (a) to node[below] {$\rmat{L,M} \conv \id_N$} node[above] {$\sim$} (b);
\draw[->] (b) to node[below] {$\id_M \conv \rmat{L, N}$}(c);
\draw[->>] (c) to (d);
\end{tikzpicture}
\]
vanishes. The latter would imply that
$
M \conv (\im \; \rmat{L, N}) \subset \ker \rmat{M, N} \conv L.
$
Thus by separability of triple products there exists a subobject $K \subset N$ such that
$
\im \; \rmat{L, N} \subset K \conv L$ and $M \conv K \subset \ker \rmat{M, N}. 
$
But this implies $K$ is proper and nonzero, contradicting the simplicity of $N$.
\end{proof}
\end{comment}

Setting $L = M$ above, we obtain
$$\Lambda(M,M \hconv N) = \Lambda(M, M \conv N) = \Lambda(M, N).$$
In fact, this characterizes $M \hconv N$ among simple subquotients of $M \conv N$.

\begin{Proposition}[c.f. Theorem 4.1.1 \cite{KKKO18}]\label{prop:KKKO3.1}
Let $M$, $N$ be simple and $M$ real. If $L$ is a simple subquotient of $M \conv N$ we have $\Lambda(M, L) < \Lambda(M, N)$ unless $L$ is its head and $\Lambda(L, M) < \Lambda(N, M)$ unless $L$ is its socle. Likewise, if $L$ is a simple subquotient of $N \conv M$ we have $\Lambda(M, L) < \Lambda(M, N)$ unless $L$ is its socle and $\Lambda(L, M) < \Lambda(N, M)$ unless $L$ is its head. Thus $M \hconv N$ and $M \sconv N$ each appear exactly once as a factor in any composition series of $M \conv N$ or $N \conv M$.
\end{Proposition}
\begin{proof}
We prove the first claim, the remaining cases being similar. Let $K := \ker \rmat{M,N}$. By Lemma \ref{lem:subquotientsandLambda} it suffices to show that
$\Lambda(M, K) < \Lambda(M,N),$
since a simple subquotient of $M \conv N$ other than its head is necessarily a subquotient of $K$. Moreover, since the composition
\[
\begin{tikzpicture}
[thick,>=\arrtip]
\node (a) at (0,0) {$M \conv K$};
\node (b) at (4,0) {$K \conv M$};
\node (c) at (8,0) {$M \conv N \conv M$};
\draw[->] (a) to node[below] {$\rmat{M,K}$} (b);
\draw[right hook->] (b) to (c);
\end{tikzpicture}
\]
is nonzero, it follows from Definition \ref{def:rmat}(\ref{p5}) that $\Lambda(M, K) < \Lambda(M, N)$ if and only if $M \conv K$ is in the kernel of $\rmat{M, M\conv N}$.

By Definition \ref{def:rmat}(\ref{p3}) we have a commutative diagram
\[
\begin{tikzpicture}
[thick,>=\arrtip,box/.style={text width=3.7cm, align=center}]
\node (a) at (0,0) {$M \conv M \conv N$};
\node (b) at (3.5,-1) {$M \conv M \conv N$};
\node (c) at (7,0) {$M \conv N \conv N$};
\draw[->] (a) to node[below left] {$\rmat{M,M} \conv \id_N$} (b.west);
\draw[->] (b.east) to node[below right]  {$\id_M \conv \rmat{M,N}$} (c);
\draw[->] (a) to node[above] {$\rmat{M,M \conv N}$} (c);
\end{tikzpicture}
\]
Since $M$ is real the bottom left map is a multiple of the identity. In particular, $\rmat{M, M \conv N}$ and $\id_M \conv \rmat{M,N}$ coincide up to a nonzero scalar, hence $\rmat{M, M \conv N}$ annihilates $M \conv K$.
\end{proof}

Finally, we will need the following fact, where we write $M^n$ for the $n$th monoidal power of an object $M$.

\begin{Lemma}[c.f. Lemma 3.2.22 \cite{KKKO18}]\label{lem:KKKO2.23}
Let $M$ and $N$ be real simple objects such that $M \hconv N$ is real and commutes with $N$. Then for any $n \in \Z_{\geq 0}$ we have
$
M^{n} \hconv N^{n} \cong (M \hconv N)^{n}.
$
\end{Lemma}
\begin{comment}
\begin{proof}
Let $L := M \hconv N$. Since $L^{n}$ is simple it suffices to give an epimorphism $M^n \conv N^n \onto L^n$, which is done by induction on $n$. The $n = 1$ case is immediate, and for $n > 1$ we have
\begin{equation*}
M^{n} \conv N^{n} \onto M^{n-1} \conv L \conv N^{n-1}\cong M^{n-1} \conv N^{n-1} \conv L  \onto L^{n-1} \conv L.\qedhere
\end{equation*}
\end{proof}
\end{comment}

\subsection{Graded realizations}\label{sec:gradedrealization}

A category $\grl{\cC}$ is $\Z$-graded if it is endowed with an autoequivalence, called the grading shift, which we denote by $\{\frac12\}$. If $\grl{\cC}$ is $\Z$-graded we can forget the grading to obtain a category $\cC$ whose objects are the same but where
$$\Hom_\cC(M, N) = \bigoplus_{k \in \Z} \Hom_{\grl{\cC}}\,( \grl{M}, \grl{N}\{\frac{k}{2}\}).$$
Here we write $\grl{M}$ for an object in $\grl{\cC}$ and $M$ for its image in $\cC$. We say that $\grl{\cC}$ is a \newword{graded realization} of $\cC$.

If $\grl{\cC}$ is monoidal we require that $\{\frac12\}$ is compatible with the monoidal structure in the sense that, for any $\grl{M},\grl{N} \in \grl{\cC}$, we have isomorphisms
$$(\grl{M} \{\frac12\}) \conv \grl{N} \cong (\grl{M} \conv \grl{N})\{\frac12\} \cong \grl{M} \conv ( \grl{N} \{\frac12\}).$$
Note that if $\cC$ has a system of renormalized $r$-matrices then
$$\rmat{M,N} \in \Hom_\cC(M \conv N, N \conv M) = \bigoplus_{k \in \Z} \Hom_{\grl{\cC}}\,( \grl{M} \conv \grl{N}, \grl{N} \conv \grl{M} \{\frac{k}{2}\}).$$

\begin{Definition}
A system of renormalized $r$-matrices in $\cC$ is \newword{compatible with the graded realization $\grl{\cC}$} if $\rmat{M,N}$ has pure degree $\Lambda(M,N)$ for all $M,N \in \cC$. This means that for all $\grl{M},\grl{N} \in \grl{\cC}$ we have a morphism
$$\rmat{\grl{M}, \grl{N}}: \grl{M} \conv \grl{N} \to \grl{N} \conv \grl{M} \{\La(M,N)\}.$$
\end{Definition}

\begin{Example}
Recall from Example \ref{ex:KLR} the monoidal category $\oplus_{\beta \in Q_+} R(\beta)\textrm{-mod}$, where $R(\beta)\textrm{-mod}$ is the category of gradable modules over the symmetric KLR algebra $R(\beta)$ of the root $\beta$. By definition, this has a graded realization given by the direct sum $\oplus_{\beta \in Q_+} R(\beta)\textrm{-gmod}$ of the associated categories of graded modules. This realization is compatible with the renormalized $r$-matrices of \cite{KKK13} in the above sense.
\end{Example}

\begin{Example}
Suppose $\cC$ inherits a system of renormalized $r$-matrices from a compatible $\G_m$-equivariant chiral category $\grl{\cCh}$ on $\A^1$ via the construction of Section \ref{sec:chiralrmat}. Then these $r$-matrices are compatible with the natural graded realization $\grl{\cC}$ contained in the homotopy category of the fiber of $\grl{\cCh}$ at $\{0\}/\Gm$ (c.f. Proposition \ref{prop:Gmequiv}). 
\end{Example}

The shift $\{\frac12\}$ is equivalent to the functor $M \mapsto M \conv 1_{\grl{\cC}}\{\frac12\}$. Following the notation used in Section \ref{sec:cohsatdefs} we denote by $q^{-1/2} \in K_0(\grl{\cC})$ the class of $1_{\grl{\cC}}\{\frac12\}$ (so that multiplication by $q^{-1/2}$ corresponds to shifting by $\{\frac12\}$). Thus $K_0(\grl{\cC})$ inherits the structure of a $\Z[q^{\pm 1/2}]$-algebra. We say two objects $\grl{M}$, $\grl{N} \in \grl{\cC}$ \newword{$q$-commute} if
$\grl{M} \conv \grl{N}$ and $\grl{N} \conv \grl{M}$ are isomorphic up to a grading shift. In particular, their images in $\cC$ commute. 

Clearly in any $\Z$-graded monoidal category if $\grl{M}$ and $\grl{N}$ $q$-commute then their classes in the Grothendieck ring $q$-commute. By the following result, which is an application of Proposition \ref{prop:KKKO3.1}, the existence of a compatible system of renormalized $r$-matrices guarantees that the converse also holds when the objects are simple and one of them is real. Note that in this case for $q$-commuting simples $\grl{M}$, $\grl{N}$ we must have $\grl{M} \conv \grl{N} \cong \grl{N} \conv \grl{M} \{ \Lambda(M,N)\}.$

In the statement below we assume $\cC$ is as in the previous section and let $\{ \grl{S}_b \}_{b \in \cB}$ be a set of representatives up to grading shifts of all isomorphism classes of simple objects in $\grl{\cC}$. As we have assumed $\cC$, hence $\grl{\cC}$, to be of finite length, the classes $\{ [\grl{S}_b] \}_{b \in \cB}$ form a $\Z[q^{\pm 1/2}]$-basis of $K_0(\grl{\cC})$. 

\begin{Theorem}[c.f. Theorem 4.1.3 \cite{KKKO18}]\label{thm:KKKO3.3}
Consider an element 
$$ \Phi = \sum_{b \in \cB} a_b [\grl{S}_b] \in \Q(q^{1/2}) \otimes_{\Z[q^{\pm 1/2}]}K_0(\grl{\cC}),$$
and suppose that $\grl{M}$ is a real simple object of $\grl{\cC}$ such that $\Phi [\grl{M}] = q^n [\grl{M}] \Phi$ for some $n \in \frac{1}{2}\Z$. Then for any $b \in \cB$ with $a_b \neq 0$ the objects $\grl{M}$ and $\grl{S}_b$ $q$-commute with $\La(M, S_b) = n$.
\end{Theorem}

We will also need the following graded version of Lemma \ref{lem:KKKO2.23}.

\begin{Lemma}\label{lem:gradedKKKO2.23}
Let $\grl{M}$ and $\grl{N}$ be real simple objects such that $\grl{M} \hconv \grl{N}$ is real and $q$-commutes with $\grl{N}$. Then for any $n \in \Z_{\geq 0}$ we have
\[
\grl{M}^{n} \hconv \grl{N}^{n} \cong (\grl{M} \hconv \grl{N})^{n}\{\frac12 n(n-1)\La(M,N)\}.
\]
\end{Lemma}
\begin{proof}
Follows from the proof of Lemma \ref{lem:KKKO2.23}, keeping track of how many times the $q$-commutation relation of $(\grl{M} \hconv \grl{N})$ and $\grl{N}$ is invoked.
\end{proof}

\subsection{Mutation via $r$-matrices}\label{sec:monoidalseeds}

We are now ready to discuss the relationship between renormalized $r$-matrices and monoidal categorifications of cluster algebras, following \cite{KKKO18}. We first recall some basic notions from \cite{FZ02} and \cite{BZ05}.

Let $I = I_{ex} \sqcup I_{fr}$ be a finite index set divided into two subsets called exchangeable and frozen. Given an $I \times I_{ex}$ integer matrix $\wt{B}$ we refer to its $I_{ex} \times I_{ex}$ submatrix $B$ as its \newword{principal part}. We say $\wt{B} = (b_{ij})$ is an \newword{exchange matrix} if $B$ is skew-symmetrizable (that is, $DB$ is skew-symmetric for some diagonal matrix $D$ with positive diagonal entries). 
A skew-symmetric $I \times I$ matrix $L = (\lmat_{ij})$ is \newword{compatible} with $\wt{B}$ if $L \wt{B}$ is zero except along the diagonal of its principal part, where it has positive integer entries. We refer to $(L, \wt{B})$ as a \newword{compatible pair}.

The \newword{quantum torus} $\cT_q(L)$ attached to $L$ is the $\Z[q^{\pm 1/2}]$-algebra generated by elements $\{ X^v \}_{v \in \Z^I}$ with multiplication
  \[
X^u X^v = q^{\frac12 u^T L v} X^{u+v}.
\]
Thus, for example, $X^u X^v = q^{u^T L v} X^v X^u$. 
We write $\cF_q(L)$ for the division ring of fractions of $\cT_q(L)$, and $\cR_q(L) \subset \cT_q(L)$ for the quantum affine space spanned as a $\Z[q^{\pm 1/2}]$-module by $\{ X^v \}_{v \in \N^I}$. 

A \newword{quantum seed} $S = (\{x_i\}_{i \in I}, L, \wt{B})$ in a division ring $\cF$ is the data of
\begin{enumerate}
\item a compatible pair $(L, \wt{B})$,
\item nonzero elements $x_i \in \cF$ such that $x_i x_j = q^{\lmat_{ij}}x_j x_i$ and such that the induced map $M_S: \cF_q(L) \to \cF$ is injective.
\end{enumerate}
The set $\{x_i\}_{i \in I}$ is a \newword{quantum cluster} and its elements \newword{quantum cluster variables}. We refer to an element of $\cF$ of the form $M_S(X^u)$ for $u \in \N^I \subset \Z^I$ as a \newword{quantum cluster monomial}.

Given $k \in I_{ex}$, the \newword{mutation of $S$ in direction $k$} is the quantum seed 
$$\mu_k(S) = (\{\mu_k(x_i)\}_{i \in I}, \mu_k(L), \mu_k(\wt{B}))$$ 
defined as follows. Consider auxiliary matrices $E = (e_{ij})$, $F = (f_{ij})$ given by
\[
e_{ij} = \begin{cases} \delta_{ij} & j \neq k \\ -1 & i = k = j \\ \max(0, -b_{ik}) & i \neq k = j, \end{cases} \quad
f_{ij} = \begin{cases} \delta_{ij} & j \neq k \\ -1 & i = k = j \\ \max(0, b_{kj}) & i = k \neq j. \end{cases}
\]
With these in hand we set
$$\mu_k(L) := E^T L E, \quad \mu_k(\wt{B}) := E \wt{B} F.$$
These have the property that $\mu_k(L) \mu_k(\wt{B}) = L \wt{B}$. We then set
\[
\mu_k(x_i) = \begin{cases} M_S(X^{v_+}) + M_S(X^{v_-}) & i = k \\ x_i & i \neq k, \end{cases}
\]
where 
\[
v_+ = -e_k + \sum_{b_{ik} > 0} b_{ik} e_i, \quad v_- = -e_k - \sum_{b_{ik} < 0} b_{ik} e_i.
\]

The pair $(L, \wt{B})$ defines a canonical quantum seed in $\cF_q(L)$ with $x_i = X^{e_i}$. The \newword{quantum cluster algebra} $\cA_{(L, \wt{B})}$ is the $\Z[q^{\pm 1/2}]$-subalgebra of $\cF_q(L) \subset \cF$ generated by all quantum cluster variables of all seeds obtained by iterated mutation from this initial one. By the quantum cluster variables of $\cA_{(L, \wt{B})}$ we mean the set of quantum cluster variables from all such seeds, likewise for quantum cluster monomials. We also write $\cA^{loc}_{(L, \wt{B})}$ for the localization of $\cA_{(L, \wt{B})}$ at its frozen variables. The \newword{cluster algebra} $\cA_{\wt{B}}$ is the commutative ring is defined analogously but with all notions replaced by their specializations at $q=1$.

\begin{Example}\label{ex:n=2mutseq}
Let $I_{ex} = \{1,2\}$ and $I_{fr} = \{3,4\}$, and consider the following exchange matrix $\wt{B}$ on the left together with its mutation in direction 1.
\[
\begin{tikzpicture}[thick,>=\arrtip]
\node (original) at (0,0) {$\begin{bmatrix}
  0 & 2 \\ -2 & 0 \\ 1 & 0 \\ 0 & -1
  \end{bmatrix}$};
\node (mutant) at (4,0) {$\begin{bmatrix}
  0 & -2 \\ 2 & 0 \\ -1 & 2 \\ 0 & -1
  \end{bmatrix}$};
\draw[-implies,line width=1pt,double distance=4pt] (1.5,0) -- (2.5,0) node[midway,above=1.2mm] {$\mu_1$};
\end{tikzpicture}
\]
When its principal part is skew-symmetric, it is convenient to encode an exchange matrix as a quiver with $Q_0 = I$, no self-loops or oriented 2-cycles, no arrows between frozen vertices, and
$$b_{ij} = \#\{ \text{arrows } j \to i\} - \#\{ \text{arrows } i \to j\}.$$
In terms of quivers, the above mutation has the following form.
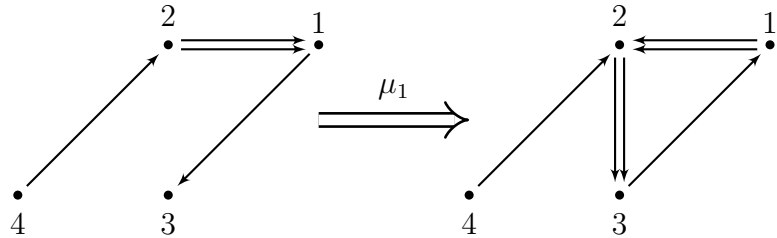
\[
\begin{tikzpicture}[thick,>=\arrtip]
\newcommand*{\Ddotsdist}{2}
\newcommand*{\shft}{1}
\newcommand*{\DrawDots}[1]{
  \fill ($(#1) + .25*(\Ddotsdist,0)$) circle (.03);
  \fill ($(#1) + .5*(\Ddotsdist,0)$) circle (.03);
  \fill ($(#1) + .75*(\Ddotsdist,0)$) circle (.03);
} 

\node [matrix] (Q1) at (6,0)
{
\coordinate [label={[label distance=0mm]above:{$1$}}] (1) at (4,2);
\coordinate [label={[label distance=1mm]above:{$2$}}] (2) at (2,2);
\coordinate [label={[label distance=1mm]below:{$3$}}] (3) at (2,0);
\coordinate [label={[label distance=1mm]below:{$4$}}] (4) at (0,0);

\foreach \v in {1,2,3,4} {\fill (\v) circle (.06);};
\foreach \s/\t in {2/3} {
  \draw [->,shorten <=1.7mm,shorten >=1.7mm] ($(\s)+(0.06,0)$) to ($(\t)+(0.06,0)$);
  \draw [->,shorten <=1.7mm,shorten >=1.7mm] ($(\s)-(0.06,0)$) to ($(\t)-(0.06,0)$);
};
\foreach \s/\t in {1/2} {
  \draw [->,shorten <=1.7mm,shorten >=1.7mm] ($(\s)+(0,0.06)$) to ($(\t)+(0,0.06)$);
  \draw [->,shorten <=1.7mm,shorten >=1.7mm] ($(\s)-(0,0.06)$) to ($(\t)-(0,0.06)$);
};
\foreach \s/\t in {4/2, 3/1} {
  \draw [->,shorten <=1.7mm,shorten >=1.7mm] ($(\s)$) to ($(\t)$);
};\\
};

\node [matrix] (Q1) at (0,0)
{
\coordinate [label={[label distance=0mm]above:{$1$}}] (1) at (4,2);
\coordinate [label={[label distance=1mm]above:{$2$}}] (2) at (2,2);
\coordinate [label={[label distance=1mm]below:{$3$}}] (3) at (2,0);
\coordinate [label={[label distance=1mm]below:{$4$}}] (4) at (0,0);

\foreach \v in {1,2,3,4} {\fill (\v) circle (.06);};
\foreach \s/\t in {2/1} {
  \draw [->,shorten <=1.7mm,shorten >=1.7mm] ($(\s)+(0,0.06)$) to ($(\t)+(0,0.06)$);
  \draw [->,shorten <=1.7mm,shorten >=1.7mm] ($(\s)-(0,0.06)$) to ($(\t)-(0,0.06)$);
};
\foreach \s/\t in {4/2, 1/3} {
  \draw [->,shorten <=1.7mm,shorten >=1.7mm] ($(\s)$) to ($(\t)$);
};\\
};

\draw[-implies,line width=1pt,double distance=4pt] (2,0) -- (4,0) node[midway,above=1.2mm] {$\mu_1$};
\end{tikzpicture}
\]
One can check that $\wt{B}$ and $\mu_1(\wt{B})$ are compatible with $L$ and $\mu_1(L)$ given by
\[
\begin{tikzpicture}[thick,>=\arrtip]
\node (original) at (-.7,0) {$\begin{bmatrix}
0 & -2 & -2 & 0 \\ 2 & 0 & 0 & 2 \\ 2 & 0 & 0 & 4 \\ 0 & -2 & -4 & 0 \end{bmatrix}$};
\node (mutant) at (4.5,0) {$\begin{bmatrix}
0 & 2 & 2 & 4 \\ -2 & 0 & 0 & 2 \\ -2 & 0 & 0 & 4 \\ -4 & -2 & -4 & 0 \end{bmatrix}.$};
\draw[-implies,line width=1pt,double distance=4pt] (1.5,0) -- (2.5,0) node[midway,above=1.2mm] {$\mu_1$};
\end{tikzpicture}
\]
In Section \ref{sec:K0asclusteralgebra} we identify the quantum cluster algebra $\cA_{(L,\wt{B})}$ of this example with the ring $K^{GL_2(\cO) \rtimes \Gm}(\Gr_{GL_2})$. The matrix $\mu_1(\wt{B})$ coincides with the matrix $\wt{B}(\mb{i})$ attached to the unipotent cell $N_+^{(s_0s_1)^2}$ of $LSL_2$ as in \cite{GLS13}, whose indexing convention is also illustrated here.
\end{Example}

\begin{Definition}
A $\Z$-graded, finite length, monoidal abelian category $\grl{\cC}$ is a \newword{monoidal categorification} of $\cA_{(L,\wt{B})}$ if there is a $\Z[q^{\pm 1/2}]$-algebra isomorphism $K_0(\grl{\cC}) \cong \cA_{(L,\wt{B})}$ which identifies all quantum cluster monomials with classes of simple objects. 
\end{Definition}

Similarly, if $\cC$ does not have a $\Z$-grading we say it is a monoidal categorification of $\cA_{\wt{B}}$ if there is an isomorphism $K_0(\cC) \cong \cA_{\wt{B}}$ which identifies all cluster monomials with classes of simple objects. Monoidal categorifications of $\cA^{loc}_{(L,\wt{B})}$ and $\cA^{loc}_{\wt{B}}$ are defined the same way. Note that the simple objects whose classes are (quantum) cluster variables are necessarily real and that two such objects in a common cluster necessarily ($q$-)commute.

\begin{Remark}
The notion of monoidal cluster categorification was introduced in \cite{HL10} and the quantum version studied in \cite{KKKO18}. We note that \cite{HL10} asks for the stronger condition that all cluster variables are classes of prime objects, but as in \cite{KKKO18} we do not (though see Remark \ref{rem:prime}). Conversely, as in \cite{HL10} but not \cite{KKKO18} we do not insist on any particular exact sequences in our definition. In the quantum case we also deviate from \cite{KKKO18} in omitting any discussion of self-dual or bar-invariant objects (though see Section \ref{sec:barinvolution}).
\end{Remark}

We have the following categorical counterparts of the algebraic notions reviewed above, again fixing a ring category $\cC$ with a system of renormalized $r$-matrices and separable triple products. 
We say that a collection $\{M_i\}_{i \in I}$ of pairwise commuting real simples in $\cC$ is a \newword{monoidal cluster} if the objects $\{\bigconv_{i \in I} M_i^{a_i}\}_{(a_i) \in \N^I}$ are pairwise nonisomorphic. This data together with an $I \times I_{ex}$ exchange matrix constitutes a \newword{monoidal seed}. The \newword{coefficient matrix} $L = (\lmat_{ij})$ of a monoidal cluster is defined by $\lmat_{ij} := -\La(M_i, M_j)$.

We say $(\{M_i\}_{i \neq k} \cup \{M'_k\}, \mu_k(\wt{B}))$ is a \newword{mutation of $(\{M_i\}_{i \in I}, \wt{B})$ in direction $k$} if $M_k \conv M'_k$ fits into an exact sequence
\begin{equation}\label{eq:seedseq}
0 \to \bigconv_{b_{ik} > 0} M_i^{b_{ik}} \to M_k \conv M'_k \to \bigconv_{b_{ik} < 0} M_i^{(-b_{ik})} \to 0.
\end{equation}
When clearer, we also write $\mu_k(M_i):= M_i$ for $i \neq k$ and $\mu_k(M_k) := M'_k$.

\begin{Remark}
It follows from properties of $r$-matrices that the existence of the exact sequence appearing in the definition of monoidal seed mutation is equivalent to the existence of the reverse sequence whose middle term is $M'_k \conv M_k$. A similar remark holds in the quantum case below, given suitable grading shifts.
\end{Remark}

\begin{Proposition}[c.f. Proposition 7.1.2 \cite{KKKO18}]\label{prop:KKKO6.2a}
Let $(\{M_i\}_{i \in I}, \wt{B})$ be a monoidal seed such that the $k$th column of $\wt{B}$ is nonzero, and let $M'_k$ be a simple object such that $M_k \conv M'_k$ fits into the exact sequence (\ref{eq:seedseq}). Then
\begin{enumerate}[(i)]
\item $(\{M_i\}_{i \neq k} \cup \{M'_k\}, \mu_k(\wt{B}))$ is a mutation of $(\{M_i\}_{i \in I}, \wt{B})$ in direction $k$ if $M'_k$ commutes with $M_i$ for $i \neq k$ and $K_0(\cC)$ is an integral domain, and
\item $M'_k$ commutes with $M_i$ for $i \neq k$ if the coefficient matrix $L$ of $\{M_i\}_{i \in I}$ is compatible with $\wt{B}$.
\end{enumerate}
\end{Proposition}
\begin{proof}
For (i), we only need that $M'_k$ is real, the required algebraic independence of $[M'_k]$ following from (\ref{eq:seedseq}). Reality follows from Proposition \ref{prop:KKKO2.21} applied to (\ref{eq:seedseq}). Specifically, the outer terms of (\ref{eq:seedseq}) are nonisomorphic since the $k$th column of $\wt{B}$ is nonzero, and integrality of~$K_0(\cC)$ implies their products with $M'_k$ are nonisomorphic.

For (ii), the commutativity assumption together with Proposition \ref{prop:KKKO2.20} yield
\begin{align}\label{eq:coefficient}
  \sum_{j \in I} \La(M_i, M_j)b_{jk} &= \sum_{b_{jk} > 0} \La(M_i, M_j)b_{jk} + \sum_{b_{jk} < 0} \La(M_i, M_j)b_{jk}\nonumber\\
&= \La(M_i, M'_k \hconv M_k) - \La(M_i, M_k \hconv M'_k) \nonumber\\
  &= - \La(M'_k \hconv M_k, M_i) - \La(M_i, M_k \hconv M'_k)\\
&= - (\La(M'_k, M_i) + \La(M_i, M'_k)) - (\La(M_k, M_i) + \La(M_i, M_k)) \nonumber
  %&= - (\La(M_i, M'_k) + (\La(M'_k, M_i)).\nonumber
\end{align}
Thus by Proposition \ref{prop:KKKO2.11} compatibility together with the fact that $M_i$ and $M_k$ commute implies that $M_i$ and $M'_k$ commute.
\end{proof}

\begin{Proposition}[c.f. Proposition 7.1.2 \cite{KKKO18}]\label{prop:KKKO6.2b}
Let $(\{M_i\}_{i \in I}, \wt{B})$ be a monoidal seed admitting a mutation in direction $k$ for each $k \in I_{ex}$. Then the coefficient matrix of $\{M_i\}_{i \in I}$ is compatible with $\wt{B}$, and  the coefficient matrix of $\{M_i\}_{i \neq k} \cup \{M'_k\}$ is compatible with $\mu_k(\wt{B})$.
\end{Proposition}

\begin{proof}
The first claim follows from (\ref{eq:coefficient}), reversing the role of hypothesis and conclusion. From Proposition \ref{prop:KKKO2.20} and inspection of (\ref{eq:seedseq}) it follows that the coefficient matrix of $\{M_i\}_{i \neq k} \cup \{M'_k\}$ is the mutation of that of $\{M_i\}_{i \in I}$, so the second claim follows from the first.
\end{proof}

Suppose $\grl{\cC}$ is a graded realization of $\cC$ compatible with its $r$-matrices, and let $\{\grl{M}_i\}_{i \in I}$ be a collection of pairwise $q$-commuting real simple objects in $\grl{\cC}$. Given a sequence of these objects (possibly with repetitions) we define the normalized product
$$  \grl{M}_{i_1} \odot \cdots \odot \grl{M}_{i_n} := (\grl{M}_{i_1} \conv \cdots \conv \grl{M}_{i_n})\{ -\frac12 \sum_{j < k} \Lambda(M_{i_j}, M_{i_k})\}. $$
This is independent of the ordering of the $\grl{M}_{i_j}$, hence given $(a_i) \in \N^I$ we have a well-defined object $\sodot_{i \in I} \grl{M}_i^{a_i}$.
  We say $\{\grl{M}_i\}_{i \in I}$ is a \newword{quantum monoidal cluster} if the objects $\{\sodot_{i \in I} \grl{M}_i^{a_i}\}_{(a_i) \in \N^I}$ are pairwise nonisomorphic, even up to grading shifts (equivalently the map $\cR_q(L) \to K_0(\grl{\cC})$ generated by the classes $\{[\grl{M}_i]\}_{i \in I}$ is injective). The coefficient matrix $L = (\lmat_{ij})$ of $\{\grl{M}_i\}_{i \in I}$ is then characterized by $\grl{M}_i \conv \grl{M}_j \{\lmat_{ij}\} \cong \grl{M}_j \conv \grl{M}_i$.

The additional data of an exchange matrix compatible with $L$ defines a \newword{quantum monoidal seed} $(\{\grl{M_i}\}_{i \in I}, \wt{B})$. We say $(\{\grl{M}_i\}_{i \neq k} \cup \{\grl{M}'_k\}, \mu_k(\wt{B}))$ is a mutation of $(\{\grl{M}_i\}_{i \in I}, \wt{B})$ in direction $k$ if $\grl{M}_k \conv \grl{M}'_k$ fits into an exact sequence
\begin{equation}\label{eq:qseedseq}
0 \to (\sodot_{b_{ik} > 0} \grl{M}_i^{b_{ik}})\{-1\} \to (\grl{M}_k \conv \grl{M}'_k)\{-\frac12 \Lambda(M_k, M'_k)\} \to (\sodot_{b_{ik} < 0} \grl{M}_i^{(-b_{ik})}) \to 0.
\end{equation}
Note that implicitly we also ask for $\Lambda(M_k, M'_k)$ to be integral (rather than half-integral) here. 

If $\cC$ is rigid then mutations of monoidal seeds are unique: it follows from (\ref{eq:seedseq}) that $M'_k$ is determined by the two identities
$$ M'_k \cong \left(\bigconv_{b_{ik} < 0} M_i^{(-b_{ik})} \right) \hconv M_k^R \quad \text{and} \quad
M'_k \cong M_k^L \hconv \left(\bigconv_{b_{ik} > 0} M_i^{b_{ik}} \right). $$
However, for an arbitrary monoidal seed one cannot generally expect that the right-hand sides of these formulas are isomorphic and fit into the needed sequence to constitute a mutation. In particular, a monoidal seed need not admit arbitrary sequences of mutations.

Nonetheless, the following result gives two criteria for when the existence of a single initial mutation in every direction guarantees the existence of further sequences of mutations.

\begin{Theorem}[c.f. Theorem 7.1.3 \cite{KKKO18}]\label{thm:K3Omain}
Let $(\{M_i\}_{i \in I}, \wt{B})$ be a monoidal seed that admits a mutation in direction $k$ for all $k \in I_{ex}$, and suppose $K_0(\cC)$ is an integral domain. Then $(\{M_i\}_{i \in I}, \wt{B})$ admits a sequence $ \mu_{i_n} \circ \cdots \circ \mu_{i_1} $ of mutations if either of the following holds.
\begin{enumerate}[(i)]
\item For all $1 \leq j \leq n-1$, $i_j$ is not a vertex of an oriented 3-cycle of the quiver associated to $\mu_{i_{j-1}} \circ \cdots \circ \mu_{i_1}(\wt{B})$.
\item There is a graded realization $\grl{\cC}$ with the following properties. 
\begin{itemize}
\item $(\{M_i\}_{i \in I}, \wt{B})$ lifts to a quantum monoidal seed $(\{\grl{M_i}\}_{i \in I}, \wt{B})$ which admits a mutation in direction $k$ for all $k \in I_{ex}$. 
\item The induced embedding $\cR_q(L) \into K_0(\grl{\cC})$ extends to an embedding 
$$\Q(q^{1/2}) \otimes_{\Z[q^{\pm 1/2}]}\cA_{(L, \wt{B})} \into \Q(q^{1/2}) \otimes_{\Z[q^{\pm 1/2}]} K_0(\grl{\cC}).$$
\end{itemize}
In this case, $(\{\grl{M_i}\}_{i \in I}, \wt{B})$ admits the sequence $ \mu_{i_n} \circ \cdots \circ \mu_{i_1} $ of mutations as a quantum monoidal seed.
\end{enumerate}
\end{Theorem}

\begin{Corollary}
Let $(\{\grl{M_i}\}_{i \in I}, \wt{B})$ be a quantum monoidal seed in $\grl{\cC}$ which admits a mutation in direction $k$ for all $k \in I_{ex}$. If the induced embedding $\cR_q(L) \into K_0(\grl{\cC})$ extends to an isomorphism $\cA_{(L, \wt{B})} \cong K_0(\grl{\cC})$ (resp. $\cA^{loc}_{(L, \wt{B})} \cong K_0(\grl{\cC})$) then $\grl{\cC}$ is a monoidal categorification of $\cA_{(L, \wt{B})}$ (resp. $\cA^{loc}_{(L, \wt{B})}$).
\end{Corollary}

\begin{Remark}
Only case (ii) of Theorem \ref{thm:K3Omain} appears explicitly in \cite{KKKO18}, but case (i) is implicit in its proof. We highlight case (i) because we find it to be conceptually significant in relation to potentials, as discussed in the introduction. We have also removed the hypothesis that the objects in (ii) are self-dual, as the bar involution is a somewhat less elementary structure in geometric settings than algebraic ones (see Section \ref{sec:barinvolution}) and the axiomatics we consider are kept more flexible if we do not assume its presence. This does, however, add a few extra steps to the proof given in \cite{KKKO18}, which is the main reason we reproduce the argument here.
\end{Remark}

\begin{proof}[Proof of Theorem \ref{thm:K3Omain}]
The claim is proved by induction. Consider first the base case of a sequence of length 2, and let $x:= i_1$, $y:=i_2$. For now we do not assume either of the hypotheses in (i) or (ii), since the beginning of the argument is the same in either case. We write $M'_x$, $M'_y$ for the objects which replace $M_x$, $M_y$ when mutating $(\{M_i\}_{i \in I}, \wt{B})$ in directions $x$ and $y$, respectively.

To show that $(\{M_i\}_{i \neq x} \cup \{M'_x\}, \mu_x(\wt{B}))$ admits a mutation in direction $y$ we must produce a real simple object $M''_y$ which commutes with $M'_x$ and $M_i$ for $i \notin \{x,y\}$, and which fits into the needed exact sequence. 
We assume that $b_{xy} >0$, the $b_{xy} <0$ case being similar and the $b_{xy} = 0$ case being trivial. From now one we let $b'_{ij} = \mu_x(\wt{B})_{ij}$.
We introduce simple objects
\begin{gather*}
A = \left(\bigconv_{\substack{b'_{iy}\leq 0\\ b_{ix}>0}} M_i^{b_{ix}b_{iy}}\right) \conv \left( \bigconv_{\substack{b'_{iy}>0, \:b_{iy}<0\\ b_{ix}>0}} M_i^{ -b_{iy}}\right), \quad B = \left(\bigconv_{\substack{b_{iy}>0\\ b_{ix}>0}} M_i^{b_{ix}b_{iy}}\right) \conv \left( \bigconv_{\substack{b'_{iy}>0, \:b_{iy}<0\\ b_{ix}>0}} M_i^{ b'_{iy}}\right), \\
C = \bigconv_{b_{ix}>0}M_i^{b_{ix}}, \quad
L = (M'_x)^{b_{xy}}, \quad P = \bigconv_{\substack{b_{iy} > 0\\ i \neq x}} M_i^{ b_{iy}}, \quad Q = \bigconv_{\substack{b'_{iy} < 0\\ i \neq x}}M_i^{(-b_{iy})}, \quad S = \bigconv_{\substack{b_{ix} < 0\\ i \neq y}}M_i^{(-b_{ix})},
\end{gather*}
letting us write the needed exact sequence as
\begin{equation}\label{eq:goal}
0 \to P\conv B \to M_y \conv M''_y \to (M'_x)^{b_{xy}} \conv Q \to 0,
\end{equation}
and the sequences we have by hypothesis as
\begin{gather}\label{eq:yexseq}
0 \to M_x^{b_{xy}} \conv P \to M_y \conv M'_y \to Q \conv A \to 0,\\
0 \to C \to M_x \conv M'_x \to M_y^{b_{xy}} \conv S \to 0,
\end{gather}
We will produce (\ref{eq:goal}) from (\ref{eq:yexseq}) by commuting each term in the latter sequence past~$(M'_x)^{b_{xy}}$.

We first observe that
\[
\La(M'_x, M_x^{b_{xy}} \conv P)  = \La(M'_x, M_y \conv M'_y) = \La(M'_x, Q \conv A).
\]
The equality $\La(M'_x, Q \conv A) = \La(M'_x, M_y\conv M'_y)$ follows from Proposition \ref{prop:KKKO2.20} since $M_y$ and~$M'_x$ commute. On the other hand, one can verify using  (\ref{eq:coefficient}) and Proposition~\ref{prop:KKKO2.20} that  $\Lambda(M'_x, M_x^{b_{xy}} \conv P) = \Lambda(M'_x, Q \conv A).$
\begin{comment}
\begin{align*}
  \Lambda(M'_x, M_x^{b_{xy}} \conv P) - \Lambda(M'_x, Q \conv A) &= \sum_{i \in I} \Lambda(M'_x, M_i)b_{iy}\\
                                      &= \sum_{i \neq x} \Lambda(M'_x, M_i)b'_{iy} - \sum_{b_{ix} > 0}\Lambda(M'_x, M_i) b_{ix}b_{xy} + \Lambda(M'_x, M_x) b_{xy}\\
                                      &= (- \Lambda(M'_x, \bigconv_{b_{ix} >0}M_i^{ b_{ix}}) + \Lambda(M'_x, M_x))b_{xy}\\
  &= (- \Lambda(M'_x, M'_x \hconv M_x) + \Lambda(M'_x, M_x))b_{xy} = 0,
\end{align*}
where the second equality follows from (\ref{eq:coefficient}) and the last from Proposition \ref{prop:KKKO2.20}.
\end{comment}

It follows then from Definition \ref{def:rmat}(\ref{p4}) that the following diagram is commutative.
\begin{eqnarray*}
&&\begin{array}{c}\xymatrix{
0 \ar[r] & (M'_x)^{b_{xy}} \conv (M_x^{b_{xy}} \conv P) \ar[r] \ar[d]_{\rmat{(M'_x)^{b_{xy}},(M_x^{b_{xy}} \conv P)}} &  (M'_x)^{b_{xy}} \conv (M_y \conv M_y')
 \ar[d]_{\rmat{(M'_x)^{b_{xy}},(M_y \conv M_y')}} \ar[r] &  (M'_x)^{b_{xy}} \conv Q \conv A \ar[r] \ar[d]_{\rmat{(M'_x)^{b_{xy}},(Q \conv A)}} ^{\mbox{\large$\wr$}}&0\\
0 \ar[r] & (M_x^{b_{xy}} \conv P) \conv (M'_x)^{b_{xy}} \ar[r]  & (M_y \conv M_y') \conv (M'_x)^{b_{xy}} \ar[r] & \, Q \conv A \conv (M'_x)^{b_{xy}} \ar[r] & 0
}\end{array}
\end{eqnarray*}
In particular, we obtain an exact sequence
\begin{gather}\label{eq:imseqorig}
0 \to \im(\rmat{(M'_x)^{b_{xy}},(M_x^{b_{xy}} \conv P)}) \to \im(\rmat{(M'_x)^{b_{xy}}, (M_y \conv M'_y)}) \to (M'_x)^{b_{xy}} \conv Q \conv A \to 0.
\end{gather}

Since
\[
\rmat{(M'_x)^{b_{xy}}, (M_y \conv M'_y)} = (\id_{M_y} \conv \rmat{(M'_x)^{b_{xy}}, M'_y}) \circ (\rmat{(M'_x)^{b_{xy}},M_y} \conv \id_{M'_y})
\]
and $\rmat{(M'_x)^{b_{xy}},M_y}$ is an isomorphism, we have $\im(\rmat{(M'_x)^{b_{xy}}, (M_y \conv M'_y)}) \cong M_y \conv ((M'_x)^{b_{xy}} \hconv M'_y)$. Likewise, since
\[
\rmat{(M'_x)^{b_{xy}}, ((M_x)^{b_{xy}} \conv P)} = (\id_{(M_x)^{b_{xy}}} \conv \rmat{(M'_x)^{b_{xy}}, P}) \circ (\rmat{(M'_x)^{b_{xy}},(M_x)^{b_{xy}}} \conv \id_{P})
\]
and  $\rmat{(M'_x)^{b_{xy}}, P}$ is an isomorphism, we have $\im(\rmat{(M'_x)^{b_{xy}}, ((M_x)^{b_{xy}} \conv P)}) \cong ((M'_x)^{b_{xy}} \hconv (M_x)^{b_{xy}}) \conv P$.

By Lemma \ref{lem:KKKO2.23}, we have
$ (M'_x)^{b_{xy}} \hconv (M_x)^{b_{xy}} \cong (M'_x \hconv M_x)^{b_{xy}}.$
On the other hand, by inspection we have
$ (M'_x \hconv M_x)^{b_{xy}} \cong C^{\, b_{xy}} \cong A \conv B,$
hence
$
\im(\rmat{(M'_x)^{b_{xy}}, (M_x)^{b_{xy}} \conv P}) \cong A \conv B \conv P.
$
Thus we have an exact sequence
\begin{gather}\label{eq:imseq}
0 \to (P \conv B) \conv A \to M_y \conv ((M'_x)^{b_{xy}} \hconv M'_y) \to ((M'_x)^{b_{xy}} \conv Q) \conv A \to 0.
\end{gather}

Inspecting the factors appearing in the definition of $A$, one sees that it can be nontrivial only if $x$ and $y$ are vertices of an oriented 3-cycle in the quiver associated to $\wt{B}$. In this case, taking $M''_y = (M'_x)^{b_{xy}} \hconv M'_y$ we obtain the needed sequence (\ref{eq:goal}). By Propositions~\ref{prop:KKKO6.2a} and~\ref{prop:KKKO6.2b} $M''_y$ is real and commutes with $\mu_x(M_i)$ for $i \neq y$, hence defines a mutation in direction~$y$. If $x$ is not a vertex of any oriented 3-cycle of the quiver associated to $\wt{B}$, it follows that $(\{\mu_x(M_i)\}_{i \in I}, \mu_x(\wt{B}))$ admits a mutation in direction $k$ for any $k \in I_{ex}$. Claim (i) then follows by induction.

Suppose now that $A$ is possibly nontrivial, and that we are in the setting of claim (ii). Given a real simple $N = \bigconv_{i \in I} M_i^{n_i}$ we write $\grl{N} \in \grl{\cC}$ for its lift to the normalized product $\sodot_{i \in I} \grl{M}_i^{n_i}$. Sequence (\ref{eq:imseqorig}) now descends from the sequence
\begin{flalign*}
0 \to \im(\rmat{(\grl{M}'_x)^{b_{xy}},(\grl{M}_x^{b_{xy}} \odot \grl{P})\{-1\}}) \to \im(\rmat{(\grl{M}'_x)^{b_{xy}}, (\grl{M}_y \conv \grl{M}'_y)\{-\frac12 \Lambda(M_y, M'_y)\}}) \to \quad \quad \quad \quad \\ \quad \quad \quad \quad \to (\grl{Q} \odot \grl{A}) \conv (\grl{M}'_x)^{b_{xy}} \{\Lambda((M'_x)^{b_{xy}}, Q \conv A)\} \to 0.
\end{flalign*}
in $\grl{\cC}$. In light of (\ref{eq:imseq}), we may rewrite this as
\[
0 \to ((\grl{P} \odot \grl{B}) \conv \grl{A})\{\alpha\} \to (\grl{M}_y \conv ((\grl{M}'_x)^{b_{xy}} \hconv \grl{M}'_y))\{\beta\} \to (((\grl{M}'_x)^{b_{xy}} \odot \grl{Q}) \conv \grl{A})\{\gamma\} \to 0
\]
for some grading shifts $\alpha$, $\beta$, $\gamma$; let us calculate $\alpha$ and $\gamma$.

On one hand, we have
\begin{align*}
((\grl{Q} \odot \grl{A}) \conv (\grl{M}'_x)^{b_{xy}}) \{ \Lambda((M'_x)^{b_{xy}}, Q \conv A)\} %& \cong (\grl{M}'_x)^{b_{xy}} \conv (\grl{Q} \odot \grl{A})\\
%& \cong ((\grl{M}'_x)^{b_{xy}} \conv \grl{Q} \conv \grl{A})\{ \frac12 \Lambda(Q, A)\}\\
& \cong (((\grl{M}'_x)^{b_{xy}} \odot \grl{Q}) \conv \grl{A})\{\frac12\Lambda((M'_x)^{b_{xy}}, Q)-\frac12 \Lambda(Q, A)\}.
\end{align*}
On the other, we see that
\begin{align*}
  ((\grl{M}'_x)^{b_{xy}} \conv (\grl{M}_x^{b_{xy}} \odot \grl{P}))\{-1\} & \cong ((\grl{M}'_x)^{b_{xy}} \conv \grl{P} \conv \grl{M}_x^{b_{xy}})\{-1 - \frac12 \Lambda(P, M_x^{b_{xy}}) \}\\
%& \hspace{-1.5cm}\cong (\grl{P} \conv (\grl{M}'_x)^{b_{xy}} \conv \grl{M}_x^{b_{xy}}) \{\Lambda((M'_x)^{b_{xy}}, P) - 1 - \frac12 \Lambda(P, M_x^{b_{xy}})\}\\
& \hspace{-1.5cm}\onto (\grl{P} \conv ((\grl{M}'_x)^{b_{xy}} \hconv \grl{M}_x^{b_{xy}})) \{\Lambda((M'_x)^{b_{xy}}, P) - 1 - \frac12 \Lambda(P, M_x^{b_{xy}})\}\\
                                                                     & \hspace{-1.5cm}\cong (\grl{P} \conv \grl{C}^{\,b_{xy}}) \{\Lambda((M'_x)^{b_{xy}}, P) - 1 - \frac12 \Lambda(P, \grl{M}_x^{b_{xy}}) + \frac12 \Lambda((M'_x)^{b_{xy}}, M_x^{b_{xy}})\},
\end{align*}
where in the last line we use Lemma \ref{lem:gradedKKKO2.23} and the fact that $\grl{M}'_x \hconv \grl{M}_x \cong \grl{C}\{\frac12 \Lambda(M'_x, M_x)\}$. To finish computing  $\alpha$ we then check that
\begin{align*}
  \grl{P} \conv \grl{C}^{\,b_{xy}} \cong (\grl{P} \conv \grl{B} \conv \grl{A}) \{ -\frac12 \Lambda(B, A)\}
   \cong ((\grl{P} \odot \grl{B}) \conv \grl{A})  \{ \frac12 \Lambda(P,B) - \frac12 \Lambda(B, A)\}.
\end{align*}
Putting these calculations together we obtain
\begin{align*}
  \gamma - \alpha %& = 1 + \frac12\Lambda(P,M_x^{b_{xy}}) - \Lambda((M'_x)^{b_{xy}},P) - \frac12\Lambda((M'_x)^{b_{xy}},M_x^{b_{xy}}) \\
  %& \quad \quad \quad - \frac12 \Lambda(P,B) + \frac12 \Lambda(B,A) - \frac12 \Lambda(Q,A) + \Lambda((M'_x)^{b_{xy}},Q)\\
                  & = 1 + \frac12\Lambda(P,M_x^{b_{xy}}) - \frac12 \Lambda((M'_x)^{b_{xy}},P) - \frac12 \Lambda(P,B) \\
  & \quad \quad\quad + \frac12\Lambda(B,A) - \frac12 \Lambda(Q,A) - \frac12 \Lambda((M'_x)^{b_{xy}},A)\\
                  & = 1 + \frac12 \Lambda(P,A) + \frac12(B,A) - \frac12 \Lambda(Q,A) - \frac12((M'_x)^{b_{xy}}, A) = 1.
\end{align*}
Here the first equality follows from the earlier computation that
$ \Lambda(M'_x, M_x^{b_{xy}} \conv P) = \Lambda(M'_x, Q \conv A). $
The second follows from the fact that since $P$ commutes with $M'_x$,
$$ \Lambda(P, (M'_x)^{b_{xy}} \conv M_x^{b_{xy}}) = \Lambda(P, C^{\,b_{xy}}) = \Lambda(P, B \conv A). $$
The last equality follows since $A$ is a product of factors $M_i$ with $i \notin \{x, y\}$, and since
\begin{align*}
  \Lambda(B \conv P, M_i) - \Lambda(Q \conv (M'_x)^{b_{xy}}, M_i) &= \sum_{j \in I} \Lambda(\mu_x(M_j),M_i) b'_{jy} = 0
\end{align*}
because $\mu_x(\wt{B})$ is compatible with the coefficient matrix of $\{\mu_x(M)_i\}_{i \in I}$ by Proposition \ref{prop:KKKO6.2b}.

From this calculation it follows in particular that in $K_0(\grl{\cC})$ we have
$$q^{\gamma - \beta} [\grl{M}_y][(\grl{M}'_x)^{b_{xy}} \hconv \grl{M}'_y] = (q[\grl{P} \odot \grl{B}] + [(\grl{M}'_x)^{b_{xy}} \odot \grl{Q}])[\grl{A}]. $$
By hypothesis, we have a quantum cluster variable $\Phi \in \Q(q^{1/2}) \otimes_{\Z[q^{\pm 1/2}]} K_0(\grl{\cC})$ which satisfies
\[
q^m[\grl{M}_y]\Phi = q[\grl{P} \odot \grl{B}] + [(\grl{M}'_x)^{b_{xy}} \odot \grl{Q}]
\]
for $m = -\frac12 \sum_{b'_{iy}<0} \Lambda(M_y, \mu_x(M_i))b'_{iy}$.
Thus $q^{\gamma - \beta}[\grl{M}_y][(\grl{M}'_x)^{b_{xy}} \hconv \grl{M}'_y] = q^m [\grl{M}_y]\Phi[\grl{A}]$, and since $K_0(\grl{\cC})$ is an integral domain $q^{\gamma - \beta - m}[(\grl{M}'_x)^{b_{xy}} \hconv \grl{M}'_y] = \Phi[\grl{A}].$

By hypothesis we also know that $\Phi$ $q$-commutes with $[\grl{M}_i]$ for $i \notin \{x, y\}$, hence with $[\grl{A}]$. Following the notation of Theorem \ref{thm:KKKO3.3}, consider its expansion
$ \Phi = \sum_{b \in \cB} a_b [\grl{S}_b] $
in the $\Q(q^{1/2})$-basis of $\Q(q^{1/2}) \otimes_{\Z[q^{\pm 1/2}]} K_0(\grl{\cC})$ given by the classes of simple objects up to grading shifts. By Theorem \ref{thm:KKKO3.3} $\grl{A}$ $q$-commutes with $\grl{S}_b$ whenever $a_b \neq 0$, and in particular $\grl{S}_b \conv \grl{A}$ is then simple. But then since
$$ q^{\gamma - \beta - m}[(\grl{M}'_x)^{b_{xy}} \hconv \grl{M}'_y] = \Phi[\grl{A}] = \sum_{b \in \cB} a_b [\grl{S}_b \conv \grl{A}] $$
and $(\grl{M}'_x)^{b_{xy}} \hconv \grl{M}'_y$ is itself simple, it follows that $\Phi = [\grl{S}_b\{n\}]$ for some $b \in \cB$, $n \in \Z$.

Letting $\grl{M}''_y := \grl{S}_b\{n\}$,  we then have
\begin{gather*}
q^m[\grl{M}_y] [\grl{M}''_y] = q[\grl{P} \odot \grl{B}] + [(\grl{M}'_x)^{b_{xy}} \odot \grl{Q}], \quad q^\ell[\grl{M}''_y][\grl{M}_y]  = [\grl{P} \odot \grl{B}] + q[(\grl{M}'_x)^{b_{xy}} \odot \grl{Q}]
\end{gather*}
in $K_0(\grl{\cC})$, where $\ell = \frac12 \sum_{b'_{iy} > 0} \Lambda(\mu_x(M_i), M_y)b'_{iy}$. It follows that $(\grl{M}'_x)^{b_{xy}} \odot \grl{Q}$ is either the head or socle of $(\grl{M}_y \conv \grl{M}''_y)\{-m\}$. In the latter case, $((\grl{M}'_x)^{b_{xy}} \odot \grl{Q})\{-1\}$ would necessarily be the head of $(\grl{M}''_y \conv \grl{M}_y)\{-\ell\}$, which would imply that $\Lambda(M''_y, M_y) = \ell - m - 1$. Applying the same argument to $\grl{P} \odot \grl{B}$, we would have $\Lambda(M_y, M''_y) = m - \ell - 1$, hence $\Lambda(M_y, M''_y) + \Lambda(M''_y, M_y) < 0$, a contradiction. Thus we conclude there exists an exact sequence
$$ 0 \to (\grl{P} \odot \grl{B})\{-1\} \to (\grl{M}_y \conv \grl{M}''_y)\{-m\} \to (\grl{M}'_x)^{b_{xy}} \odot \grl{Q} \to 0, $$
and as a corollary that
$$m = -\frac12 \sum_{b'_{iy}<0} \Lambda(M_y, \mu_x(M_i))b'_{iy} = \frac12 \Lambda(M_y, M_y \hconv M''_y) = \frac12 \Lambda(M_y, M''_y).$$

By hypothesis $\Phi = [\grl{M}''_y]$ $q$-commutes with $[\mu_x(\grl{M}_i)]$ for $i \neq y$, hence by Theorem \ref{thm:KKKO3.3} $\grl{M}''_y$ $q$-commutes with $\mu_x(\grl{M}_i)$ for $i \neq y$. Moreover, it follows from Proposition \ref{prop:KKKO6.2a} that $\grl{M}''_y$ is real. Altogether, we thus obtain that $(\{\mu_x(\grl{M}_i)\}_{i \neq y} \cup \{\grl{M}''_y\}, \mu_y\mu_x(\wt{B}))$ is a mutation of $(\{\mu_x(\grl{M}_i)\}_{i \in I}, \mu_x(\wt{B}))$ in direction $y$. In particular, since $x$ and $y$ were chosen arbitrarily, $(\{\mu_x(\grl{M}_i)\}_{i \in I}, \mu_x(\wt{B}))$ admits a mutation in all directions. Claim (ii) follows by induction.
\end{proof}

\section{Chiral categories and $r$-matrices}\label{sec:chiral}

In this section we construct a system of renormalized $r$-matrices in the coherent Satake category using the Beilinson-Drinfeld Grassmannian. The construction is a meromorphic analogue of the Eckmann-Hilton argument, and is similar to its counterparts in the setting of KLR and quantum loop algebras.
It can be carried out in any monoidal category~$\cC$ which can be extracted from a fiber of a chiral category with some suitable compatibilities. At this level of generality our construction should find applications in other situations and is better illuminated by the physical perspective discussed in the introduction. 

\subsection{Definitions and background}

The theory of chiral algebras is a geometric reformulation of the theory of vertex algebras \cite{BD04}. In its equivalent formulation as a factorization algebra, a chiral algebra on a smooth complex curve $X$ is, roughly speaking, a family of vector spaces $\cV_{\{x_i\}}$ which are parametrized by all finite subsets $\{x_i\} \subset X$ and which behave well with respect to disjoint unions. Informally, a chiral category is the same notion but with categories in place of vector spaces. 

\begin{Definition} \cite{Ras14} A (unital) chiral category $\cCh$ on a curve $X$ is a multiplicative sheaf of DG categories on the (unital) Ran space of $X$.
\end{Definition}

We unpack this definition as far as needed, referring to \cite{Ras14} for details. In this section all DG categories are cocomplete, and we denote by $\QCoh(X)$ the DG category of quasicoherent sheaves on $X$. A sheaf of DG categories on $X$ is a module category $\cC_X$ of $\QCoh(X)$. Given a morphism $f: Y \to X$ we have the pullback $f^*\cC_X := \QCoh(Y) \otimes_{\QCoh(X)} \cC_X$.

The \newword{Ran space} $\Ran_X$ of $X$ is an object that parametrizes nonempty finite subsets of $X$. If $f: J \to I$ is a surjection of finite sets, we write $\Delta_f: X^I \to X^J$ for the induced diagonal morphism. A sheaf of categories on $\Ran_X$ includes the data of a sheaf of categories $\cCh_{X^I}$ on $X^I$ for each finite set $I$. These are related by equivalences 
\begin{equation}\label{eq:diagonal} 
\Delta_f^*(\cCh_{X^J}) \cong \cCh_{X^I} 
\end{equation} 
for all surjections $f$, and these equivalences are related by natural compatibilities. 

The operation of disjoint union of finite sets endows $\Ran_X$ with a partially-defined multiplication. Let us write $[X^I \times X^J]_{disj}$ for the variety parametrizing $I$-tuples and $J$-tuples of points with no overlap. In other words
$$[X^I \times X^J]_{disj} = \{(\ux,\ux') \in X^I \times X^J: x_i \ne x'_j \text{ for all } i \in I, j \in J \}.$$
\begin{comment}
Then we have a correspondence of open immersions
\[
  \begin{tikzpicture}
    \node (a) at (0,0) {$[X^I \times X^J]_{disj}$};
    \node (b) at (-3,-1.5) {$X^I \times X^J$};
    \node (c) at (2.5,-1.5) {$X^{I \sqcup J}$.};
    \draw [thick,-\arrtip] ($(a.south) + (-.9,.1)$) to ($(b.north)+(.5,0)$);
    \draw [thick,-\arrtip] ($(a.south) + (.3,.1)$) to ($(c.north)+(-.6,0)$);
  \end{tikzpicture}
  \]
\end{comment}
A \newword{multiplicative} sheaf of categories on $\Ran_X$ is one equipped with the additional structure of equivalences
\begin{equation}\label{eq:factorization} 
\cCh_{X^I} \boxtimes \cCh_{X^J}|_{[X^I \times X^J]_{disj}} \cong \cCh_{X^{I \sqcup J}}|_{[X^I \times X^J]_{disj}} \end{equation}
satisfying natural compatibilities (with respect both to iterated unions of the form $I \sqcup J \sqcup K$ and to the equivalence in (\ref{eq:diagonal})). 

The equivalence in (\ref{eq:factorization}) is often called the factorization equivalence. In the simplest case, if we denote $\Delta: X \rightarrow X^2$ the diagonal embedding and $j: U \rightarrow X^2$ its complement then the equivalences (\ref{eq:diagonal}) and (\ref{eq:factorization}) say that
$$\Delta^*(\cCh_{X^2}) \cong \cCh_X \ \ \text{ and } \ \ j^*(\cCh_{X^2}) \cong j^*(\cCh_X \boxtimes \cCh_X).$$

The unital Ran space $\Ran_X^{un}$ parametrizes possibly empty subsets of $X$ and carries information about not necessarily surjective maps between finite sets. In particular, given a sheaf of categories $\cCh$ on $\Ran_X^{un}$ we have for each finite set $I$ a unit functor $\eta: \QCoh(X^I) \to \cCh_{X^I}$. Thus, given an inclusion $I \into I \sqcup J$, the unit induces a functor $\eta: \cCh_{X^I} \boxtimes \QCoh(X^J) \to \cCh_{X^{I \sqcup J}}$.

Specializing to $X = \A^1$, the $\G_m$ action on $\A^1$ is naturally compatible with diagonal and unit maps, hence induces an action on $\Ran_X$ and $\Ran_X^{un}$. This action preserves the disjoint locus, hence the multiplicative structures on $\Ran_X$ and $\Ran_X^{un}$ descend to $\Ran_X/\G_m$ and $\Ran_X^{un}/\G_m$.

\begin{Definition} A $\G_m$-equivariant (unital) chiral category $\cCh$ on $X = \A^1$ is a multiplicative sheaf of DG categories on $\Ran_X/\G_m$ (resp. $\Ran_X^{un}/\G_m$).
\end{Definition}

\subsection{Renormalized $r$-matrices}\label{sec:chiralrmat}

From hereon we let $\cCh$ be a unital chiral category over $X = \Spec\, \C[t]$. We suppose that $\cCh_{X}$ and $\cCh_{X^2}$ possess monoidal (i.e.~$\E_1$) structures compatible with their $\QCoh(X)$ and $\QCoh(X^2)$-module structures. By restriction we obtain a monoidal structure on the fiber $\cC_0 := i^*_0(\cCh_X) = \QCoh(pt) \otimes_{\QCoh(X)} \cCh_X$ of $\cCh_X$ at $0 \in X$. In natural examples $\cCh$ is a chiral $\E_1$-category (in particular $\cC_{X^I}$ is monoidal for all $I$), but as our construction only requires a much smaller amount of structure we do not impose this as a hypothesis. 

We assume $\cCh$ has the following properties. 
\begin{enumerate}
\item \label{co:0} The diagonal and factorization functors relating $\cCh_{X}$ and $\cCh_{X^2}$ are monoidal. 
\item \label{co:1} The restriction functor $\cCh_X \to \cC_0$ admits a monoidal section $\cC_0 \rightarrow \cCh_X$ which preserves compact objects.
\item \label{co:2} The unit functor $\eta: \cCh_{X} \boxtimes \QCoh(X) \to \cCh_{X^2}$ preserves compact objects. 
\item \label{co:3} The monoidal structure on $\cCh_{X^2}$ preserves compact objects. 
\end{enumerate}

For $M \in \cC_0$ we denote by $\wt{M}$ its image in $\cCh_X$ under the section in (\ref{co:1}). Note that when $\cCh_X$ can be trivialized over $X$ the object $\wt{M}$ is a deformation of $M$, which could be taken to be trivial. Triviality of $\cCh_X$, as well as all properties involving compactness, will be clear in our main example involving the BD Grassmannian. Also note that (\ref{co:3}) is automatic if $\cCh_{X^2}$ is rigid. 

Let $h\cC_0$ denote the homotopy category of $\cC_0$, and let $\cC \subset h\cC_0$ be a full monoidal subcategory consisting of compact objects. Given the above data, we will say that the monoidal category $\cC$ is equipped with a \newword{compatible chiral category}. We now explain how a compatible chiral category assigns a canonical map $\rmat{M,N} : M \conv N \to N \conv M$ and an integer $\La(M,N) \in \Z \cup \{-\infty\}$ to any two objects $M,N \in \cC$.

Consider the two unit functors 
$$\eta_1: \cCh_X \boxtimes \QCoh(X) \to \cCh_{X^2} \ \ \text{ and } \ \ \eta_2: \QCoh(X) \boxtimes \cCh_X \to \cCh_{X^2}.$$
These correspond to the inclusions $i_1: \{1\} \hookrightarrow \{1,2\}$ and $i_2: \{2\} \hookrightarrow \{1,2\}$. On the other hand, the equivalence (\ref{eq:diagonal}) applied to the two projections $p_1: \{1,2\} \twoheadrightarrow \{1\}$ and $p_2: \{1,2\} \twoheadrightarrow \{2\}$ give us equivalences 
$$\Delta_{p_1}^*(\cCh_{X^2}) \cong \cCh_X \cong \Delta_{p_2}^*(\cCh_{X^2}).$$
Since $p_1 \circ i_1$ and $p_2 \circ i_2$ are identity maps it follows that these equivalences identify the restrictions of $\eta_1(\wt{M} \boxtimes \O_X)$ and $\eta_2(\O_X \boxtimes \wt{N})$ to the diagonal $\Delta \subset X^2$ with $\wt{M}$ and $\wt{N}$, respectively. 

Next, consider the two products 
\begin{equation}
\label{eq:rmat1} \eta_1(\wt{M} \boxtimes \O_X) * \eta_2(\O_X \boxtimes \wt{N}) \ \ \text{ and } \ \ \eta_2(\O_X \boxtimes \wt{N}) * \eta_1(\wt{M} \boxtimes \O_X).
\end{equation}
Since the factorization equivalence 
$$\cCh_{X^2}|_{[X^2]_{disj}} \cong (\cCh_X \boxtimes \cCh_X)|_{[X^2]_{disj}}$$
is monoidal by condition (\ref{co:0}) the restriction of these two products to $[X^2]_{disj}$ is isomorphic (via this equivalence) to $\wt{M} \boxtimes \wt{N} \in (\cCh_X \boxtimes \cCh_X)|_{[X^2]_{disj}}$. 

We will study the restrictions of these two products to $X \times \{0\} \subset X^2$. We denote these $C_{M,N}$ and $C'_{M,N}$ respectively. If we let $\iota: \{0\} \rightarrow X$ denote the inclusion of the origin and $j: U \rightarrow X$ the inclusion of its complement then following the above discussion we have 
$$\iota^*(C_{M,N}) \cong M \conv N, \ \ \iota^*(C'_{N,M}) \cong N \conv M, \ \  j^*(C_{M,N}) \cong j^*((\wt{M} \boxtimes \wt{N})|_{X \times \{0\}}) \cong j^*(C'_{N,M}).$$
The identity map gives us a natural section 
\begin{equation}
\label{eq:rmat0} \O_U \rightarrow \Hom(j^* C_{M,N}, j^* C'_{N,M}).
\end{equation}

Combining conditions (\ref{co:1}), (\ref{co:2}) and (\ref{co:3}) it follows that the two objects in (\ref{eq:rmat1}) are compact. Moreover, the restriction functor to $X \times \{0\} \subset X^2$ has a right adjoint which commutes with colimits (c.f. \cite[Remark 19.2.1]{Ras14}). This implies that $C_{M,N}$ and $C'_{N,M}$ are also compact. 
But we then have
\begin{align*}
\Hom(j^* C_{M,N}, j^* C'_{N,M}) 
&\cong \Hom(C_{M,N}, j_* j^* C'_{N,M}) \cong \Hom(C_{M,N}, C'_{N,M} \otimes_{\C[t]} \C[t^\pm]) \\
&\cong \Hom(C_{M,N}, C'_{N,M} \otimes_{\C[t]} \colim t^{-k} \C[t]) \\
&\cong \Hom(C_{M,N}, \colim (C'_{N,M} \otimes_{\C[t]} t^{-k} \C[t])) \\
&\cong \colim (\Hom(C_{M,N}, C'_{N,M} \otimes_{\C[t]} t^{-k} \C[t])) \\
&\cong \Hom(C_{M,N}, C'_{N,M}) \otimes_{\C[t]} \C[t^\pm],
\end{align*}
where compactness of $C_{M,N}$ is used to obtain the second-to-last isomorphism. It follows that (\ref{eq:rmat0}) gives a map $\O_U \rightarrow j^*(V_{M,N})$ where $V_{M,N} := \Hom(C_{M,N}, C'_{N,M})$. 

Consider the composition 
\begin{equation}\label{eq:rmat2}
\alpha: \O_X \rightarrow j_* \O_U \rightarrow j_* j^* V_{M,N} \cong V_{M,N} \otimes_{\C[t]} \C[t^{\pm 1}]. 
\end{equation}
The $\C[t]$-module $\C[t^{\pm 1}]$ has a natural increasing filtration 
$$\dots \subset F^{k-1} \subset F^k \subset F^{k+1} \subset \dots \subset \C[t^{\pm 1}]$$
where $F^k := t^{-k} \C[t]$. For a $\C[t]$-module $M$ we let $[M]^k$ denote the $k$th piece of the induced filtration on $M \otimes_{\C[t]} \C[t^\pm]$, letting $[M]^{-\infty}$ denote $\cap_{k \in \Z}[M]^k$. 

Since $\O_X$ is a compact object there exists some $k \in \Z \cup \{-\infty\}$ so that the map $\alpha$ factors through $[V_{M,N}]^k$. Given such a $k$ we let $\alpha^k \in H^0([V_{M,N}]^{k})$ denote the resulting section.  We further let $\La(M,N) \in \Z \cup \{-\infty\}$ be the minimal value such that $\alpha$ factors through $[V_{M,N}]^k$. 

Suppose $\La(M,N) \in \Z$. We can tensor the exact sequence
$$ 0 \to F^{\La(M,N)-1} \rightarrow F^{\La(M,N)} \rightarrow \C_0 \otimes_{\C[t]} F^{\La(M,N)} \to 0 $$
with $V_{M,N}$ to obtain a long exact sequence
$$\cdots \rightarrow H^0([V_{M,N}]^{\La(M,N)-1}) \rightarrow H^0([V_{M,N}]^{\La(M,N)}) \xrightarrow{\pi} H^0(V_{M,N}|_{\{0\}} \otimes_{\C[t]} F^{\La(M,N)}) \rightarrow \cdots$$
in cohomology. By construction we have $\pi(\alpha^{\La(M,N)}) \ne 0$. But 
$$V_{M,N}|_{\{0\}} \otimes_{\C[t]} F^{\La(M,N)} \cong V_{M,N}|_{\{0\}} \cong \Hom(\iota^*(C_{M,N}), \iota^*(C'_{M,N})) \cong \Hom(M*N,N*M),$$
so $\pi(\alpha^{\La(M,N)})$ yields a nonzero map $\rmat{M,N} \in H^0(\Hom(M*N,N*M))$.

If $\La(M,N) = -\infty$, on the other hand, we set $\rmat{M,N} = 0$. Note this is just the result of applying the above construction to the exact sequence of complexes whose first two terms are $[V_{M,N}]^{\La(M,N)-1} = [V_{M,N}]^{\La(M,N)} = [V_{M,N}]^{-\infty}$ but whose last term is now zero.

\begin{Remark}\label{rem:BD}
The construction above uses the restriction to $X \times \{0\} \subset X^2$. Restriction to other lines through the origin will give an {\it a priori} different collection of $r$-maps. However, in the case that $\cCh_X$ is really a trivial family over $X$ (as is the case with the Beilinson-Drinfeld construction in Section \ref{sec:BD}) the $r$-matrices will be independent of this choice. 
\end{Remark}

\begin{figure}
\centering
\begin{tikzpicture}[thick,>=\arrtip,decoration={brace,mirror,amplitude=7}]
\newcommand*{\Cht}{1}
\newcommand*{\Cwh}{3}
\newcommand*{\Coff}{.75}
\newcommand*{\Rwh}{1.2}
\newcommand*{\Rht}{1.2}
\newcommand*{\xa}{1}
\newcommand*{\xb}{2.5}
\newcommand*{\ya}{.5}
\newcommand*{\yb}{1.5}
\newcommand*{\bigx}{4}
\newcommand*{\bigy}{5}
\newcommand*{\loff}{\Rwh*.25}
\newcommand*{\lnameoff}{.3}
\newcommand*{\pnameoffa}{.25}
\newcommand*{\pnameoffb}{.25}
\newcommand*{\boff}{2}
\draw [decorate] (-2.0,-\boff) --node[below=3mm]{$C_{M,N}$} (5.2,-\boff);
\draw [decorate] (6,-\boff) --node[below=3mm]{$C'_{N,M}$} (13.2,-\boff);
\node [matrix] (UL) at (-.2,0)
{
  \node at (-.0,\Cht*.5+.2) {$\C$};
  \node at (\xa-\Rwh*.5-.4,\yb+\Rht*.5+.05) {$\R^2$};
\draw[dashed] (\xa,\ya) -- (\xa,\yb);
%\draw[dashed] (\xb,\ya) -- (\xb,\yb);
\fill (\xa,\ya) circle (.03);
%\fill (\xb,\ya) circle (.03);
\node at (\xa+\pnameoffa,\ya) {$0$};
%\node at (\xb+\pnameoffb,\ya) {$z$};
\foreach \xval in {\xa} {
  \draw (\xval-\Rwh*.5,\yb) -- (\xval-\Rwh*.5,\yb+\Rht) -- (\xval+\Rwh*.5,\yb+\Rht) -- (\xval+\Rwh*.5,\yb) -- (\xval-\Rwh*.5,\yb);
};
\draw[color=blue, very thick] (\xa-\Rwh*.5+\loff,\yb) -- (\xa-\Rwh*.5+\loff,\yb+\Rht);
\draw[color=blue, very thick] (\xa+\Rwh*.5-\loff,\yb) -- (\xa+\Rwh*.5-\loff,\yb+\Rht);
\node at (\xa-\Rwh*.5+\loff,\yb+\Rht+\lnameoff) {$M$};
\node at (\xa+\Rwh*.5-\loff,\yb+\Rht+\lnameoff) {$N$};
\node at (\xa+.02,\yb+\Rht+\lnameoff) {$\conv$};
\draw (0,0) -- (\Coff,\Cht) -- (\Coff+\Cwh,\Cht) -- (\Cwh,0) -- (0,0);\\
};

\node [matrix] (UR) at (\bigx,0)
{
\draw[dashed] (\xa,\ya) -- (\xa,\yb);
\draw[dashed] (\xb,\ya) -- (\xb,\yb);
\fill (\xa,\ya) circle (.03);
\fill (\xb,\ya) circle (.03);
\node at (\xa+\pnameoffa,\ya) {$0$};
\node at (\xb+\pnameoffb,\ya) {$z$};
\foreach \xval in {\xa, \xb} {
  \draw (\xval-\Rwh*.5,\yb) -- (\xval-\Rwh*.5,\yb+\Rht) -- (\xval+\Rwh*.5,\yb+\Rht) -- (\xval+\Rwh*.5,\yb) -- (\xval-\Rwh*.5,\yb);
};
\draw[color=blue, very thick] (\xb-\Rwh*.5+\loff,\yb) -- (\xb-\Rwh*.5+\loff,\yb+\Rht);
\draw[color=blue, very thick] (\xa+\Rwh*.5-\loff,\yb) -- (\xa+\Rwh*.5-\loff,\yb+\Rht);
\node at (\xb-\Rwh*.5+\loff,\yb+\Rht+\lnameoff) {$M$};
\node at (\xa+\Rwh*.5-\loff,\yb+\Rht+\lnameoff) {$N$};
\draw (0,0) -- (\Coff,\Cht) -- (\Coff+\Cwh,\Cht) -- (\Cwh,0) -- (0,0);\\
};

\node [matrix] (BL) at (2*\bigx,0)
{
\draw[dashed] (\xa,\ya) -- (\xa,\yb);
\draw[dashed] (\xb,\ya) -- (\xb,\yb);
\fill (\xa,\ya) circle (.03);
\fill (\xb,\ya) circle (.03);
\node at (\xa+\pnameoffa,\ya) {$0$};
\node at (\xb+\pnameoffb,\ya) {$z$};
\foreach \xval in {\xa, \xb} {
  \draw (\xval-\Rwh*.5,\yb) -- (\xval-\Rwh*.5,\yb+\Rht) -- (\xval+\Rwh*.5,\yb+\Rht) -- (\xval+\Rwh*.5,\yb) -- (\xval-\Rwh*.5,\yb);
};
\draw[color=blue, very thick] (\xa-\Rwh*.5+\loff,\yb) -- (\xa-\Rwh*.5+\loff,\yb+\Rht);
\draw[color=blue, very thick] (\xb+\Rwh*.5-\loff,\yb) -- (\xb+\Rwh*.5-\loff,\yb+\Rht);
\node at (\xa-\Rwh*.5+\loff,\yb+\Rht+\lnameoff) {$N$};
\node at (\xb+\Rwh*.5-\loff,\yb+\Rht+\lnameoff) {$M$};
\draw (0,0) -- (\Coff,\Cht) -- (\Coff+\Cwh,\Cht) -- (\Cwh,0) -- (0,0);\\
};

\node [matrix] (BR) at (3*\bigx,0)
{
\draw[dashed] (\xa,\ya) -- (\xa,\yb);
%\draw[dashed] (\xb,\ya) -- (\xb,\yb);
\fill (\xa,\ya) circle (.03);
%\fill (\xb,\ya) circle (.03);
\node at (\xa+\pnameoffa,\ya) {$0$};
%\node at (\xb+\pnameoffb,\ya) {$z$};
\foreach \xval in {\xa} {
  \draw (\xval-\Rwh*.5,\yb) -- (\xval-\Rwh*.5,\yb+\Rht) -- (\xval+\Rwh*.5,\yb+\Rht) -- (\xval+\Rwh*.5,\yb) -- (\xval-\Rwh*.5,\yb);
};
\draw[color=blue, very thick] (\xa-\Rwh*.5+\loff,\yb) -- (\xa-\Rwh*.5+\loff,\yb+\Rht);
\draw[color=blue, very thick] (\xa+\Rwh*.5-\loff,\yb) -- (\xa+\Rwh*.5-\loff,\yb+\Rht);
\node at (\xa-\Rwh*.5+\loff,\yb+\Rht+\lnameoff) {$N$};
\node at (\xa+\Rwh*.5-\loff,\yb+\Rht+\lnameoff) {$M$};
\node at (\xa-.015,\yb+\Rht+\lnameoff) {$\conv$};
\draw (0,0) -- (\Coff,\Cht) -- (\Coff+\Cwh,\Cht) -- (\Cwh,0) -- (0,0);\\
};
\end{tikzpicture}
\caption{}\label{fig:lineops}
\end{figure}

\begin{Remark}\label{rem:EHarg}
Following the discussion in the introduction of line operators in 4d holomorphic-topological field theory it is instructive to visualize the present construction as in Figure~\ref{fig:lineops}. The objects of $\cC_0$ are interpreted as being extended along lines in $\R^2$ with a fixed slope. Morphisms are interpreted as junctions between objects extended along different regions of a fixed line, while the monoidal structure is interpreted as collision of parallel lines. The structure of $\cC_0$ is constrained by the existence of an extension to a family $\cCh_X$ over $X = \C$ in different ways depending on the nature of the dependence of $\cC_z$ on the point $z \in X$. 

Were we considering a topological family, the usual Eckmann-Hilton argument depicted in the figure would force the monoidal structure on the homotopy category $h\cC_0$, and interesting subcategories thereof, to be symmetric ($\cC_0$ itself possessing an $\E_3$ structure). This is the case for the constructible Satake category, or more generally for line operators in a fully topological theory in 4d (e.g. topologically twisted $\cN=4$ gauge theory \cite{KW07}).

On the other hand, the existence of a holomorphic or algebraic family $\cCh_X$ does not force~$h\cC_0$ to be symmetric. Instead, what one can salvage from the Eckmann-Hilton construction is exactly the pair $C_{M,N}$, $C'_{N,M}$ of sections of $\cCh_X$ which are isomorphic away from $z=0$. The resulting constraint is that the subcategory of compact objects in $h\cC_0$ possesses a system of renormalized $r$-matrices.
\end{Remark}

\begin{Theorem}\label{thm:rmat}
The maps $\rmat{M,N}$ and elements $\Lambda(M,N) \in \Z \cup \{-\infty\}$ defined above form a system of renormalized $r$-matrices in $\cC$.
\end{Theorem}
\begin{proof}
Properties (\ref{p1}) and (\ref{p2}) of Definition~\ref{def:rmat} are clear. To check property (\ref{p3}) consider $M,N_1,N_2 \in \cC$ and let
\begin{align*}
C_{M,N_1,N_2} &:= [\eta_1(\wt{M} \boxtimes \O_X) * \eta_2(\O_X \boxtimes \wt{N_1}) * \eta_2(\O_X \boxtimes \wt{N_2})]|_{X \times \{0\}} \\
C'_{M,N_1,N_2} &:= [\eta_2(\O_X \boxtimes \wt{N_1}) * \eta_1(\wt{M} \boxtimes \O_X) * \eta_2(\O_X \boxtimes \wt{N_2})]|_{X \times \{0\}} \\
C''_{M,N_1,N_2} &:= [\eta_2(\O_X \boxtimes \wt{N_1}) * \eta_2(\O_X \boxtimes \wt{N_2}) * \eta_1(\wt{M} \boxtimes \O_X)]|_{X \times \{0\}}.
\end{align*}
As in the construction of $\rmat{M,N}$, after restricting to $U \subset X$ the factorization equivalence identifies each of these with $\wt{M} \boxtimes \wt{N_1*N_2}$.
The identity of $\wt{M} \boxtimes \wt{N_1*N_2}$
then induces maps
\begin{align*}
  \alpha_1: \cO_X \to \Hom(C_{M,N_1,N_2},C'_{M,N_1,N_2}) \otimes_{\C[t]} \C[t^{\pm 1}] \\ \alpha_2: \cO_X \to \Hom(C'_{M,N_1,N_2},C''_{M,N_1,N_2}) \otimes_{\C[t]} \C[t^{\pm 1}]
\end{align*}
which factor through $[\Hom(C_{M,N_1,N_2},C'_{M,N_1,N_2})]^{\La(M,N_1)}$ and $[\Hom(C'_{M,N_1,N_2},C''_{M,N_1,N_2})]^{\La(M,N_2)}$, respectively.

On the other hand, 
\begin{gather*}
C_{M,N_1,N_2} \cong [\eta_1(\wt{M} \boxtimes \O_X) * \eta_2(\O_X \boxtimes \wt{N_1*N_2})]|_{X \times \{0\}}  \cong C_{M,N_1 \conv N_2} \\
C''_{M,N_1,N_2} \cong [\eta_2(\O_X \boxtimes \wt{N_1*N_2}) * \eta_1(\wt{M} \boxtimes \O_X)]|_{X \times \{0\}} \cong C'_{N_1 \conv N_2, M},
\end{gather*}
where the right-hand sides refer to the notation used in the construction of $\rmat{M, N_1 \conv N_2}$. Since they are both obtained by adjunction from the same identity section over $U$, we have an equality $\alpha_2 \circ \alpha_1 = \alpha$ of maps $\cO_X \to \Hom(C_{M,N_1 \conv N_2}, C'_{N_1 \conv N_2, M}) \otimes_{\C[t]} \C[t^{\pm 1}]$ (where $\alpha$ is the map used to define $\rmat{M, N_1 \conv N_2}$).

Since $\alpha_2 \circ \alpha_1$ factors through $[\Hom(C_{M,N_1 \conv N_2}, C'_{N_1 \conv N_2, M})]^{\La(M, N_1) + \La(M, N_2)}$ it follows from the minimality of $\La(M, N_1 \conv N_2)$ that
\begin{equation}\label{eq:p3ineq} \La(M, N_1 \conv N_2) \leq \La(M, N_1) + \La(M, N_2).
\end{equation}
On the other hand, the final part of the constructions of $\rmat{M, N_1 \conv N_2}$ applied to $\alpha_1$ and $\alpha_2$ produces the maps $\rmat{M,N_1} * id_{N_2}$ and $id_{N_1} * \rmat{M,N_2}$, respectively. If (\ref{eq:p3ineq}) is an equality it follows straightforwardly that the composition of these is equal to $\rmat{M, N_1 \conv N_2}$. If instead the inequality is strict, the composition is zero since $\alpha^{\La(M, N_1) + \La(M, N_2)}$ is the image of $\alpha^{\La(M, N_1) + \La(M, N_2) - 1}$ in the relevant long exact sequence. This establishes property (\ref{p3}).

To check property (\ref{p4}) let
$$\alpha_1: \O_X \to V_{M,N} \otimes_{\C[t]} \C[t^{\pm 1}], \quad \alpha_2: \O_X \to V_{N,M} \otimes_{\C[t]} \C[t^{\pm 1}]$$
denote the maps appearing in the construction of $\rmat{M,N}$ and $\rmat{N,M}$, respectively. The composition $\alpha_2 \circ \alpha_1$ factors through $[\Hom(C_{M,N}, C_{M,N})]^{\La(M,N)+\La(N,M)}$. On the other hand, it must also equal the map $\cO_X \to \Hom(C_{M,N}, C_{M,N}) \otimes_{\C[t]} \C[t^{\pm 1}]$ induced by the identity of $C_{M,N}$, which factors through $[\Hom(C_{M,N}, C_{M,N})]^0$ but not  $[\Hom(C_{M,N}, C_{M,N})]^k$ for any $k < 0$. It follows as above that $\La(M,N)+\La(N,M) \ge 0$ with $\rmat{N,M} \circ \rmat{M,N}$ equal to $\id_{M\conv N}$ if equality holds, zero otherwise. This proves condition (\ref{p4}). 

Finally, to check property (\ref{p5}) consider 
$$C_{M,N_1} := [\eta_1(\wt{M} \boxtimes \O_X) * \eta_2(\O_X \boxtimes \wt{N_1})]|_{X \times \{0\}}, \quad C'_{N_2,M} := [\eta_2(\O_X \boxtimes \wt{N_2}) * \eta_1(\wt{M} \boxtimes \O_X)]|_{X \times \{0\}}.$$
These satisfy
\begin{gather*}
\iota^*(C_{M,N_1}) \cong M*N_1, \ \ \iota^*(C'_{N_2,M}) \cong N_2*M, \\
j^*(C_{M,N_1}) \cong j^*(\wt{M} \boxtimes \wt{N_1}), \ \ j^*(C'_{N_2,M}) \cong j^*(\wt{M} \boxtimes \wt{N_2}).
\end{gather*}
Letting $V_{M,N_\bullet} := \Hom(C_{M,N_1},C'_{N_2,M})$ we have a map
$$\alpha_f: \O_X \rightarrow V_{M,N_\bullet} \otimes_{\C[t]} \C[t^{\pm 1}]$$ obtained by adjunction from the morphism 
$$id_{\wt{M}} \boxtimes \wt{f}: \wt{M} \boxtimes \wt{N_1} \rightarrow \wt{M} \boxtimes \wt{N_2}.$$
We let $\La(M,f) \in \Z \cup \{-\infty\}$ be the minimal value such that $\alpha_f$ factors through $[V_{M,N_\bullet}]^{\La(M,f)}$. As before we obtain a map $\rmat{M,f}: M*N_1 \rightarrow N_2*M,$ where for example if $N_1=N=N_2$ and $f=id$  we recover $\rmat{M,N}$. 

On the other hand, we also have 
\begin{align*}
\beta_1 := [\eta_2(id_{\O_X} \boxtimes \wt{f}) * \eta_1(\id_{\wt{M} \boxtimes \O_X})]|_{X \times \{0\}} \circ \alpha_1 & : \O_X \to V_{M,N_\bullet} \otimes_{\C[t]} \C[t^{\pm 1}], \\
\beta_2 := \alpha_2 \circ [\eta_1(id_{\wt{M} \boxtimes \O_X}) * \eta_2(\id_{\O_X} \boxtimes \wt{f})]|_{X \times \{0\}} & : \O_X \to V_{M,N_\bullet} \otimes_{\C[t]} \C[t^{\pm 1}],
\end{align*}
where $\alpha_1,\alpha_2$ are the maps corresponding to $\rmat{M,N_1}$ and $\rmat{M,N_2}$. These factor through $[V_{M,N_\bullet}]^{\La(M, N_1)}$ and $[V_{M,N_\bullet}]^{\La(M, N_2)}$, respectively. Considering the fiber over $0 \in X$, the usual construction applied to $\beta_1,$ $\beta_2$ recovers
$$(f*id_{M}) \circ \rmat{M,N_1}\ \ \text{and}\ \  \rmat{M,N_2} \circ (id_M*f) \in \Hom(M*N_1,N_2*M).$$
From this we see that 
\begin{itemize}
\item $\La(M,f) \le \La(M,N_1)$ with $(f*id_M) \circ \rmat{M,N_1}$ equal to $\rmat{M,f}$ when equality holds and zero otherwise, 
\item $\La(M,f) \le \La(M,N_2)$ with $\rmat{M,N_2} \circ (id_M*f)$ equal to $\rmat{M,f}$ when equality holds and zero otherwise. 
\end{itemize}
Condition (\ref{p5}) now follows. Note that if $\La(M,N_1) = \La(M,N_2)$ but both are greater than $\La(M,f)$ then the diagram in (\ref{p5}) commutes with both compositions being zero.
\end{proof}

Finally, suppose that $\cCh$ is pulled back from a $\Gm$-equivariant chiral category $\grl{\cCh}$ along the map $\Ran_X^{un} \to \Ran_X^{un}/\G_m$. The fiber $\grl{\cC}_{0}$ of $\grl{\cCh_X}$ at $\{0\}/\G_m$ is the category of $\G_m$-equivariant objects in $\cC_0$. The full subcategory $\grl{\cC} \subset h\grl{\cC}_{0}$ consisting of objects whose image under $\grl{\cC}_{0} \to \cC_0$ lies in $\cC$ is a graded realization in the sense of Section \ref{sec:gradedrealization}.

\begin{Proposition}\label{prop:Gmequiv}
The system of renormalized $r$-matrices constructed in $\cC$ is compatible with the graded realization $\grl{\cC}.$
\end{Proposition}
\begin{proof}
As before we denote by $\{\cdot\}$ the $\Gm$-equivariance shift. Recall that 
$$\pi(\alpha) \in H^0(V_{M,N}|_{\{0\}} \otimes_{\C[t]} F^k).$$
The result then follows since $F^k = t^{-k} \C[t] \xrightarrow{\sim} \C[t] \{2k\}$ where $t^{-k} \mapsto 1$. 
\end{proof}

\subsection{The Beilinson-Drinfeld Grassmannian}\label{sec:BD}

The affine Grassmannian has the following moduli interpretation. Fix a smooth curve $X$ and $x \in X$. Then $\Gr_G$ can be identified with the moduli of pairs $(P,\phi)$ where $P$ is a principal $G$-bundle on $X$ and 
$$\phi: P_0|_{X \setminus \{x\}} \xrightarrow{\sim} P|_{X \setminus \{x\}}$$
is an isomorphism (here $P_0$ denotes the trivial $G$-bundle on $X$). The twisted product $\Gr_G \ttimes \Gr_G = G(\cK) \times_{G(\O)} \Gr_G$ also has a moduli interpretation as
$$\{(P,P',\phi,\phi'): \phi: P_0|_{X \setminus \{x\}} \xrightarrow{\sim} P|_{X \setminus \{x\}}, \phi': P|_{X \setminus \{x\}} \xrightarrow{\sim} P'|_{X \setminus \{x\}} \}.$$

Allowing $x \in X$ to move we obtain a global version of $\Gr_G$ called the Beilinson-Drinfeld (BD) Grassmannian \cite{BD04}. Consider a tuple $\ux \in X^I$ of points labeled by a finite set $I$. Then the BD Grassmannian $\Gr_{G,X^I}$ parametrizes 
$$\{(\ux, P, \phi): \ux \in X^I, \phi: P_0|_{X \setminus \{\ux\}} \xrightarrow{\sim} P|_{X \setminus \{\ux\}} \}$$ 
where again $P$ is a principal $G$-bundle. Remembering only $\ux \in X^I$ gives us a map $\Gr_{G,X^I} \rightarrow X^I$. 

The case most relevant to us is $|I|=2$. In this case the fibers over a point $(x_1,x_2) \in X^2$ are given by 
$$(\Gr_{G,X^2})_{(x_1,x_2)} \cong \begin{cases}
\Gr_G \times \Gr_G & \text{ if } x_1 \ne x_2 \\
\Gr_G & \text{ if } x_1 = x_2.
\end{cases}$$
This is a somewhat counterintuitive feature which is only possible because $\Gr_G$ is infinite dimensional. If we restrict to Schubert varieties we get a clearer picture: for a pair of dominant coweights $(\l^\vee,\mu^\vee)$ the general fiber is isomorphic to $\oGr_G^{\l^\vee} \times \oGr_G^{\mu^\vee}$ while the special fiber is isomorphic to $\oGr_G^{\l^\vee+\mu^\vee}$. These finite-dimensional varieties indeed have the same dimension, as expected.

\begin{Remark}\label{rem:BDdeform}
If $x \in X$ and we restrict the discussion above to $X \times \{x\}$ then we obtain a family over $X$ where the fiber over $x \in X$ is $\oGr_G^{\l^\vee+\mu^\vee}$ while the fibers away from $x$ are isomorphic to $\oGr_G^{\l^\vee} \times \oGr_G^{\mu^\vee}$. This is a flat deformation of $\oGr_G^{\l^\vee+\mu^\vee}$ which has many applications (c.f. Section \ref{sec:Qsystem}). 
\end{Remark}

Taken together over all finite sets $I$, the BD Grassmannians $\Gr_{G,X^I}$ provide the prototypical example of a factorization space. That is, we have compatibilities along diagonal maps and disjoint unions generalizing the ones described above for $|I|=2$.

There is also a global analogue $G_{\O,X^I}$ of $G(\O)$. It parametrizes automorphisms
$$\{(\ux, \gamma): \ux \in X^I, \gamma: P_0|_{\widehat{\{\ux\}}} \xrightarrow{\sim} P_0|_{\widehat{\{\ux\}}} \}$$
where $\widehat{\{\ux\}}$ denotes the formal completion of $X$ along $\{\ux\}$. One can show that that $G_{\O,X^I}$ acts on $\Gr_{G,X^I}$ and that it has the usual naturality properties with respect to restriction to diagonals and disjoint unions. Subsequently, the DG categories $\IndCoh^{G_{\O,X^I}}(\Gr_{G,X^I})$ of equivariant ind-coherent sheaves give us a chiral category $\cCh_{G}$ on $X$. The section $X^I \rightarrow \Gr_{G,X^I}$ which takes $\ux \mapsto \{(\ux,P_0,id)\}$ equips it with a unital structure. 

We also have global analogues of the convolution product on $\Gr_G$. First consider the space parametrizing tuples
$$\{(\ux, \ux', P, P', \phi, \phi')| \ux \in X^I, \ux' \in X^J, \phi: P_0|_{X \setminus \{\ux\}} \xrightarrow{\sim} P|_{X \setminus \{\ux\}}, \phi': P|_{X \setminus \{\ux'\}} \xrightarrow{\sim} P'|_{X \setminus \{\ux'\}} \}.$$
This has the structure of a twisted product $\Gr_{G,X^I} \widetilde{\times} \Gr_{G,X^J}$ (c.f. \cite[Eq. 3.1.22]{Zhu16}). More precisely, there exists a $G_{\O,X^J}$-torsor over $\Gr_{G,X^I} \times X^J$ classifying 
$${\bf E} := \{(\ux,P,\phi,\ux',\gamma): \ux \in X^I, \ux' \in X^J, \phi: P_0|_{X \setminus \{\ux\}} \xrightarrow{\sim} P|_{X \setminus \{\ux\}}, \gamma: P|_{\widehat{\{\ux'\}}} \xrightarrow{\sim} P_0|_{\widehat{\{\ux'\}}} \}.$$
Then $\Gr_{G,X^I} \widetilde{\times} \Gr_{G,X^J}$ is the fiber product ${\bf E} \times_{G_{\O,X^J}} \Gr_{G,X^J}$. 

As in the local case, given $\cF \in \IndCoh^{G_{\cO,X^I}}(\Gr_{G,X^I})$, $\cG \in \IndCoh^{G_{\cO,X^J}}(\Gr_{G,X^J})$ we can form the twisted product $\cF \,\tbox\, \cG \in \IndCoh^{G_{\cO,X^I}}(\Gr_{G,X^I} \widetilde{\times} \Gr_{G,X^J})$. We then pushforward along the generalized multiplication map 
\begin{align*}
\om: \Gr_{G,X^I} \ttimes \Gr_{G,X^J} &\to \Gr_{G,X^{I \sqcup J}} \\
(\ux, \ux', P, P', \phi, \phi') &\mapsto ((\ux, \ux'), P', \phi' \circ \phi)
\end{align*}
to obtain an object in $\IndCoh^{G_{\cO,X^I}}(\Gr_{G,X^{I \sqcup J}})$. If $I=J$ we can restrict this object to the diagonal $X^I \subset X^{I \sqcup I}$ to obtain an object $\cF * \cG \in \IndCoh^{G_{\cO,X^I}}(\Gr_{G,X^I})$. This defines a monoidal structure on each $(\cCh_{G})_{X^I}$.

Finally, if $X$ has a $\Gm$ action then the construction above can be done $\Gm$-equivariantly to obtain a $\G_m$-equivariant chiral category $\grl{\cCh_{G}}$. We are particularly interested in the case $X=\A^1$ where the standard $\Gm$ action is taken with weight $2$. 

\begin{Theorem}\label{thm:satakermats}
The category $D_{coh}^{G(\O)}(\Gr_G)$ possesses a system of renormalized $r$-matrices compatible with the graded realization $D_{coh}^{G(\O) \rtimes \Gm}(\Gr_G)$. By restriction, the coherent Satake category $\cP_{coh}^{G(\cO)}(\Gr_G)$ inherits a system of renormalized $r$-matrices compatible with $\cP_{coh}^{G(\cO) \rtimes \G_m}(\Gr_G)$, and in which all $r$-matrices are nonzero.
\end{Theorem}
\begin{proof}
Using Theorem \ref{thm:rmat}, we can construct a system of renormalized $r$-matrices in $D_{coh}^{G(\O)}(\Gr_G)$ by applying the construction of Section \ref{sec:chiralrmat} to the chiral category $\cCh_{G}$ over $\A^1$. Note that the subcategory of compact objects in $\IndCoh^{G(\O)}(\Gr_G)$ is $D_{coh}^{G(\O)}(\Gr_G)$. The fact that the functors from conditions (\ref{co:1}), (\ref{co:2}) and (\ref{co:3}) of Section \ref{sec:chiralrmat} preserve compact objects is immediate since they all send coherent sheaves to coherent sheaves. 

Compatibility with the graded realizations follows from applying Proposition \ref{prop:Gmequiv} to the $\G_m$-equivariant chiral category $\grl{\cCh_{G}}$ over $\A^1$. The fact that the $r$-matrices in $\cP_{coh}^{G(\cO)}(\Gr_G)$ are nonzero follows from Corollaries \ref{cor:adjoints} and \ref{cor:rmat}. 
\end{proof}

Note that in general $\Gr_{G,X}$ is not isomorphic to the trivial product $\Gr_X \times X$ (it is a twisted product $\Gr_X \ttimes X$). But in the case $X=\A^1$ this twist is indeed trivial, justifying Remark \ref{rem:BD}.

From Theorem \ref{thm:satakermats} we immediately obtain a number of results on the convolution structure of $\cP_{coh}^{G(\cO) \rtimes \Gm}(\Gr_G)$.

\begin{Theorem}\label{thm:consequences}
Let $\cF$, $\cG \in \cP_{coh}^{G(\cO) \rtimes \Gm}(\Gr_G)$ be simple perverse coherent sheaves and suppose either $\cF$ or $\cG$ is real. Then
\begin{enumerate}
\item $\cF \conv \cG$ has a simple head and a simple socle,
\item the head and socle of $\cF \conv \cG$ each appear exactly once as a factor in any composition series of $\cF \conv \cG$, and
\item if $[\cF]$ and $[\cG]$ $q$-commute in $K^{G(\cO) \rtimes \G_m}(\Gr_{G})$ then $\cF \conv \cG$ is simple and isomorphic to $\cG \conv \cF \{ \Lambda(\cF,\cG)\}$.
\end{enumerate}
\end{Theorem}

\begin{proof}
The three claims follow from Proposition \ref{prop:KKKO3.2}, Proposition \ref{prop:KKKO3.1}, and Theorem \ref{thm:KKKO3.3}, respectively. 
\end{proof}

We note that the second claim says that the basis of $K^{G(\cO) \rtimes \G_m}(\Gr_{G})$ formed by classes of simple perverse coherent sheaves satisfies a partial analogue of Leclerc's conjecture \cite{Lec03} on the multiplicative structure of the dual canonical (or upper global) basis of a quantized enveloping algebra $\cU_q(\fn)$.

\section{Monoidal cluster categorification}\label{sec:GL_n}

We are now ready to combine the results of the previous sections and prove that the coherent Satake category of $GL_n$ is a monoidal cluster categorification. We then discuss the bar involution and quantum twist from this point of view, as well as the connection to $Q$-systems.

\subsection{A finite generating set}\label{sec:K0asclusteralgebra}
To begin we define an index set $I = I_{ex} \sqcup I_{fr}$ by
\[
I_{ex}= \{(k,0), (k,1)\}_{1 \leq k < n}, \quad I_{fr} = \{(n,0), (n,1)\}.
\]
We let $\wt{B}_n = (b_{ij})$ and $L_n = (\lmat_{ij})$ be the $I_{ex} \times I$ and $I \times I$ matrices given by
\begin{gather*}
b_{(k_1, \ell_1), (k_2, \ell_2)} = (\ell_1 - \ell_2)a_{(k_1,k_2)}, \\
\lmat_{(k_1, \ell_1), (k_2, \ell_2)} = 2 (\ell_2 - \ell_1)\min \{k_1, k_2\},
\end{gather*}
where $A = (a_{ij})$ is the Cartan matrix of type $A_{n+1}$. These form a compatible pair with the principal part of $L_n \wt{B}_n$ equal to $2 \Id_{I_{ex} \times I_{ex}}$.

\begin{Example}\label{ex:explicitbmatrix}
In the case $n=3$, we have
\[
L_3 = \begin{pmatrix}
0 & 0 & 0 & 2 & 2 & 2 \\
0 & 0 & 0 & 2 & 4 & 4 \\
0 & 0 & 0 & 2 & 4 & 6 \\
-2 & -2 & -2 & 0 & 0 & 0 \\
-2 & -4 & -4 & 0 & 0 & 0 \\
-2 & -4 & -6 & 0 & 0 & 0
\end{pmatrix}, \quad 
\wt{B}_3 = \begin{pmatrix}
0 & 0 & -2 & 1 \\
0 & 0 & 1 & -2 \\
0 & 0 & 0 & 1 \\
2 & -1 & 0 & 0 \\
-1 & 2 & 0 & 0 \\
0 & -1 & 0 & 0 
\end{pmatrix}, \quad L_3\wt{B}_3 = \begin{pmatrix}
2 & 0 & 0 & 0 \\
0 & 2 & 0 & 0 \\
0 & 0 & 0 & 0 \\
0 & 0 & 2 & 0 \\
0 & 0 & 0 & 2 \\
0 & 0 & 0 & 0
\end{pmatrix}.
\]
Here the rows of $\wt{B}_3$ and $L_3$ correspond in order to the indices $(1,0)$, $(2,0)$, $(3,0)$, $(1,1)$, $(2,1)$, $(3,1)$. Note that $\wt{B}_3$ is the adjacency matrix of the quiver $\wt{Q}_3$ of Figure \ref{fig1}. 
\end{Example}

\begin{Proposition}\label{prop:initialcluster}
The pair $(\{\cP_{k,\ell}\}_{(k,\ell) \in I}, \wt{B}_n)$ is a quantum monoidal seed with coefficient matrix $L_n$, and admits a mutation in all exchangeable directions. Its mutation in direction $(k,1)$ is
$(\{\cP_{k,\ell}\}_{(k,\ell) \in I \smallsetminus \{(k,1)\}} \sqcup \{\cP_{k,-1}\}, \mu_{(k,1)}(\wt{B}_n))$, and its mutation in direction $(k,0)$ is
$(\{\cP_{k,\ell}\}_{(k,\ell) \in I \smallsetminus \{(k,0)\}}\sqcup \{\cP_{k,2}\}, \mu_{(k,0)}(\wt{B}_n))$. 
\end{Proposition}
\begin{proof}
We know that the objects $\{\cP_{k,\ell}\}_{(k,\ell) \in I}$ are simple and real (c.f. Lemma~\ref{lem:real}). The fact that they $q$-commute with powers of $q$ prescribed by $L_n$ is the content of Proposition~\ref{prop:comm}. It is clear from Proposition~\ref{prop:leadingterms} that the normalized monomials in $\{\cP_{k,\ell}\}_{(k,\ell) \in I}$ are pairwise nonisomorphic.

Consider the first of the stated mutations, the other being similar. We need to check that the first exact sequence of Proposition~\ref{prop:mutation} coincides with the sequence in (\ref{eq:qseedseq}). The agreement of the outer terms follows by inspection of $\wt{B}_n$. The agreement of the grading shifts of the middle terms follows from
\begin{align*} \frac12 \Lambda(\cP_{k,1}, \cP_{k,-1}) = \frac12 \Lambda(\cP_{k,1}, \cP_{k,1} \hconv \cP_{k,-1}) = \frac12 \Lambda(\cP_{k,1}, \cP_{k,0}^2) = 2k.
\end{align*}
Finally, given that the matrices $L_n$ and $\wt{B}_n$ are easily checked to be compatible, it follows from Proposition~\ref{prop:KKKO6.2a} that $(\{\cP_{k,\ell}\}_{(k,\ell) \in I \smallsetminus \{(k,1)\}}\sqcup \{\cP_{k,-1}\}, \mu_{(k,1)}(\wt{B}_n))$ is indeed a quantum monoidal seed.
\end{proof}

To show that $(\{\cP_{k,\ell}\}_{(k,\ell) \in I}, \wt{B}_n)$ admits arbitrary mutation sequences using Theorem \ref{thm:K3Omain}, we must show that the classes $\{[\cP_{k,\ell}]\}_{(k,\ell) \in I}$ induce an embedding $\cA_{(L, \wt{B})} \into K^{GL_n(\cO) \rtimes \G_m}(\Gr_{GL_n})$ (at least after extending scalars to $\Q(q^{1/2})$). We do this by interpreting $\cA_{(L_n, \wt{B}_n)}$ as the quantized coordinate ring of a unipotent subgroup of $LSL_2$. This applies the main results of \cite{GLS13}, which we briefly recall.

Let $A = (a_{ij})$ be a symmetrizable Cartan matrix, $P$ its weight lattice, $w$ an element of its Weyl group, and $\mb{i} = (i_1, \dotsc, i_r)$ a list of indices such that $s_{i_1}\cdots s_{i_r}$ is a reduced word for $w$. We associate to this data a compatible pair $(L(\mb{i}) = (\lmat(\mb{i})_{k \ell}), \wt{B}(\mb{i}) = (b(\mb{i})_{k \ell}))$ as follows. For $1 \leq k \leq r$ we write
\begin{gather*}
k^- := \max(\{0\} \cup \{1 \leq \ell < k | i_k = i_\ell\})\\
k^+ := \min(\{k < \ell \leq r| i_k = i_\ell\} \cup \{r+1\})
\end{gather*}
and set $\beta_k := \omega_{i_k} - w_{\leq k} \omega_{i_k}$, where $w_{\leq k} := s_{i_1} \cdots s_{i_k}$. 
We then set
\[
b(\mb{i})_{k \ell} = \begin{cases} a_{i_k i_\ell} & k < \ell <k^+ \leq \ell^+ \\
1 & k = \ell^- \\
0 & \text{else,} \end{cases}
\quad
%\lmat(\mb{i})_{k \ell} = (\beta_k| 2\omega_{i_\ell} - \beta_{\ell}),
\lmat(\mb{i})_{k \ell} = \begin{cases} (\beta_k| 2\omega_{i_\ell} - \beta_{\ell}) & k < \ell\\
-(\beta_{\ell}| 2\omega_{i_k} - \beta_{k}) & k > \ell \\ 0 & k=\ell,\end{cases}
\]
where $(-|-): P \times P \to \Z$ denotes a $W$-invariant symmetric form with $(\alpha_i|\alpha_j) = A_{ij}$.

We let $A_q(\fn(w))$ denote the quantized coordinate ring of the unipotent group $N_+(w)$. We will be interested in the quantum unipotent minors $D[b,d] := D_{w_{\leq b} \omega_b, w_{\leq d} \omega_d}$, where $1 \leq b \leq d \leq r$ with $i_b = i_d$. 

\begin{Theorem}\label{thm:GLSunipotent}\cite{GLS13} 
There is a $\Q(q^{1/2})$-algebra isomorphism 
$$\Q(q^{1/2}) \otimes_{\Z[q^{\pm 1}]} A_q(\fn(w)) \cong \Q(q^{1/2}) \otimes_{\Z[q^{\pm1 /2}]} \cA_{(L(\mb{i}), \wt{B}(\mb{i}))}$$ 
which takes the minors $\{D[b,d]\}_{1 \leq b \leq d \leq r}$ to quantum cluster variables rescaled by powers of~$q^{1/2}$. Since $\Q(q^{1/2}) \otimes_{\Z[q^{\pm 1}]} A_q(\fn(w))$ is an iterated skew polynomial ring in the generators $\{D[b,b]\}_{1 \leq b \leq r}$, it follows that $\Q(q^{1/2}) \otimes_{\Z[q^{\pm1 /2}]} \cA_{(L(\mb{i}), \wt{B}(\mb{i}))}$ is generated by the images of these minors.
\end{Theorem}

From now on we fix $\mb{i} = (0,1,\dotsc,0,1)$ to be the indices of the unique reduced word of the element $(s_0s_1)^n$ of the Weyl group of type $A_1^{(1)}$. It is easy to check that $L({\mb{i}})$ is as follows.
%Let us compute the matrix $L({\mb{i}})$ explicitly.

\begin{Lemma}\label{lem:minorslambda}
For $1 \leq k \leq n$ we have
\[
\beta_{2k} = k(k+1) \alpha_0 + k^2 \alpha_1, \quad \beta_{2k-1} = k^2 \alpha_0 + k(k-1) \alpha_1.
\]
For $1 \leq k, l \leq n$ we have
\begin{gather*}
( \beta_{2k} | 2\omega_1 - \beta_{2\ell} ) = 2k(k-\ell), \quad ( \beta_{2k} | 2\omega_0 - \beta_{2\ell - 1} ) = 2k(k + 1 -\ell)\\
( \beta_{2k-1} | 2\omega_1 - \beta_{2\ell} ) = 2k(k-\ell-1), \quad ( \beta_{2k-1} | 2\omega_0 - \beta_{2\ell - 1} ) = 2k(k-\ell).
\end{gather*}
\end{Lemma}
\begin{comment}
\begin{proof}
The case $\beta_1 = \alpha_0$ of the first claim is immediate. Write $\beta_k^{op}$ for $\omega_{j_k} - ((s_1s_0)^n)_{\leq k} \omega_{j_k}$, where $\mb{j} = (1,0,\dotsc,1,0)$. If $\beta_{2k-1} = k^2 \alpha_0 + k(k-1) \alpha_1$ holds for some $k$, then we also have $\beta_{2k-1}^{op} = k(k-1) \alpha_0 + k^2 \alpha_1$. We then have
\begin{gather*}
  \omega_1 - \beta_{2k} = s_0(\omega_1 - \beta_{2k-1}^{op}) 
                         = s_0(\omega_1 - (k^2 \alpha_0 + k(k-1) \alpha_1)) 
  = \omega_1 - (k(k+1) \alpha_0 + k^2 \alpha_1).
\end{gather*}
One checks the same way that if the given formula holds for $\beta_{2k}$ it also holds for $\beta_{2k+1}$.

The second claim can now be calculated directly; for example, we have
\begin{align*}
  ( \beta_{2k} | 2\omega_1 - \beta_{2\ell} ) & = ( k(k+1) \alpha_0 + k^2 \alpha_1 | 2 \omega_1 - (\ell(\ell+1) \alpha_0 + \ell^2 \alpha_1))\\
                                             & = 2k^2 - ( k \alpha_0 | \ell \alpha_0) \\
                                             & = 2k(k - \ell). \qedhere
\end{align*}
\end{proof}
\end{comment}

To identify $\cA_{(L(\mb{i}), \wt{B}(\mb{i}))}$ and $\cA_{(L_n, \wt{B}_n)}$, we define a larger set $\{X_{k,\ell}\}_{(k, \ell) \in [1,n-1] \times \Z}$ of quantum cluster variables in $\cA_{(L_n, \wt{B}_n)}$ extending $\{X_{k,\ell}\}_{(k,\ell) \in I_{ex}}$ as follows.

First set $X_{k,2}:=\mu_{(k,0)}X_{k,0}$ for any $k \in [1,n-1]$. Note that the mutations in directions $(k,0)$ for different $k$ commute with each other. The result of performing these mutations in any order on $\wt{B}_n$ is an exchange matrix whose principal part is identified with that of $\wt{B}_n$ upon permuting the index set $I_{ex}$ by the involution $(k,0) \leftrightarrow (k,1)$. For $\ell \geq 1$ we inductively define $X_{k,\ell+1}$ as the variable introduced by mutating the cluster $\{X_{k,\ell}\}_{(k, \ell) \in [1,n-1] \times [\ell-1,\ell]} \sqcup \{X_{n,0}, X_{n,1}\}$ in direction $(k,\ell-1)$. For $\ell < 0$ we define $X_{k,\ell}$ similarly, starting with $X_{k,-1}:=\mu_{(k,1)}X_{k,1}$.
  
Now let $\mu$ be the sequence of mutations of $(L_n, \wt{B}_n)$ defined inductively as the result of the following $n-1$ steps. At the beginning of the $k$th step, the associated cluster in $\cA_{(L_n, \wt{B}_n)}$ contains the $n-k$ cluster variables $X_{1,k-1}, \dotsc, X_{n-k,k-1}$, the mutations at these variables all commute, and the $k$th step consists of mutating at them in any order. 

\begin{Proposition}\label{prop:mutequiv}
The compatible pairs $(\mu(L_n), \mu(\wt{B}_n))$ and $(L(\mb{i}), \wt{B}(\mb{i}))$ coincide for a unique identification of their index sets. Under the induced isomorphism $\cA_{(L(\mb{i}), \wt{B}(\mb{i}))} \cong \cA_{(L_n, \wt{B}_n)}$ the initial cluster $\{X_i\}_{i \in [1,2n]}$ of $\cA_{(L(\mb{i}), \wt{B}(\mb{i}))}$ is identified with the cluster $\{X_{k,\ell}\}_{k +\ell \in [n,n+1]}$ via
\[
X_{2k} \leftrightarrow X_{k, n-k}, \quad X_{2k-1} \leftrightarrow X_{k, n - k + 1}.
\]
\end{Proposition}
\begin{proof}
It is easy to check pictorially using quivers that $\mu(\wt{B}_n)$ is equal to $\wt{B}(\mb{i})$ via the indicated correspondence of indices (for example, as illustrated in Figure \ref{fig:mutseq}). To show that $\mu(L_n)$ is equal to $L(\mb{i})$, we must show that the exponents appearing in the $q$-commutation relations of the cluster $\{X_i\}_{i \in [1,2n]}$ in $\cA_{(L(\mb{i}), \wt{B}(\mb{i}))}$ coincide with those of the cluster $\{X_{k,\ell}\}_{k +\ell \in [n,n+1]}$ in $\cA_{(L_n, \wt{B}_n)}$. Given $q$-commuting elements $X$, $Y$ of a $\Z[q^{\pm 1/2}]$-algebra, we let $\Lambda(X, Y)$ denote the exponent appearing in the identity $q^{\Lambda(X,Y)}XY = YX$.

Recall that the product $L_n \wt{B}_n$ is invariant under mutation. Since $L_n \wt{B}_n = L(\mb{i}) \wt{B}(\mb{i})$ and the principal part of $\wt{B}(\mb{i})$ has full rank, it suffices to show that the $I \times I_{fr}$ submatrices of $\mu(L_n)$ and $L(\mb{i})$ are equal. Thus we only need to check that for $1 \leq k \leq n$,
  \begin{gather*}
\lmat(\mb{i})_{2k,2n} = -\Lambda(X_{k, n-k}, X_{n,0}), \quad \lmat(\mb{i})_{2k,2n-1} = -\Lambda(X_{k, n-k}, X_{n,1})\\
\lmat(\mb{i})_{2k-1,2n} = -\Lambda(X_{k, n-k+1}, X_{n,0}), \quad \lmat(\mb{i})_{2k-1,2n-1} = -\Lambda(X_{k, n-k+1}, X_{n, 1}).
\end{gather*}

Since for any $X_{k, m} \in \cA_{(L_n, \wt{B}_n)}$ we have a relation of the form
\[
X_{k,m-1}X_{k,m+1} = X_{k,m}^2 + \text{another monomial},
\]
these $q$-commutation exponents satisfy the recurrence
\[
\Lambda(X_{k,m+1}, X_{n,0}) = 2 \Lambda(X_{k,m}, X_{n,0}) - \Lambda(X_{k, m-1}, X_{n,0}),
\]
as well as the recurrence obtained by writing $X_{n,1}$ in place of $X_{n,0}$.
Starting from the values of $\Lambda(X_{k,m}, X_{n,0})$ and $\Lambda(X_{k,m}, X_{n,1})$ for $m \in \{0, 1\}$ recorded in the matrix $L$, it follows by induction that
\[
\Lambda(X_{k, m}, X_{n,0}) = 2km, \quad \Lambda(X_{k, m}, X_{n,1}) = 2k(m-1).
\]
One can now check directly using Lemma \ref{lem:minorslambda} that the remaining entries of $\mu(L_n)$ and $L(\mb{i})$ agree. For example, by the Lemma and the above computation we obtain
\[
\Lambda(X_{2k}, X_{2n}) = 2k(n-k) = \Lambda(X_{k, n-k}, X_{n,0})
\]
as desired.
\end{proof}

\begin{figure}
\centering
\begin{tikzpicture}[thick,>=\arrtip]
\newcommand*{\Ddotsdist}{2}
\newcommand*{\shft}{1}
\newcommand*{\DrawDots}[1]{
  \fill ($(#1) + .25*(\Ddotsdist,0)$) circle (.03);
  \fill ($(#1) + .5*(\Ddotsdist,0)$) circle (.03);
  \fill ($(#1) + .75*(\Ddotsdist,0)$) circle (.03);
}
\node [matrix] (Q1) at (0,-7)
{
\coordinate [label={[label distance=1mm]above:{$(1,3)$}}] (1) at (6,4);
\coordinate [label={[label distance=1mm]above:{$(1,2)$}}] (2) at (4,4);
\coordinate [label={[label distance=0mm]-45:{$(2,2)$}}] (3) at (4,2);
\coordinate [label={[label distance=0mm]135:{$(2,1)$}}] (4) at (2,2);
\coordinate [label={[label distance=1mm]below:{$(3,1)$}}] (5) at (2,0);
\coordinate [label={[label distance=1mm]below:{$(3,0)$}}] (6) at (0,0);

\foreach \v in {1,2,3,4,5,6} {\fill (\v) circle (.06);};
%\foreach \v in {2n+1,2n+2} {\draw ($(\v)+(-.05,-.05)$) rectangle ($(\v)+(.05,.05)$);};
\foreach \s/\t in {2/3, 4/5} {
  \draw [->,shorten <=1.7mm,shorten >=1.7mm] ($(\s)+(0.06,0)$) to ($(\t)+(0.06,0)$);
  \draw [->,shorten <=1.7mm,shorten >=1.7mm] ($(\s)-(0.06,0)$) to ($(\t)-(0.06,0)$);
};
\foreach \s/\t in {1/2,  3/4} {
  \draw [->,shorten <=1.7mm,shorten >=1.7mm] ($(\s)+(0,0.06)$) to ($(\t)+(0,0.06)$);
  \draw [->,shorten <=1.7mm,shorten >=1.7mm] ($(\s)-(0,0.06)$) to ($(\t)-(0,0.06)$);
};
\foreach \s/\t in {4/2, 3/1, 6/4, 5/3} {
  \draw [->,shorten <=1.7mm,shorten >=1.7mm] ($(\s)$) to ($(\t)$);
};\\
};

\node [matrix] (Q2) at (7,-7)
{
\coordinate [label={[label distance=1mm]above:{$(1,1)$}}] (1) at (2,4);
\coordinate [label={[label distance=1mm]above:{$(1,2)$}}] (2) at (4,4);
\coordinate [label={[label distance=0mm]right:{$(2,2)$}}] (3) at (4,2);
\coordinate [label={[label distance=0mm]135:{$(2,1)$}}] (4) at (2,2);
\coordinate [label={[label distance=1mm]below:{$(3,1)$}}] (5) at (2,0);
\coordinate [label={[label distance=1mm]below:{$(3,0)$}}] (6) at (0,0);

\foreach \v in {1,2,3,4,5,6} {\fill (\v) circle (.06);};
% up/down doubled arrows
\foreach \s/\t in {4/5} {
  \draw [->,shorten <=1.7mm,shorten >=1.7mm] ($(\s)+(0.06,0)$) to ($(\t)+(0.06,0)$);
  \draw [->,shorten <=1.7mm,shorten >=1.7mm] ($(\s)-(0.06,0)$) to ($(\t)-(0.06,0)$);
};
% left/right doubled arrows
\foreach \s/\t in {2/1,  3/4} {
  \draw [->,shorten <=1.7mm,shorten >=1.7mm] ($(\s)+(0,0.06)$) to ($(\t)+(0,0.06)$);
  \draw [->,shorten <=1.7mm,shorten >=1.7mm] ($(\s)-(0,0.06)$) to ($(\t)-(0,0.06)$);
};
% single arrows
\foreach \s/\t in {4/2, 1/3, 6/4, 5/3} {
  \draw [->,shorten <=1.7mm,shorten >=1.7mm] ($(\s)$) to ($(\t)$);
};\\
};

\node [matrix] (Q3) at (7,0)
{
\coordinate [label={[label distance=1mm]above:{$(1,1)$}}] (1) at (2,4);
\coordinate [label={[label distance=1mm]above:{$(1,2)$}}] (2) at (4,4);
\coordinate [label={[label distance=0mm]left:{$(2,0)$}}] (3) at (0,2);
\coordinate [label={[label distance=0mm]-45:{$(2,1)$}}] (4) at (2,2);
\coordinate [label={[label distance=1mm]below:{$(3,1)$}}] (5) at (2,0);
\coordinate [label={[label distance=1mm]below:{$(3,0)$}}] (6) at (0,0);

\foreach \v in {1,2,3,4,5,6} {\fill (\v) circle (.06);};
% up/down doubled arrows
\foreach \s/\t in {1/4} {
  \draw [->,shorten <=1.7mm,shorten >=1.7mm] ($(\s)+(0.06,0)$) to ($(\t)+(0.06,0)$);
  \draw [->,shorten <=1.7mm,shorten >=1.7mm] ($(\s)-(0.06,0)$) to ($(\t)-(0.06,0)$);
};
% left/right doubled arrows
\foreach \s/\t in {2/1,  4/3} {
  \draw [->,shorten <=1.7mm,shorten >=1.7mm] ($(\s)+(0,0.06)$) to ($(\t)+(0,0.06)$);
  \draw [->,shorten <=1.7mm,shorten >=1.7mm] ($(\s)-(0,0.06)$) to ($(\t)-(0,0.06)$);
};
% single arrows
\foreach \s/\t in {4/2, 3/1, 6/4, 3/5} {
  \draw [->,shorten <=1.7mm,shorten >=1.7mm] ($(\s)$) to ($(\t)$);
};\\
};

\node [matrix] (Q4) at (0,-0)
{
\coordinate [label={[label distance=1mm]above:{$(1,1)$}}] (1) at (2,4);
\coordinate [label={[label distance=1mm]above:{$(1,0)$}}] (2) at (0,4);
\coordinate [label={[label distance=0mm]left:{$(2,0)$}}] (3) at (0,2);
\coordinate [label={[label distance=0mm]right:{$(2,1)$}}] (4) at (2,2);
\coordinate [label={[label distance=1mm]below:{$(3,1)$}}] (5) at (2,0);
\coordinate [label={[label distance=1mm]below:{$(3,0)$}}] (6) at (0,0);

\foreach \v in {1,2,3,4,5,6} {\fill (\v) circle (.06);};
% left/right doubled arrows
\foreach \s/\t in {1/2,  4/3} {
  \draw [->,shorten <=1.7mm,shorten >=1.7mm] ($(\s)+(0,0.06)$) to ($(\t)+(0,0.06)$);
  \draw [->,shorten <=1.7mm,shorten >=1.7mm] ($(\s)-(0,0.06)$) to ($(\t)-(0,0.06)$);
};
% single arrows
\foreach \s/\t in {2/4, 3/1, 6/4, 3/5} {
  \draw [->,shorten <=1.7mm,shorten >=1.7mm] ($(\s)$) to ($(\t)$);
};\\
};

\draw[-implies,line width=1pt,double distance=4pt] (5,-7) -- (3,-7) node[midway,above=1.2mm] {$\mu_{(1,1)}$};
\draw[-implies,line width=1pt,double distance=4pt] (Q3) -- (Q2) node[midway,right=1mm] {$\mu_{(2,0)}$};
\draw[-implies,line width=1pt,double distance=4pt] (Q4) -- (Q3) node[pos=.4,above=1.2mm] {$\mu_{(1,0)}$};

\node (equals1) at (-3.1,.1) {$\wt{B} = $};
\node (equals4) at (-3.1,-6.9) {$\mu(\wt{B}) = $};
%\node (equals2) [right=0mm of Q2] {\quad\quad\quad};
%\node (equals3) [right=0mm of Q3] {\quad\quad\quad};

%\node (equals) [left=0mm of Q] {$\wt{Q}_n = $};
%\node (equalsr) [right=0mm of Q] {\quad\quad\quad};
\end{tikzpicture}
\caption{%The mutation sequence $\mu$ in the case $n = 3$. The resulting matrix $\mu(\wt{B})$ is equal to $\wt{B}(\mb{i})$ given the index correspondence $2k \leftrightarrow (k,n-k)$, $2k-1 \leftrightarrow (k,n-k+1).$
}\label{fig:mutseq}
\end{figure}

\begin{Example}
In the case $n=2$, $\mu$ consists of a single mutation, illustrated in Example \ref{ex:n=2mutseq}. In the case $n = 3$, $\mu$ consists of 3 mutations, illustrated in Figure \ref{fig:mutseq}.
\end{Example}

Consider now the composition
$$ \Q(q^{1/2}) \otimes_{\Z[q^{\pm 1}]} A_q(\fn((s_0s_1)^n)) \cong \Q(q^{1/2}) \otimes_{\Z[q^{\pm1 /2}]} \cA_{(L(\mb{i}), \wt{B}(\mb{i}))} \cong \Q(q^{1/2}) \otimes_{\Z[q^{\pm1 /2}]} \cA_{(L_n, \wt{B}_n)} $$
of the isomorphisms of Theorem \ref{thm:GLSunipotent} and Proposition \ref{prop:mutequiv}. Up to a power of $q^{1/2}$, the first isomorphism identifies the initial cluster variable $X_k \in \cA_{(L(\mb{i}), \wt{B}(\mb{i}))}$ with $D[k_{min}, k]$, where $k_{min} = \min\{ 1 \leq s \leq r | i_s = i_k\}$. The remaining $D[b,d]$ are obtained along the following distinguished mutation sequence of \cite{GLS11}.

For $1 \leq k \leq r$ and $j$ a row of the Cartan matrix $A$, let
\[
k[j] = |\{ 1 \leq s \leq k-1| i_s = j\}|,
\]
so for example $t_j := (r+1)[j]$ is the number of appearances of $j$ in $\mb{i}$. The distinguished sequence consists of $r$ steps, where at step $k$ we perform mutations in increasing order at the lowest $t_{i_k} - 1 - k[i_k]$ indices $s$ such that $i_s = i_k$. Up to powers of $q$, each of these mutations replaces a minor $D[b, d]$ with the minor $D[b^+, d^+]$. Transporting this mutation sequence across the second isomorphism and recalling from Proposition \ref{prop:mutequiv} that $\Q(q^{1/2}) \otimes_{\Z[q^{\pm 1}]} A_q(\fn((s_0s_1)^{-1}))$ is generated by $\{D[b,b]\}_{1 \leq b \leq r}$, we obtain the following claim.

\begin{Lemma}\label{lem:keyminors}
Up to a power of $q^{1/2}$, the composition
$$ \Q(q^{1/2}) \otimes_{\Z[q^{\pm 1}]} A_q(\fn((s_0s_1)^n)) \cong \Q(q^{1/2}) \otimes_{\Z[q^{\pm1 /2}]} \cA_{(L(\mb{i}), \wt{B}(\mb{i}))} \cong \Q(q^{1/2}) \otimes_{\Z[q^{\pm1 /2}]} \cA_{(L_n, \wt{B}_n)} $$
identifies the unipotent minor $D[b,d]$ with the quantum cluster variable $X_{k, \ell}$ via the correspondences
\[
k \leftrightarrow 1 + \frac12 (d - b), \quad \ell \leftrightarrow n + 1 - \frac12 (b + d).
\]
In particular, $\Q(q^{1/2}) \otimes_{\Z[q^{\pm 1/2}]} \cA_{(L_n, \wt{B}_n)}$ is generated by the quantum cluster variables $\{X_{1,\ell}\}_{\ell \in [1-n, n]}$.
\end{Lemma}

\begin{figure}
\centering
\begin{tikzpicture}[thick,>=\arrtip]
\newcommand*{\Ddotsdist}{2}
\newcommand*{\shft}{1}
\newcommand*{\DrawDots}[1]{
  \fill ($(#1) + .25*(\Ddotsdist,0)$) circle (.03);
  \fill ($(#1) + .5*(\Ddotsdist,0)$) circle (.03);
  \fill ($(#1) + .75*(\Ddotsdist,0)$) circle (.03);
}
\node (a) at (-4,0) {$X_{1,-2}$};
\node (b) at (-2,0) {$X_{1,-1}$};
\node (c) at (0,0) {$X_{1,0}$};
\node (d) at (2,0) {$X_{1,1}$};
\node (e) at (4,0) {$X_{1,2}$};
\node (f) at (6,0) {$X_{1,3}$};
\node (g) at (-2,-2) {$X_{2,-1}$};
\node (h) at (0,-2) {$X_{2,0}$};
\node (i) at (2,-2) {$X_{2,1}$};
\node (j) at (4,-2) {$X_{2,2}$};
\node (k) at (0,-4) {$X_{3,0}$};
\node (l) at (2,-4) {$X_{3,1}$};
\node (m) at (-6,0) {$X_{1,-3}$};
\node (n) at (-4,-2) {$X_{2,-2}$};
\node (o) at (-2,-4) {$X_{3,-1}$};
\node (p) at (0,-6) {$X_{4,0}$};
\node (q) at (8,0) {$X_{1,4}$};
\node (r) at (6,-2) {$X_{2,3}$};
\node (s) at (4,-4) {$X_{3,2}$};
\node (t) at (2,-6) {$X_{4,1}$};

\draw[ultra thick, blue, rounded corners=1ex]  (-.5,-6.5) -- (2.5,-6.5) -- (8.8,-.2)  -- (8.8,.5) -- (5.5,.5) -- (-1,-6) -- (-1,-6.1);
\draw[ultra thick, blue, rounded corners=1ex]  (-.5,-6.5) -- (-6.8,-.2) -- (-6.8,.5) -- (-3.5,.5) -- (3,-6) -- (3,-6.1);
\draw[ultra thick, blue, rounded corners=2ex] (-1,-6.5) -- (3,-6.5) -- (3,.5) -- (-1,.5) -- cycle;
\draw[ultra thick, red, rounded corners=2ex] (-7,-.4) -- (-7,.7) -- (9,.7) -- (9,-.4) -- cycle;

\foreach \s/\t in {d/c, i/h, l/k} {
  \draw [->,shorten <=6mm,shorten >=6mm] ($(\s)+(0,0.06)$) to ($(\t)+(0,0.06)$);
  \draw [->,shorten <=6mm,shorten >=6mm] ($(\s)-(0,0.06)$) to ($(\t)-(0,0.06)$);
};
\foreach \s/\t in {h/d, k/i, p/l} {
  \draw [->,shorten <=4mm,shorten >=4mm] ($(\s)$) to ($(\t)$);
};
\foreach \s/\t in {h/l, c/i, k/t} {
  \draw [->,shorten <=4mm,shorten >=4mm] ($(\s)$) to ($(\t)$);
};

\end{tikzpicture}
\caption{}\label{fig:trapezoid}
\end{figure}

\begin{Example}
In Figure \ref{fig:trapezoid} we diagram the quantum cluster variables of $\cA_{(L_4, \wt{B}_4)}$ corresponding to the unipotent minors $D[b,d]$ for $n=4$. The collections circled in blue are clusters. The middle is our initial cluster with its quiver superimposed. The right is the result of the mutation sequence $\mu$, identified with the initial cluster of $\cA_{(L(\mb{i}), \wt{B}(\mb{i}))}$ by Proposition \ref{prop:mutequiv}. The left is the result of applying the distinguished mutation sequence of \cite{GLS11} to the right cluster, and as discussed in the next section its exchangeable variables are related to those of the right cluster by the quantum twist. The collection circled in red is the generating set corresponding to the dual PBW generators of $A_q(\fn((s_0s_1)^n))$. 
\end{Example}

We now return to the quantum monoidal seed $(\{\cP_{k,\ell}\}_{(k,\ell) \in I}, \wt{B}_n)$ from the beginning of the section. Recall that we denote by $\cR_q(L_n) \subset \cT_q(L_n)$ the quantum affine space with $q$-commutation relations given by $L_n$, and that the classes of the $\{\cP_{k,\ell}\}_{(k,\ell) \in I}$ generate an embedding of $\cR_q(L_n)$ into $K^{GL_n(\cO) \rtimes \Gm}(\Gr_{GL_n})$.

\begin{Lemma} \label{lem:Ore}
For any reductive $G$ we can embed $K^{G(\cO) \rtimes \Gm}(\Gr_G)$ into a division ring.  
\end{Lemma}
\begin{proof}
Since $K^{G(\O) \rtimes \Gm}(\Gr_G)$ is free as a $\Z[q^{\pm 1/2}]$-module it embeds into $\Q(q^{1/2}) \otimes_{\Z[q^{\pm 1/2}]} K^{G(\O) \rtimes \Gm}(\Gr_G)$, hence it suffices to show the claim for $\Q(q^{1/2}) \otimes_{\Z[q^{\pm 1/2}]} K^{G(\O) \rtimes \Gm}(\Gr_G)$. Consider the filtration of $\Q(q^{1/2}) \otimes_{\Z[q^{\pm 1/2}]} K^{G(\O) \rtimes \Gm}(\Gr_G)$ given by dimension of support. It suffices to show the associated graded algebra is an Ore domain \cite{Coh61}. 

We can now define a filtration on $\gr \,\Q(q^{1/2}) \otimes_{\Z[q^{\pm 1/2}]} K^{G(\cO) \rtimes \Gm}(\Gr_G)$ where the $r$th filtered piece is spanned over $\Q(q^{1/2})$ by the $\cP_{\l^\vee\!,\,\mu}$ with $\lVert \l^\vee \rVert^2 + \lVert \mu \rVert^2 \leq r^2$ for an arbitrary fixed $W$-invariant quadratic form $\lVert \cdot \rVert^2$. By Proposition \ref{prop:leadingterms} multiplication respects this filtration. Moreover, as a $\Q(q^{1/2})$-algebra filtration it is of polynomial growth, which implies that $\gr \,\Q(q^{1/2}) \otimes_{\Z[q^{\pm 1/2}]} K^{G(\cO) \rtimes \Gm}(\Gr_G)$ is Ore (c.f. \cite[Proposition 11.1]{BZ05}).
\end{proof}

\begin{Proposition}\label{prop:K0containment}
The embedding $\cR_q(L_n) \into K^{GL_n(\cO) \rtimes \Gm}(\Gr_{GL_n})$ generated by the classes $\{[\cP_{k,\ell}]\}_{(k,\ell) \in I}$ extends to an embedding
\begin{equation}\label{eq:qrationalembed}
\Q(q^{1/2}) \otimes_{\Z[q^{\pm 1/2}]} \cA^{loc}_{(L_n, \wt{B}_n)} \into \Q(q^{1/2}) \otimes_{\Z[q^{\pm 1/2}]} K^{GL_n(\cO) \rtimes \Gm}(\Gr_{GL_n})
\end{equation}
of $\Q(q^{1/2})$-algebras. This embedding has the property that if $k + \ell \leq n+1$ and $k - \ell \leq n$, $[\cP_{k,\ell}]$ maps to the quantum cluster variable $X_{k,\ell}$.
\end{Proposition}

\begin{proof}
Following Lemma \ref{lem:Ore}, let $\cF$ denote a division ring containing $K^{GL_n(\cO) \rtimes \Gm}(\Gr_{GL_n})$. It follows that the embedding $\cR_q(L_n) \into K^{GL_n(\cO) \rtimes \Gm}(\Gr_{GL_n})$ extends to an embedding $\cF_q(L_n) \into \cF$, where $\cF_q(L_n)$ is the division ring of fractions of $\cR_q(L_n)$.

The exact sequences of Proposition \ref{prop:mutation} imply that for the indicated values of $k, \ell$ the classes $[\cP_{k,\ell}] \in K^{GL_n(\cO) \rtimes \Gm}(\Gr_{GL_n}) \subset \cF$ are contained in $\cF_q(L_n) \subset \cF$. As in the proof of Proposition~\ref{prop:initialcluster}, comparison with the quantum exchange relations of the pair $(L_n, \wt{B}_n)$ shows that $[\cP_{k,\ell}]$ coincides with the quantum cluster variable $X_{k,\ell} \in \cA^{loc}_{(L_n, \wt{B}_n)} \subset \cF_q(L_n)$ (the indicated values of $k$, $\ell$ are exactly those such that $[\cP_{k,\ell}]$ can be computed using only instances of the exact sequences not involving $\cP_{n,\ell}$ for $\ell \notin \{0,1\}$). It follows that $K^{GL_n(\cO) \rtimes \Gm}(\Gr_{GL_n}) \subset \cF$ contains the subalgebra of $\cA_{(L_n, \wt{B}_n)} \subset \cF_q(L_n)$ generated by these quantum cluster variables. But by Lemma~\ref{lem:keyminors} $\Q(q^{1/2}) \otimes_{\Z[q^{\pm 1/2}]} \cA_{(L_n, \wt{B}_n)}$ is generated as a $\Q(q^{1/2})$-algebra by the subset $\{X_{1,\ell}\}_{\ell \in [1-n, n]}$ of these variables. Since the frozen quantum cluster variables are the classes of the invertible objects $\cP_{n,0}$ and $\cP_{n,1}$, the claim follows.
\end{proof}

\begin{Corollary}\label{cor:moncatGLn}
The quantum monoidal seed $(\{\cP_{k,\ell}\}_{(k,\ell) \in I}, \wt{B}_n)$ admits arbitrary sequences of mutations in $\cP_{coh}^{GL_n(\O) \rtimes \Gm}(\Gr_{GL_n})$.
\end{Corollary}
\begin{proof}
Follows immediately by combining Proposition \ref{prop:K0containment} with Theorems \ref{thm:K3Omain} and \ref{thm:satakermats}.
\end{proof}

As a consequence, we see that the images of all quantum cluster variables under the embedding (\ref{eq:qrationalembed}) are classes of simple perverse coherent sheaves. We immediately obtain the following.

\begin{Corollary}\label{cor:integralembedding}
The embedding (\ref{eq:qrationalembed}) is induced from an embedding
\begin{equation}\label{eq:qintembedding}
\cA^{loc}_{(L_n, \wt{B}_n)} \into K^{GL_n(\cO) \rtimes \Gm}(\Gr_{GL_n})
\end{equation}
of $\Z[q^{\pm 1/2}]$-algebras.
\end{Corollary}

To conclude the proof that $\cP_{coh}^{GL_n(\O) \rtimes \Gm}(\Gr_{GL_n})$ is a monoidal cluster categorification it remains to show that (\ref{eq:qintembedding}) is an isomorphism.

\begin{Remark}
The existence of a Poisson isomorphism $\C \otimes \cA^{\,loc}_{\wt{B}(\mb{i})} \cong \C \otimes K^{GL_n(\cO)}(\Gr_{GL_n})$, where $\cA^{\,loc}_{\wt{B}(\mb{i})}$ is equipped with the Poisson structure induced by $L(\mb{i})$, is a special case of a result of \cite{FKRD15}. It passes through the complexified Poisson version of the isomorphism $\Q(q^{1/2}) \otimes_{\Z[q^{\pm 1/2}]} \cA_{(L(\mb{i}), \wt{B}(\mb{i}))} \cong \Q(q^{1/2}) \otimes_{\Z[q^{\pm 1}]} A_q(\fn((s_0s_1)^n))$ we have been using. Given Theorem~\ref{thm:mainfinal} it would follow from a proof of \cite[Conjecture 12.7]{GLS13} that $K^{GL_n(\cO)}(\Gr_{GL_n})$ and $\Z[q^{\pm 1/2}] \otimes_{\Z[q^{\pm 1}]} A_q(\fn((s_0s_1)^n))$ are isomorphic as $\Z[q^{\pm 1/2}]$ algebras.
\end{Remark}

Given Corollary \ref{cor:K0generators}, it suffices to show for any $\ell \in \Z$ that $[\cP_{1,\ell}]$ is contained in the image of $\cA^{loc}_{(L_n, \wt{B}_n)}$, since $[\cP_{n,0}]^{-1}$ manifestly is. In fact, for any $k \in [1,n-1]$ we will show that $[\cP_{k,\ell}]$ is the product of the quantum cluster variable $X_{k,\ell}$ and a Laurent monomial in the frozen variables, up to a power of $q$. Note that the identity implied upon specializing $q$ and the frozen variables to $1$ follows trivially from the exact sequences of Proposition \ref{prop:mutation} and the definition of the $X_{k,\ell}$ -- the necessary work amounts to bookkeeping of frozen variables and powers of $q$.

Recall the quantum monoidal seed $(\{X_{k,\ell}\}_{k +\ell \in [n,n+1]}, \mu(\wt{B}_n))$ resulting from the mutation sequence $\mu$ used in Proposition \ref{prop:mutequiv}. Consider now the following infinite mutation sequence $\mu'$ which starts at this seed.

We define $\mu'$ inductively as a series of steps consisting of $n-1$ mutations each. At the $j$th step we mutate in order at $X_{n-1,j}, X_{n-2, j+1}, \dotsc, X_{2, j+n-3}, X_{1, j+n-2}$. The unfrozen variables in the cluster resulting from the $j$th iteration are $\{X_{k,\ell}\}_{k +\ell \in [n+j,n+j+1]}.$%\{ X_{n-1,k+1}, X_{n-1,k+2}, \dotsc, X_{1,k+n-1}, X_{1,k+n}\}$.

The sequence $\mu'$ is periodic on the principal parts of the exchange matrices that appear: after applying the $j$th step the resulting exchange matrix has the same principal part as $\mu(\wt{B}_n)$. The key to organizing the computation of the exchange relations we must understand is to have an inductive description of how $\mu'$ interacts with the frozen rows of these exchange matrices:

\begin{Lemma}\label{lem:frozenarrows}
Let $(\{X_{k,\ell}\}_{(k,\ell) \in S}, \wt{B}_n')$ be a seed occurring after applying a nonzero number of mutations in the sequence $\mu'$. Given $(k_1,\ell_1), (k_2, \ell_2) \in S$, we write $b'_{(k_1,\ell_1), (k_2, \ell_2)}$ for the corresponding entry of $\wt{B}_n'$. Then the nonzero entries in the frozen rows of $\wt{B}_n'$ are exactly those prescribed by the following conditions.
\begin{enumerate}
\item If $(n-1,j)$ and $(n-1,j+1)$ are both in $S$ for some $j \in \Z$, then $b'_{(n,1),(n-1,j+1)} = -j$ and $b'_{(n,1),(n-1,j)} = j+1$.
%in $\mu_{k,\ell} \circ \cdots \circ \mu_{n-1,1}(\wt{Q})$ there are $j$ arrows from $X_{n,1}$ to $X_{n-1, j+1}$ and $j+1$ arrows from $X_{n-1,j}$ to $X_{n,1}$.
\item If $(1,j)$ and $(1, j+1)$ are both in $S$ for some $j \geq n$, then $b'_{(n,0),(1,j)} = j-n$ and $b'_{(n,0),(1,j+1)} = 1 + n - j$.
% in $\mu_{k,\ell} \circ \cdots \circ \mu_{n-1,1}(\wt{Q})$ there are $j - n$ arrows from $X_{1,j}$ to $X_{n,0}$ and $j-n-1$ arrows from $X_{n,0}$ to $X_{1,j+1}$.
\item If $(j, n-j)$ and $(j+1,n-j+1)$ are both in $S$ for some $1 \leq j \leq n-2$, then $b'_{(n,0),(j,n-j)} = -1$ and $b'_{(n,0),(j+1,n-j+1)} = 1$.
\end{enumerate}
The last situation does not occur after the first $n-1$ mutations in $\mu'$. In particular, after the first $n-1$ mutations the quiver of $\wt{B}_n'$ has no arrows between the $(n,1)$th vertex and the $(k,\ell)$th vertex unless $k = n-1$, and no arrows between the $(n,0)$th vertex and the $(k,\ell)$th vertex unless $k = 1$.
\end{Lemma}
\begin{proof}
As the claims are easy to check pictorially using the relevant quivers, we omit the details. The key inductive reason the first two claims hold is the following. Suppose one has a quiver with vertices $i$, $j$, $k$, and with 2 arrows from $i$ to $j$, $m+1$ arrows from $j$ to $k$, and $m$ arrows from $k$ to $i$. Then after mutating at vertex $j$ there are 2 arrows from $j$ to $i$, $m+2$ arrows from $i$ to $k$, and $m+1$ arrows from $k$ to $j$.
\end{proof}

As before given $q$-commuting elements $X$, $Y$ of a $\Z[q^{\pm1/2}]$-algebra, we define $\Lambda(X,Y)$ by $q^{\Lambda(X,Y)}XY = YX$.
As in Section \ref{sec:monoidalseeds}, if $\cF$ and $\cG$ are objects of $\cP_{coh}^{GL_n(\O) \rtimes \Gm}(\Gr_{GL_n})$ we set $\cF \odot \cG := (\cF \conv \cG)\{- \frac12 \Lambda(\cF,\cG)\}$. Similarly, given $q$-commuting elements $X$, $Y$ of a $\Z[q^{\pm1/2}]$-algebra we write $X \odot Y := q^{\frac12 \Lambda(X,Y)}XY$

\begin{Lemma}\label{lem:XtoP}
For any $k \in [1,n-1]$, $\ell \in \Z$ we have
$$ X_{k,\ell} = \begin{cases} [\cP_{k,\ell} \odot \cP_{n,0}^{\ell+k-n-1}] & \ell + k > n-1 \\ [\cP_{k,\ell} \odot \cP_{n,1}^{k - \ell - n}] & k - \ell > n \\ [\cP_{k,\ell}] & \text{otherwise} \end{cases} $$
in $\cA^{loc}_{(L_n, \wt{B}_n)} \subset K^{GL_n(\cO) \rtimes \Gm}(\Gr_{GL_n})$.
\end{Lemma}
\begin{proof}
The last case is part of Proposition \ref{prop:K0containment}. We consider the first case explicitly -- the second follows by similar computations or by duality, noting that the exchange matrix of the cluster  $\{X_{k,\ell}\}_{k - \ell \in [n-1, n]}$ is minus that of the cluster  $\{X_{k,\ell} \}_{k + \ell \in [n, n+1]}$.

We prove the claim inductively by considering the $X_{k,\ell}$ as they appear along the mutation sequence $\mu'$, using Lemma \ref{lem:frozenarrows} to identify the relevant exchange relations and compare them with the exact sequences of Proposition \ref{prop:mutation}. The first exchange relation is
$$ q^{ \frac12 \Lambda(X_{n-1,1}, X_{n-2,2}X_{n,1}^2)} X_{n-1,1} X_{n-1,3} = q X_{n-2,2} \odot X_{n,1}^2 + X_{n-1,2}^2 \odot X_{n,0}, $$
which we compare to
$$ q^{ \frac12 \Lambda(\cP_{n-1,1}, \cP_{n-2,2}\cP_{n,2})}[\cP_{n-1,1}] [\cP_{n-1,3}] = q [\cP_{n-2,2}\odot \cP_{n,2}] + [\cP_{n-1,2}]^2. $$
Multiplying by $X_{n,0} = [\cP_{n,0}]$ and normalizing as suitable, this becomes
\begin{multline*} q^{ \frac12 \Lambda(\cP_{n-1,1}, \cP_{n-2,2}\cP_{n,2}) - \frac12 \Lambda(\cP_{n-1,3},\cP_{n,0})}[\cP_{n-1,1}] [\cP_{n-1,3} \odot \cP_{n,0}] \\= q^{1 - \frac12 \Lambda(\cP_{n-2,2} \odot \cP_{n,2}, \cP_{n,0})} [\cP_{n-2,2} \odot \cP_{n,2} \odot \cP_{n,0}] + q^{-\frac12 \Lambda(\cP^2_{n-1,2},\cP_{n,0})}[\cP_{n-1,2}^2 \odot \cP_{n,0}].
\end{multline*}
Substituting in the quantum cluster variables already identified with the specified classes and dividing the right-hand side by $q^{-\frac12 \Lambda(\cP^2_{n-1,2},\cP_{n,0})}$, we obtain the right-hand side of the above exchange relation. Here use the fact that $\cP_{n,1}^2 \cong \cP_{n,2} \odot \cP_{n,0}$ by Remark \ref{rem:frozenidentity} and that $\Lambda(\cP_{n-2,2} \odot \cP_{n,2}, \cP_{n,0}) = \Lambda(\cP^2_{n-1,2},\cP_{n,0})$ by Remark \ref{rem:comm}. Dividing the left-hand sides by $X_{n-1,1} = [\cP_{n-1,1}]$ and inspecting powers of $q$, the identity $X_{n-1,3} = [\cP_{n-1,3} \odot \cP_{n,0}]$ now follows since
\begin{align*}
  {}&\Lambda( \cP_{n-1,1}, \cP_{n-2,2} \conv \cP_{n,1}^2) - \Lambda(\cP_{n-1,1}, \cP_{n-2,2} \conv \cP_{n,2}) \\
  {}&\quad \quad \quad+ \Lambda(\cP_{n-1,3}, \cP_{n,0}) - \Lambda(\cP_{n-2,2} \conv \cP_{n,2}, \cP_{n,0}) \\
{}&\quad= \Lambda(\cP_{n-1,1}, \cP_{n,0}) + \Lambda(\cP_{n-1,3}, \cP_{n,0}) - \Lambda(\cP_{n-2,2} \conv \cP_{n,2}, \cP_{n,0}) =0%\\
%{}&\quad=0
\end{align*}

The next $n-2$ mutations in $\mu$ are of the form
\begin{align*} {}& q^{\frac12 \Lambda(X_{k,n-k}, X_{k+1, n-k+1} X_{k-1, n-k+1})} X_{k,n-k} X_{k, n-k+2} \\{}& \quad \quad = q X_{k+1, n-k+1} \odot X_{k-1, n-k+1} + X^2_{k,n-k+1} \odot X_{n,0} \end{align*}
for $k$ descending from $n-2$ to $1$. Here by convention $X_{0,\ell} := 1$ for $\ell \in \Z$. Thereafter, each of the length $n-1$ steps of $\mu$ takes a uniform form: in the $k$th step the first exchange relation is
$$ q^{\frac12 \Lambda(X_{n-1,k}, X_{n-2,k+1} X_{n,1}^{k+1})}X_{n-1,k} X_{n-1,k+2} = q X_{n-2,k+1} \odot X_{n,1}^{k+1} + X^2_{n-1,k+1}, $$
the next $n-3$ exchange relations are of the form
\begin{multline*} q^{\frac12 \Lambda(X_{n-j-1, X_{n-j, j+k+1} X_{n-j-2, j+k+1}})} X_{n-j-1, j+k} X_{n-j-1, j+k+2} \\= q X_{n-j, j+k+1} \odot X_{n-j-2, j+k+1} + X^2_{n-j-1, j+k+1} \end{multline*}
for $j$ increasing from $1$ to $n-3$, and the last exchange relation is
$$ q^{\frac12 \Lambda(X_{1,n+k-2}, X_{2, n+k-1} X_{n,0}^{k-2})}X_{1,n+k-2} X_{1,n+k} = q X_{2, n+k-1} \odot X_{n,0}^{k-2} + X_{1,n+k-1}^2. $$ With these inductive expressions in hand the claim may be checked by a finite list of calculations resembling the one above, and whose details we omit. 
\end{proof}

\begin{Theorem}\label{thm:mainfinal}
The coherent Satake category $\cP_{coh}^{GL_n(\cO) \rtimes \Gm}(\Gr_{GL_n})$ is a monoidal categorification of the quantum cluster algebra $\cA^{\,loc}_{(L_n,\wt{B}_n)}$.
\end{Theorem}
\begin{proof}
As indicated in the discussion preceding Lemma \ref{lem:frozenarrows}, given Corollary \ref{cor:K0generators} it follows from Lemma \ref{lem:XtoP} that the embedding of Corollary \ref{cor:integralembedding} is in fact an isomorphism. The result then follows since by Corollary \ref{cor:moncatGLn} all quantum cluster monomials are classes of simple perverse coherent sheaves.
\end{proof}

\begin{Remark}\label{rem:prime}
Following \cite{HL10} one anticipates that all cluster variables should be classes of prime objects, i.e. objects that cannot be expressed as a product of two objects which are not invertible. While we do not know if this property holds for all perverse coherent sheaves whose classes are quantum cluster variables, it is clear from Proposition \ref{prop:leadingterms} and Lemma \ref{lem:XtoP} that the $X_{k,\ell}$ are classes of prime objects. 
\end{Remark}

\subsection{Categorification of the bar involution}\label{sec:barinvolution}

Recall that the bar involution of a quantum cluster algebra $\cA_{(\wt{B},L)}$ is the restriction of the anti-automorphism of $\cF(L)$ that preserves $X^v$ for $v \in \Z^I$ and which exchanges $q$ and $q^{-1}$ \cite{BZ05}. We now explain how to lift the bar involution of $\cA^{loc}_{(\wt{B}_n,L_n)}$ from K-theory to the category $D_{coh}^{GL_n(\O) \rtimes \Gm}(\Gr_{GL_n})$. 

Consider the bijection $\pi_0 \Gr_{GL_n} \xrightarrow{\sim} \Z$ realized at the level of orbits by the assignment $\l^\vee \mapsto \la \l^\vee,\omega_n \ra$. Under this map the component containing $\Gr_{GL_n}^k$ is labeled by $k$. We also define $\n: P^\vee \to \Z$ by $\n(\l^\vee) := \la \l^\vee,\omega_n \ra (n - \la \l^\vee,\omega_n \ra)$.

\begin{Lemma}\label{lem:dualizingsheaf}
The canonical sheaf of $\oGr_{GL_n}^{\l^\vee}$ is isomorphic to 
\begin{equation}\label{eq:omega2}
\O_{\oGr_{GL_n}^{\l^\vee}}(-n) \otimes \det(L_0/tL_0)^{\la \l^\vee, \omega_n \ra} \{ \n(\l^\vee) - \dim \Gr_{GL_n}^{\l^\vee} \}.
\end{equation}
\end{Lemma}
\begin{proof}
If $\l^\vee = \omega^\vee_k$ the result follows by the standard calculation of the canonical sheaf of the finite Grassmannian (c.f. Equation \ref{eq:omega}). More generally, if $\l^\vee = \omega^\vee_{k_1} + \dots + \omega^\vee_{k_m}$ we have a resolution 
$$\om: \Gr^{\uk}_{GL_n} := \Gr_{GL_n}^{k_1} \ttimes \dots \ttimes \Gr_{GL_n}^{k_m} \rightarrow \oGr_{GL_n}^{\sum_i \omega_{k_i}^\vee}.$$
The canonical sheaf of $\Gr^{\uk}_{GL_n}$ is 
$$\left( \O_{\Gr^{k_1}_{GL_n}}(-n) \otimes \det(L_0/tL_0)^{k_1} \right) \boxtimes \dots \boxtimes \left( \O_{\Gr^{k_1}_{GL_n}}(-n) \otimes \det(L_{m-1}/tL_{m-1})^{k_m} \right).$$
Using that $\det(L_i/tL_i) \cong \det(L_{i-1}/tL_{i-1}) \{-2k_{i-1}\}$ repeatedly reduces this expression to 
$$\O_{\Gr^{\uk}_{GL_n}}(-n) \otimes \det(L_0/tL_0)^{\sum_i k_i} \{ - 2 \sum_{i \ne j} k_ik_j \}.$$
Applying $\om_*$ recovers (\ref{eq:omega2}). The general case of $\l$ follows by also considering the convolution with $\cP_n^{-1}$. 
\end{proof}

Notice that apart from the $\{- \dim \Gr_{GL_n}^{\l^\vee} \}$ shift the canonical sheaf is global. We capture this by defining $\bL$ to be the automorphism of $D_{coh}^{GL_n(\O) \rtimes \Gm}(\Gr_{GL_n})$ which, on the $k$th component of $\Gr_{GL_n}$, acts by tensoring with 
$$\O_{\Gr_{GL_n}}(-n) \otimes \det(L_0/tL_0)^k \{ k(n-k) \}.$$

Letting $g^T$ denote the transpose of $g \in GL_n$, we consider the involution of $\Gr_{GL_n}$ given by
$$ [g(t)] \mapsto [(g(t)^T)^{-1}]. $$
Twisting the action of $GL_n(\O) \rtimes \Gm$ by $(g(t),a) \mapsto ((g(t)^T)^{-1},a)$ this defines an involution of $D_{coh}^{GL_n(\O) \rtimes \Gm}(\Gr_{GL_n})$ which we denote by $\cF \mapsto \cF'$. It is covariant with respect to both convolution and $\Hom$, and commutes with $\cF \mapsto \cF^*$. 

\begin{Lemma}\label{lem:involutions}
The involutions $\D$, $\bL$, $\cF \mapsto \cF^*$, and $\cF \mapsto \cF'$ of $D_{coh}^{GL_n(\O) \rtimes \Gm}(\Gr_{GL_n})$ take perverse sheaves to perverse sheaves. Moreover, we have
\begin{gather*}
(\bL^{-1} \circ \D) \cP_{\l^\vee\!,\,\mu} \cong \cP_{\l^\vee, -\mu}, \quad\,\,\,
\cP_{\l^\vee\!,\,\mu}^* \cong \cP_{-\l^\vee\!,\,\mu}, \quad\,\,\,
\cP_{\l^\vee\!,\,\mu}' \cong \cP_{-\l^\vee, -\mu}.
\end{gather*}
\end{Lemma}

\begin{proof}
That $\D$ and $\cF \mapsto \cF^*$ preserve perversity has been discussed (see Corollary \ref{cor:adjoints}). The case of $\cF \mapsto \cF'$ may be treated similarly to that of $\cF \mapsto \cF^*$, and the case of $\bL$ is obvious. 

Given that these functors restrict to involutions of $\cP_{coh}^{GL_n(\O) \rtimes \Gm}(\Gr_{GL_n})$, it follows that they must permute simple perverse coherent sheaves up to isomorphism. The transformation rules for $\l^\vee$ simply record the support of the image of $\cP_{\l^\vee\!,\,\mu}$ under the relevant functor, and are clearly as stated.

In each case it then remains to identify the transformation rules for $\mu$. It suffices to identify the restriction of the image of $\cP_{\l^\vee\!,\,\mu}$ to the dense $GL_n(\cO)$-orbit in its support, which up to a cohomological degree shift is a simple equivariant vector bundle. Such a bundle is determined by the $P_{\l} \times \G_m$ action on the fiber over $[t^{\l^\vee}]$ (or $[t^{-\l^\vee}]$ in the last two cases). In each case this action is the result of an elementary computation similar to that of Proposition \ref{prop:leadingterms}. 
\end{proof}

\begin{Remark}
Note that the transformation rules in Lemma \ref{lem:involutions} depend on the precise choice of $\G_m$-equivariant structure on $\cP_{\l^\vee\!,\,\mu}$ made in Section \ref{sec:cohsatdefs}. A different choice would require an additional equivariant shift in the formulas above.
\end{Remark}

We now define an involution $\cF \mapsto \overline{\cF}$ of $D_{coh}^{GL_n(\O) \rtimes \Gm}(\Gr_{GL_n})$ as the composition
$$\overline{\cF} := \bL^{-1} \circ \D((\cF')^*) \cong \D \circ \bL((\cF')^*).$$
By Lemma \ref{lem:involutions} we have the following result.

\begin{Corollary}
The involution $\cF \mapsto \overline{\cF}$ is contravariant with respect to both convolution and $\Hom$. We have $\overline{\cF\{\ell\}} \cong \overline{\cF}\{-\ell\}$ for any $\ell \in \Z$ and $\overline{\cP_{\l^\vee\!,\,\mu}} \cong \cP_{\l^\vee\!,\,\mu}$ for any $\l^\vee \in P^\vee$, $\mu \in P$.
\end{Corollary}

It follows that $\cF \mapsto \overline{\cF}$ induces an anti-automorphism of $K^{GL_n(\cO) \rtimes \Gm}(\Gr_{GL_n})$ which exchanges $q$ and $q^{-1}$. Since the isomorphism $K^{GL_n(\cO) \rtimes \Gm}(\Gr_{GL_n}) \cong \cA^{loc}_{(\wt{B}_n,L_n)}$ identifies the initial quantum cluster variables with classes of sheaves of the form $\cP_{\omega_k^\vee, \ell \omega_k}$, we further obtain the following result.

\begin{Corollary}
The isomorphism $K^{GL_n(\cO) \rtimes \Gm}(\Gr_{GL_n}) \cong \cA^{loc}_{(\wt{B}_n,L_n)}$ identifies the anti-automorphism $[\cF] \mapsto [\overline{\cF}]$ with the bar involution of $\cA^{loc}_{(\wt{B}_n,L_n)}$.
\end{Corollary}

\begin{Remark}
The above discussion should extend with minor modification to the case of general $G$, see in particular \cite[Section 4.6]{BD} in comparison with Lemma \ref{lem:dualizingsheaf}.
\end{Remark}

\begin{Remark}
The $K$-theoretic description of the bar involution given above is formally in keeping with analogous results in the setting of the Steinberg variety and nilpotent cone \cite{Lus98,Ost00}.
\end{Remark}

\subsection{Categorification of the quantum twist}\label{sec:twist}

We now interpret the quantum twist in terms of the rigidity of $\cP_{coh}^{GL_n(\cO) \rtimes \Gm}(\Gr_{GL_n})$. Quantum twists were introduced for general double Bruhat cells in \cite{GY16}, and we follow the reformulation of \cite{KO17} in the unipotent case. We then review the implied connection to maximal green sequences and generalized DT invariants following \cite{Kel11}.

First recall a few conventions from \cite{GLS11} on preprojective algebras. The quantum unipotent minors $D[b,d] \in \cA_{(\wt{B},L)}$ recalled in the previous section are associated to representations $M[b,d]$ of the preprojective algebra $\Lambda$ of the 2-Kronecker quiver. We write $V_k$ for the representation associated to $D[k_{min},k]$; it has dimension vector $\beta_k$. The $M[b,d]$ lie in the subcategory $\cC_w$ of $\Lambda \textrm{-mod}$ attached to $w$ in \cite{BIRS09}, and $V_{2n-1}, V_{2n}$ are the projective-injectives of $\cC_w$.

Associated to any $M \in \cC_w$ are elements $X_M, Y_M \in \cA_{(\wt{B},L)}$ related by $Y_M = q^{\lVert \udim M\rVert^2/4} X_M$ \cite{KO17}, where the dimension vector $\udim M$ is considered as an element of the $A_{\,1}^{(1)}$ root lattice. If $M$ is rigid, indecomposable, and reachable from the $M[b,d]$ by a sequence of mutations, then $X_M$ is a quantum cluster variable. All quantum cluster variables are of this form. 

Given $M \in \cC_w$, we let $I(M)$ denote the injective hull of $M$ in $\cC_w$, $m_i$ the multiplicity in $I(M)$ of the injective hull of the simple whose dimension vector is the $i$th simple root, and $\Omega^{-1}_w(M)$ the last term of
$$ 0 \to M \to I(M) \to \Omega^{-1}_w(M) \to 0. $$
As formulated in \cite{KO17}, the quantum twist $\eta_w$ of $A_q(N^w)$ is the $\Z[q^{\pm1}]$-algebra automorphism characterized by
$$ \eta_w(Y_M) = q^{\sum_i m_i (\udim M)_i} Y_{I(M)}^{-1} Y_{\Omega^{-1}_w(M)}. $$
For consistency with the previous section we follow the conventions of \cite{GLS13}, which are transpose to those of \cite{KO17} as discussed in Remark 6.10 of loc. cited.
The classical limit of $\eta_w$ was introduced in \cite{BZ97}, for which the categorical description above was discovered in \cite{GLS12}. When $w$ is the square of a Coxeter element, for example as in our $n=2$ case, quantizations of the classical twist were first studied in \cite{BR15}.

\begin{Theorem}\label{thm:twist}
The isomorphism
$$\Q(q^{1/2}) \otimes_{\Z[q^{\pm 1/2}]} K^{GL_n(\cO) \rtimes \Gm}(\Gr_{GL_n})  \cong \Q(q^{1/2}) \otimes_{\Z[q^{\pm 1/2}]} \cA_{(\wt{B}_n,L_n)} \cong \Q(q^{1/2}) \otimes_{\Z[q^{\pm 1}]} A_q(N^{(s_0s_1)^n})$$
identifies the anti-automorphism $[\cF] \mapsto [\cF^L]$ with the composition $X \mapsto \overline{\eta_{(s_0s_1)^n}(X)}$ of the bar involution and the quantum twist. 
\end{Theorem}

\begin{proof}
For $j < 2n-1$ there is an exact sequence
\begin{equation}\label{eq:preprojseq}
0 \to V_j \to V_{j_{max}} \to T_j \to 0,
\end{equation}
where $j_{max} = \max\{ 1 \leq s \leq r | i_s = i_j\}$ and $T_j = M[j, j_{max}]$ \cite[Section 9.8]{GLS11}.
From the above description of $\eta_{(s_0s_1)^n}$ we obtain
$$ \eta_{(s_0s_1)^n}(q^{\lVert \udim V_j \rVert^2/4}X_{V_j}) = q^{(\udim V_{j})_{i_j}} (q^{\lVert \udim V_{j_{max}} \rVert^2/4}X_{V_{j_{max}}})^{-1} (q^{\lVert \udim T_j \rVert^2/4}X_{T_j}).$$
It follows from $(\ref{eq:preprojseq})$ and the calculation of $\beta_j$ in Lemma \ref{lem:minorslambda} that if $k$ is chosen so that $j$ is equal to $2k-1$ or $2k,$ then
\begin{gather*}
(\udim V_{j})_{i_j} = k^2, \quad \lVert \udim V_j \rVert^2 = 2k^2,\quad 
\lVert \udim V_{j_{max}} \rVert^2 = 2n^2, \quad \lVert \udim T_j \rVert^2 = 2(n-k)^2.
\end{gather*}
Simplifying and applying the bar involution we then obtain
$$ \overline{\eta_{(s_0s_1)^n}(X_{V_j})} = q^{k(n-k)}X_{T_j}X_{V_{j_{max}}}^{-1}.$$

On the other hand, by Lemma \ref{lem:keyminors} and Proposition \ref{prop:K0containment} we have $X_{T_{2k}} = [\cP_{n-k,-k}]$ and $X_{T_{2k-1}} = [\cP_{n-k,1-k}]$. By Remark \ref{rem:comm} we see that whether $j$ is $2k$ or $2k-1$ we have
$$X_{T_{j}} X_{V_{j_{max}}}^{-1} = q^{-2k(n-k)}X_{V_{j_{max}}}^{-1} X_{T_{j}}.$$
Comparing with Proposition \ref{prop:example3} we thus have
$$\overline{\eta_{(s_0s_1)^n}([\cP_{n-k,-k}])} = [\cP_{n-k,-k}^L], \quad \overline{\eta_{(s_0s_1)^n}([\cP_{n-k,1-k}])} = [\cP_{n-k,1-k}^L].$$

Now consider the frozen variables $X_{V_{2n}} = [\cP_{n,0}]$ and $X_{V_{2n-1}} = [\cP_{n,1}]$. Using the calculation of $\beta_{2n}$ and $\beta_{2n-1}$ one checks as above that $\eta_{(s_0s_1)^n}(X_{V_{2n}}) = X_{V_{2n}}^{-1}$ and $\eta_{(s_0s_1)^n}(X_{V_{2n-1}}) = X_{V_{2n-1}}^{-1}$. On the other hand, $[\cP_{n,0}^L] = [\cP_{n,0}]^{-1}$ and $[\cP_{n,1}^L] = [\cP_{n,1}]^{-1}$ as these objects are invertible.

We have thus verified that the anti-automorphisms being compared agree on the elements $X_{V_1},\cdots,X_{V_{2n}}$. Since these form a cluster they generate the division ring of fractions of $\cA_{(\wt{B},L)}$, and the claim follows.
\end{proof}

The sequence of mutations taking the cluster $\{X_{V_i}\}$ to $\{X_{T_i}\}$ is an example of a maximal green sequence \cite{Kel11}. The principally framed quiver $Q_{pr}$ of a quiver $Q$ has a new vertex $v'$ for each $v \in Q_0$ and a single arrow from $v$ to $v'$. We regard $v \in Q_0$ as mutable and the new vertex $v'$ as frozen. After some sequence of mutations of $Q_{pr}$ an unfrozen vertex $v$ is said to be green if there is no arrow from any frozen vertex to $v$, red otherwise -- a maximal green sequence is then a maximal finite sequence of mutations at green vertices.

If $\mu$ is a maximal green sequence, then $\mu(Q_{pr})$ is isomorphic to the principally co-framed quiver of $Q$ \cite{BDP14}, i.e. the quiver constructed as above but in which the arrow between $v$ and $v'$ is directed from $v'$ to $v$. This isomorphism will fix the frozen vertices but possibly permute the unfrozen vertices nontrivially. This permutation specifies an isomorphism $\mu(Q) \cong Q$, hence an automorphism of $\cA_Q$ that maps the initial cluster onto the one resulting from $\mu$. This automorphism was called the total DT invariant in \cite{Kel11} in reference to its role in \cite{KS08}, though as discussed in the introduction a variety of terms appear in the literature.

From the categorical description of $\eta_w$ it follows from \cite{KR07} that after specializing frozen variables (hence necessarily $q$) to 1, the twist coincides with the total DT invariant. Note that as treated above the total DT invariant is a priori only a notion defined for quivers without frozen variables, so we would not know how to formulate a statement relating it to the twist without this specialization. 

\begin{Corollary}\label{cor:twist}
After specializing the frozen variables and $q$ to 1, the isomorphism $\cA^{loc}_{(\wt{B}_n,L_n)} \cong K^{GL_n(\cO) \rtimes \Gm}(\Gr_{GL_n})$ identifies the automorphism $[\cF] \mapsto [\cF^L]$ with the total DT invariant of $\cA_{\wt{B}}^{loc}$.
\end{Corollary}

Finally, we note that in addition to the maximal green sequence of \cite{GLS11} which starts at the cluster $X_{T_1},\cdots X_{T_{2n-2}}, X_{V_{2n-1}}, X_{V_{2n}}$, there is a another simple maximal green sequence starting at the cluster $\{X_{k,\ell}\}_{(k,\ell) \in I}$ \cite{ACCERV14}. This consists of mutating inductively at the $X_{k,0}$ in any order, then the $X_{k,1}$ in any order and so on, until one reaches the cluster $\{X_{k,\ell}\}_{(k, \ell) \in [1,n-1] \times [n,n+1]} \sqcup \{X_{n,0}, X_{n,1}\}$.
It is this maximal green sequence which also appears in \cite{Ked}, where it describes the time evolution of the $A_n^{(1)}$ $Q$-system. %In particular, the discrete integrability of the $Q$-system \cite{DFK10} manifests in our setting as the fact that the double dual $[\cF] \mapsto [\cF^{LL}]$ acts by the identity on the symmetric monoidal subcategory $\cP_{coh}^{GL_n(\cO) \rtimes \Gm}(\Gr^0_{GL_n}) \cong \Rep(GL_n(\O) \rtimes \Gm)$ of sheaves supported on the fixed point $\Gr_{GL_n}^0$. We note that from the perspective of the preprojective algebra, this integrability corresponds to the invariance of a corresponding collection of objects in its stable category under the Auslander-Reiten translation \cite{Wil16}.

\subsection{$Q$-systems and global sections}\label{sec:Qsystem}

Recall that the $Q$-system is a recurrence relation describing characters of Kirillov-Reshetikhin (KR) modules of the Yangian of a simple Lie algebra \cite{KR87,HKOTY99}. In \cite{Ked} it was shown that the relations of the $A_{n-1}^{(1)}$ $Q$-system can be normalized to coincide with exchange relations in the cluster algebra $\cA_{B_n}$, where $B_n$ denotes the principal part of $\wt{B}_n$. In this section we explain how this emerges from our main result by passing to the global sections of the sheaves $\cP_{k,\ell}$.

First, we pass to the classical limits of KR modules, representations of the current algebra rather than the Yangian. The $Q$-system describes their characters as representations of the underlying simple Lie algebra $\fg$. For $\fg = \fsl_n$ (but not for other types) they are irreducible as $\fsl_n$-representations, and are exactly the representations $V_{\ell \omega_k}$ with highest weight $\ell \omega_k$ for $\ell \in \N$ and $\omega_k$ a fundamental weight.

\begin{Proposition}\label{prop:cohomology}
For $\ell \ge 0$ we have isomorphisms
\begin{align*}
\Gamma(\cP_{k-1,\ell} * \cP_{k+1,\ell}) & \cong V_{\ell \omega_{k-1}} \otimes V_{\ell \omega_{k+1}} [k(n-k)-1] \\
\Gamma(\cP_{k,\ell+1} * \cP_{k,\ell-1}) & \cong V_{(\ell+1) \omega_k} \otimes V_{(\ell-1) \omega_k} [k(n-k)] \\
\Gamma(\cP_{k,\ell} * \cP_{k,\ell}) & \cong V_{\ell \omega_k} \otimes V_{\ell \omega_k} [k(n-k)]
\end{align*}
of $GL_n$ representations.
\end{Proposition}
\begin{proof}
Let us begin with the first identity. The statement reduces to showing that 
$$\cF := [\O_{\Gr_{GL_n}^{k-1}} \otimes \det(L_0/L_1)^\ell] * [\O_{\Gr_{GL_n}^{k+1}} \otimes \det(L_1/L_2)^\ell]$$
has global sections $V_{\ell \omega_{k-1}} \otimes V_{\ell \omega_{k+1}}$ and has no higher cohomology. Consider the map 
$$\om: \Gr_{GL_n}^{(k-1,k+1)} \rightarrow \oGr_{GL_n}^{\omega^\vee}$$
where $\omega^\vee := \omega_{k-1}^\vee + \omega_{k+1}^\vee$. Then 
$$\Gamma(\cF) \cong \Gamma(\om^*(\O_{\oGr_{GL_n}^{\omega^\vee}} \otimes \det(L_0/L_2)^\ell)) \cong \Gamma(\O_{\oGr_{GL_n}^{\omega^\vee}} \otimes \det(L_0/L_2)^\ell).$$
The higher cohomology vanishes since the canonical sheaf of $\oGr_{GL_n}^{\omega^\vee}$ is $\det(L_0/L_2)^{-n} \otimes \det(L_0/tL_0)^{2k}$ (up to an equivariant shift) and $\det(L_0/L_2)$ is ample. 

To compute $H^0(\cF)$ consider the BD deformation of $\Gr_{GL_n}^{(k-1,k+1)}$ over $\A^1$ (c.f. Remark~\ref{rem:BDdeform}). Its general fiber (away from $0 \in \A^1$) is isomorphic to the product $\Gr_{GL_n}^{\omega_{k-1}^\vee} \times \Gr_{GL_n}^{\omega_{k+1}^\vee}$. The line bundle $\det(L_0/L_1)^\ell \otimes \det(L_1/L_2)^\ell$ deforms to the exterior product $\det(L_0/L_1)^\ell \boxtimes \det(L_0/L_1)^\ell$ whose space of global sections is $V_{\ell \omega_{k-1}} \otimes V_{\ell \omega_{k+1}}$. Since $\cF$ has no higher cohomology it then follows (by semi-continuity) that $H^0(\cF)$ is also $V_{\ell \omega_{k-1}} \otimes V_{\ell \omega_{k+1}}$.

A similar argument also proves the third identity. For the second identity we must be a little more careful because the line bundle is {\em not} a pullback from $\oGr^{2 \omega^\vee_k}_{GL_n}$. Let us write
$$\cF := [\O_{\Gr_{GL_n}^{k}} \otimes \det(L_0/L_1)^{\ell+1}] * [\O_{\Gr_{GL_n}^{k}} \otimes \det(L_1/L_2)^{\ell-1}].$$
Then $\cF \cong \om_*(\O_{\Gr_{GL_n}^{(k,k)}} \otimes \cL \otimes \det(L_0/L_2)^{-n})$ where
$$\cL := \det(L_0/L_2)^{\ell-1+n} \otimes \det(L_0/L_1)^2$$
and $\om: \Gr_{GL_n}^{(k,k)} \rightarrow \Gr_{GL_n}^{2 \omega_k^\vee}$. Now, $\det(L_0/L_2)^{\ell-1+n}$ is ample and $\det(L_0/L_1)^2$ is nef (since it is the pullback via $\Gr^{(k,k)}_{GL_n} \rightarrow \Gr^k_{GL_n}$ of the ample line bundle $\det(L_0/L_1)^2$). It follows that $\cL$ is nef. On the other hand, the line bundle $\cL$ deforms to $\wt{\cL} := \det(L_0/L_1)^{\ell+1+n} \boxtimes \det(L_0/L_1)^{\ell-1+n}$ on the general fiber $\Gr^k_{GL_n} \times \Gr^k_{GL_n}$ of the BD deformation. Since $\wt{\cL}$ is ample (and hence big) it follows that $\cL$ is also big. 

Thus $\cL$ is both nef and big. The vanishing of the higher cohomology of $\cF$ now follows by Kawamata-Viehweg vanishing. Finally, $H^0(\cF)$ can be computed as before (using the BD deformation). 
\end{proof}
\begin{Remark}
More generally, one can use the argument in the proof above to show that for a sequence $\ell_1 \ge \dots \ge \ell_m \ge 0$ we have 
$$\Gamma(\cP_{k_1,\ell_1} * \dots * \cP_{k_m,\ell_m}) \cong V_{\ell_1 \omega_{k_1}} \otimes \dots \otimes V_{\ell_m \omega_{k_m}} [\frac12 \sum_i k_i(n-k_i) ].$$
\end{Remark}

Now consider the first short exact sequence in Proposition \ref{prop:mutation}. Taking global sections and using Proposition \ref{prop:cohomology} we obtain (from the associated long exact sequence) a short exact sequence
$$0 \rightarrow V_{(\ell+1) \omega_k} \otimes V_{(\ell-1) \omega_k} \rightarrow V_{\ell \omega_k} \otimes V_{\ell \omega_k} \rightarrow V_{\ell \omega_{k-1}} \otimes V_{\ell \omega_{k+1}} \rightarrow 0$$
for any $\ell \ge 0$. Passing to characters these exact sequences result exactly in the relations that comprise the $Q$-system. 

\appendix

\addtocontents{toc}{\SkipTocEntry}
\section{Unbounded derived categories of quasicoherent sheaves}

\addtocontents{toc}{\SkipTocEntry}
\subsection{Preliminaries}

We gather here some general results on functors between unbounded derived categories of quasicoherent sheaves (denoted $D_{qcoh}(\cdot)$). All schemes will be assumed to be quasi-compact and quasi-separated. % (such a scheme is often called concentrated).
All maps are assumed to be quasi-proper (for maps between Noetherian schemes this is equivalent to being proper). %We follow Lipman \cite{Lip} where this level of generality is assumed.

Given a map $f: X \rightarrow Y$ we have functors $f_* = f_!, f^*, f^!$ going between $D_{qcoh}(X)$ and $D_{qcoh}(Y)$. Here $f^*$ is left adjoint to $f_*$ and $f^!$ the right adjoint of $f_*$. We also have bi-functors $\cHom(\bullet,\bullet)$ and $\otimes$ with $\bullet \otimes \cA$ being the left adjoint of $\cHom(\cA,\bullet)$. In particular, this last adjunction gives us an isomorphism (c.f. \cite[Ex. 3.5.3(e)]{Lip})
\begin{equation}\label{eq:A}
\cHom(\cA \otimes \cB, \cC) \xrightarrow{\sim} \cHom(\cA, \cHom(\cB,\cC)).
\end{equation}

%For a morphism $f: X \rightarrow Y$ with
For $\cB,\cB' \in D_{qcoh}(Y)$ we have a functorial map  (c.f. \cite[Eq. 3.5.4.5]{Lip})
\begin{equation}\label{eq:B}
f^* \cHom(\cB,\cB') \rightarrow \cHom(f^* \cB, f^* \cB'). 
\end{equation}
If $\cA \in D_{qcoh}(X)$ we have an isomorphism
\begin{equation}\label{eq:C}
\alpha: \cHom(\cB, f_* \cA) \xrightarrow{\sim} f_* \cHom(f^* \cB, \cA)
\end{equation}
which can be written as the composition (c.f. \cite[p.105]{Lip})
\begin{equation}\label{eq:D}
\cHom(\cB, f_* \cA) \xrightarrow{adj} f_* f^* \cHom(\cB, f_* \cA) \rightarrow f_* \cHom(f^* \cB, f^* f_* \cA) \xrightarrow{adj} f_* \cHom(f^* \cB, \cA).
\end{equation}

Next, we have an isomorphism (c.f. \cite[Ex. 4.2.3(d)]{Lip})
\begin{equation}\label{eq:E}
\beta: f_* \cHom(\cA, f^! \cB) \xrightarrow{\sim} \cHom(f_* \cA, \cB)
\end{equation}
which can be written as the composition (c.f. \cite[Theorem 4.2]{Lip})
\begin{equation}\label{eq:F}
f_* \cHom(\cA, f^! \cB) \xrightarrow{adj} f_* \cHom( f^* f_* \cA, f^! \cB) \xrightarrow{\alpha^{-1}} \cHom(f_* \cA, f_* f^! \cB) \xrightarrow{adj} \cHom(f_* \cA, \cB).
\end{equation}
We also have a functorial isomorphism (c.f. \cite[Ex. 4.2.3(e)]{Lip})
\begin{equation}\label{eq:G}
\gamma: \cHom(f^* \cB, f^! \cB') \xrightarrow{\sim} f^! \cHom(\cB, \cB')
\end{equation}
which can be written as the composition (c.f. \cite[Ex. 4.2.3(f)]{Lip})
\begin{equation}\label{eq:H}
\cHom(f^* \cB, f^! \cB') \xrightarrow{adj} f^! f_* \cHom(f^* \cB, f^! \cB') \xrightarrow{\alpha^{-1}} f^! \cHom(\cB, f_* f^! \cB') \xrightarrow{adj} f^! \cHom(\cB, \cB').
\end{equation}

Next we restrict to schemes $X$ which are Noetherian and have dualizing complexes $\omega_X = \pi^! \O_{pt}$ where $\pi: X \rightarrow pt$ is the projection to a point. We denote by $\D(\bullet) = \cHom(\bullet, \omega_X)$ the dualizing functor. The adjunction properties of $\cHom$ and $\otimes$ give us a natural transformation $\tau: id \rightarrow \D \D$. This map is an isomorphism when restricted to the subcategory of complexes with coherent cohomology (but is not an isomorphism in general).  

Taking $\cB = \omega_Y$ in (\ref{eq:E}) we obtain an isomorphism $f_* \D \xrightarrow{\sim} \D f_*$. On the other hand, taking $\cB$ arbitrary and $\cB' = \omega_Y$ in (\ref{eq:G}) gives us an isomorphism
$\eta: \D f^* \xrightarrow{\sim} f^! \D$
which can be written as the composition 
$$\eta: \D f^* \xrightarrow{adj} f^! f_* \D f^* \xrightarrow{\sim} f^! \D f_* f^* \xrightarrow{adj} f^! \D.$$

%Note that this description follows from the following diagram 
%$$\xymatrix{
%f_* \cHom(f^* \cB, f^! \cB') \ar[d]^{\gamma} \ar[r]^{\alpha^{-1}} & \cHom(\cB, f_* f^! \cB') \ar[d]^{adj} \\
%f_* f^! \cHom(\cB,\cB') \ar[r]^{adj} & \cHom(\cB,\cB')}
%$$
%whose commutativity is an immediate consequence of (\ref{eq:D}) and (\ref{eq:H}). 

Finally, we denote by $\theta: f^! \rightarrow \D f^* \D$ the composition
$$f^! \xrightarrow{\tau} f^! \D \D \xrightarrow{\eta^{-1}} \D f^* \D$$
where the first map is induced by the natural map $\tau: id \rightarrow \D \D$. Notice that $\theta$ is generally not an isomorphism since $\tau$ is generally not an isomorphism.

\begin{Lemma}\label{lem:B}
For a proper map $f: X \rightarrow Y$ the following diagrams commute
\begin{align*}
& \xymatrix{
\D \ar[d]^{adj} \ar[r]^{adj} & \D f^* f_* \ar[d]^{\eta} \\
f^! f_* \D \ar[r]^{\sim} & f^! \D f_* }
& \xymatrix{
f_* \D f^* \ar[r]^{\eta} \ar[d]^{\sim} & f_* f^! \D \ar[d]^{adj} \\
\D f_* f^* \ar[r]^{adj} & \D }
\end{align*}
\end{Lemma}
\begin{proof}
The composition $\D \rightarrow f^! f_* \D \xrightarrow{\sim} f^! \D f_*$ can be written as
\begin{align*}
\cHom(\cA, f^! \omega_Y) 
& \xrightarrow{adj} f^! f_* \cHom(\cA, f^! \omega_Y) \xrightarrow{adj} f^! f_* \cHom(f^* f_* \cA, f^! \omega_Y) \\
& \xrightarrow{\alpha^{-1}} f^! \cHom(f_* \cA, f_* f^! \omega_Y) \xrightarrow{adj} f^! \cHom(f_* \cA, \omega_Y).
\end{align*}
Likewise the composition $\D \xrightarrow{adj} \D f^* f_* \xrightarrow{\eta} f^! \D f_*$ can be written as 
\begin{align*}
\cHom(\cA, f^! \omega_*) 
& \xrightarrow{adj} \cHom(f^* f_* \cA, f^! \omega_Y) \xrightarrow{adj} f^! f_* \cHom(f^* f_* \cA, f^! \omega_Y) \\
& \xrightarrow{\alpha^{-1}} f^! \cHom(f_* \cA, f_* f^! \omega_Y) \xrightarrow{adj} f^! \cHom(f_* \cA, \omega_Y).
\end{align*}
These two compositions are clearly equivalent which proves that the first square commutes. The second square commutes by a similar argument. 
\end{proof}

The following is an easy Corollary of Lemma \ref{lem:B}.

\begin{Corollary}\label{cor:A}
For a proper map $f: X \rightarrow Y$ the following diagrams commute
\begin{align*}
& \xymatrix{
id \ar[d]^{adj} \ar[r]^{\tau} & \D \D \ar[r]^{adj} & \D f^* f_* \D \ar[d]^{\sim} \\
f^! f_* \ar[rr]^{\theta} & & \D f^* \D f_* }
& \xymatrix{
f_* f^! \ar[d]^{adj} \ar[r]^{\theta} & f_* \D f^* \D \ar[r]^{\sim} & \D f_* f^* \D \ar[d]^{adj} \\
id \ar[rr]^{\tau} & & \D \D. }
\end{align*}
\end{Corollary}

\addtocontents{toc}{\SkipTocEntry}
\subsection{The natural transformation $\phi$} \label{sec:natural}

\begin{Definition}[\cite{Lip}]\label{def:independent}
We say that a fiber product diagram 
\begin{equation}\label{eq:ind}
\begin{gathered}\xymatrix{
Z \ar[r]^{j_1} \ar[d]_{j_2} & X' \ar[d]^{i_2} \\
X \ar[r]^{i_1} & Y }
\end{gathered}
\end{equation}
is \newword{independent} if the composition
\begin{equation}\label{eq:*}
i_2^* i_{1*} \rightarrow i_2^* i_{1*} (j_{2*} j_2^*) = (i_2^* i_{2*}) j_{1*} j_2^* \rightarrow j_{1*} j_2^*
\end{equation}
is an isomorphism.
\end{Definition}
\begin{Remark}
One can check \cite[Prop. 3.7.2]{Lip} that the composition in (\ref{eq:*}) is the same as the composition
$$i_2^* i_{1*} \xrightarrow{adj} (j_{1*} j_1^*) i_2^* i_{1*} = j_{1*} j_2^* (i_1^* i_{1*}) \xrightarrow{adj} j_{1*} j_2^*.$$
\end{Remark}
Condition (\ref{eq:*}) is satisfied, for instance, if $X,X',Y$ are smooth, $i_1,i_2$ are inclusions and $X,X'$ intersect in the expected dimension inside $Y$. Although this condition {\it a priori} seems to depend on an orientation of the square the following Lemma shows otherwise. 

\begin{Lemma}\label{lem:star}
The following four conditions are equivalent.
\begin{enumerate}
\item The square in (\ref{eq:ind}) is independent. 
\item The composition $i_1^* i_{2*} \xrightarrow{adj} (j_{2*} j_2^*) i_1^* i_{2*} = j_{2*} j_1^* (i_2^* i_{2*}) \xrightarrow{adj} j_{2*} j_1^*$ is an isomorphism. 
\item The composition $j_{1*} j_2^! \xrightarrow{adj} j_{1*} j_2^! (i_1^! i_{1*}) = (j_{1*} j_1^!) i_2^! i_{1*} \xrightarrow{adj} i_2^! i_{1*}$ is an isomorphism. 
\item The composition $j_{2*} j_1^! \xrightarrow{adj} j_{2*} j_1^! (i_2^! i_{2*}) = (j_{2*} j_2^!) i_1^! i_{2*} \xrightarrow{adj} i_1^! i_{2*}$ is an isomorphism. 
\end{enumerate}
\end{Lemma}
\begin{proof}
This follows from \cite[Theorem 3.10.3]{Lip} and \cite[Ex. 3.10.4]{Lip}.
\end{proof}

Given an independent square as in (\ref{eq:ind}) we obtain a natural transformation $\phi: j_2^* i_1^! \rightarrow j_1^! i_2^*$ as the composition 
\begin{equation}\label{eq:phi}
j_2^* i_1^! \xrightarrow{adj} (j_1^! j_{1*}) j_2^* i_1^! \xrightarrow{\sim} j_1^! i_2^* (i_{1*} i_1^!) \xrightarrow{adj} j_1^! i_2^*
\end{equation}
where the middle isomorphism is the inverse of $i_2^* i_{1*} \xrightarrow{\sim} j_{1*} j_2^*$ from (\ref{eq:*}).
%\begin{Remark}
The map $\phi$ appears in \cite[Remark 3.10.2.1(c)]{Lip} where it is called $\beta_\sigma$. %As Lipman notes, later on it plays a role in the development of Grothendieck duality. 
%\end{Remark}

\begin{Lemma}\cite[Ex. 3.10.4b]{Lip}\label{lem:phi}
The map $\phi$ is equivalently given by the composition 
\begin{equation}\label{eq:8}
j_2^* i_1^! \xrightarrow{adj} j_2^* i_1^! (i_{2*} i_2^*) \xrightarrow{\sim} (j_2^* j_{2*}) j_1^! i_2^* \xrightarrow{adj} j_1^! i_2^*
\end{equation}
where the middle isomorphism is the inverse of $j_{2*} j_1^! \xrightarrow{\sim} i_1^! i_{2*}$.
\end{Lemma}

\addtocontents{toc}{\SkipTocEntry}
\subsection{Some properties of $\phi$}

\begin{Lemma}\label{lem:2}
The map $\phi$ is natural in the following sense. Suppose one has two independent fiber product diagrams
\begin{equation*}
\xymatrix{
\bullet \ar[r]^{j_1} \ar[d]_{j_2} & \bullet \ar[d]^{i_2} \ar[r]^{g_1} & \bullet \ar[d]^{f_2} \\
\bullet \ar[r]^{i_1} & \bullet \ar[r]^{f_1} & \bullet }
\end{equation*}
Then the fiber product diagram 
\begin{equation}\label{eq:5}
\begin{gathered}\xymatrix{
\bullet \ar[r]^{g_1 \circ j_1} \ar[d]_{j_2} & \bullet \ar[d]^{f_2} \\
\bullet \ar[r]^{f_1 \circ i_1} & \bullet }
\end{gathered}
\end{equation}
is also independent and the corresponding natural map 
\begin{equation}\label{eq:6}
j_2^* (f_1 \circ i_1)^! \xrightarrow{\phi} (g_1 \circ f_1)^! f_2^*
\end{equation} 
is equal to the composition
\begin{equation}\label{eq:7}
j_2^* i_1^! f_1^! \xrightarrow{\phi_1 \circ id} j_1^! i_2^* f_1^! \xrightarrow{id \circ \phi_2} j_1^! g_1^! f_2^*.
\end{equation}
\end{Lemma}
\begin{proof}
The first part of this Lemma appears as \cite[Lemma 3.10.3.2]{Lip} while the second part is proved in \cite[Prop. 4.6.8]{Lip}. 
\end{proof}

\begin{Proposition}\cite[Ex. 4.2.3i]{Lip}\label{prop:phi}
Suppose the square in (\ref{eq:ind}) is independent and take $\cB,\cB' \in D_{qcoh}(Y)$. Then the following diagram commutes
\begin{equation}\label{eq:commute1}
\begin{gathered}\xymatrix{
j_2^* i_1^! \cHom(\cB,\cB') \ar[d]^{\phi} \ar[r]^{\gamma^{-1}} & j_2^* \cHom(i_1^* \cB, i_1^! \cB') \ar[r] & \cHom(j_2^* i_1^* \cB, j_2^* i_1^! \cB') \ar[d]^{\phi} \\
j_1^! i_2^* \cHom(\cB, \cB) \ar[r] & j_1^! \cHom( i_2^* \cB, i_2^* \cB') \ar[r]^{\gamma^{-1}} & \cHom(j_1^* i_2^* \cB, j_1^! i_2^* \cB') 
}
\end{gathered}
\end{equation}
where in the right hand vertical map we use $\phi$ and that $j_2^* i_1^* (\cB) = j_1^* i_2^* (\cB)$. 
\end{Proposition}

\begin{Proposition}\label{prop:dual}
Suppose the square in (\ref{eq:ind}) is independent and consider maps $\rho, \rho'$ given as the compositions
\begin{align} 
\label{eq:12} 
& j_1^* \D i_2^* \D \xrightarrow{\theta^{-1}} j_1^* i_2^! \xrightarrow{\phi} j_2^! i_1^* \xrightarrow{\theta} \D j_2^* \D i_1^* \\
\label{eq:13}
& j_2^* \D i_1^* \D \xrightarrow{\theta^{-1}} j_2^* i_1^! \xrightarrow{\phi} j_1^! i_2^* \xrightarrow{\theta} \D j_1^* \D i_2^*
\end{align}
of functors $D_{coh}(Y) \rightarrow D_{qcoh}(Z)$. Then $\rho'$ is equal to the composition 
$$j_2^* \D i_1^* \D \xrightarrow{\tau} \D \D j_2^* \D i_1^* \D \xrightarrow{\D \rho \D} \D j_1^* \D i_2^* \D \D \xrightarrow{\tau^{-1}} \D j_1^* \D i_2^*.$$
Notice that the $\tau^{-1}$ exists since the domain is $D_{coh}(Y)$.  
\end{Proposition}
\begin{proof}
Starting with $\D \rho$ and writing out $\phi$ leaves us with 
\begin{equation}\label{eq:14}
\D \D j_2^* \D i_1^* \xrightarrow{\theta} \D j_2^! i_1^* \xrightarrow{adj} \D j_1^* j_{1*} j_2^! i_1^* \xrightarrow{\sim} \D j_1^* i_2^! i_{1*} i_1^* \xrightarrow{adj} \D j_1^* i_2^! \xrightarrow{\theta^{-1}} \D j_1^* \D i_2^* \D.
\end{equation}
By Lemma \ref{lem:4} one has the following commutative diagram 
$$\xymatrix{
\D j_1^* (j_{1*} j_2^!) i_1^* \ar[rrr]^{\sim} & & & \D j_1^* (i_2^! i_{1*}) i_1^* \ar[r]^{adj} & \D j_1^* i_2^! \\
\D j_1^* j_{1*} \D j_2^* \D i_1^* \ar[r]^{\sim} \ar[u]_{\theta} & \D j_1^* \D j_{1*} j_2^* \D i_1^* \ar[r]^{\sim} & \D j_1^* \D i_2^* i_{1*} \D i_1^* \ar[r]^{\sim} & \D j_1^* \D i_2^* \D i_{1*} i_1^* \ar[r]^{adj} \ar[u]^{\theta} & \D j_1^* \D i_2^* \D \ar[u]_{\theta}
} $$
This means that we can rewrite (\ref{eq:14}) as the following composition 
\begin{align*}
\D \D j_2^* \D i_1^* \xrightarrow{adj} \D j_1^* j_{1*} \D j_2^* \D i_1^* & \xrightarrow{\sim} \D j_1^* \D j_{1*} j_2^* \D i_1^* \\ 
& \xrightarrow{\sim} \D j_1^* \D i_2^* i_{1*} \D i_1^* \xrightarrow{\sim} \D j_1^* \D i_2^* \D i_{1*} i_1^* \xrightarrow{adj} \D j_1^* \D i_2^* \D
\end{align*}
Precomposing with $j_2^* \D i_1^* \rightarrow \D \D j_2^* \D i_1^*$ and using the left hand diagram in Corollary \ref{cor:A} this composition becomes 
$$j_2^* \D i_1^* \xrightarrow{adj} j_1^! j_{1*} j_2^* \D i_1^* \xrightarrow{\theta} \D j_1^* \D j_{1*} j_2^* \D i_1^* \xrightarrow{\sim} \D j_1^* \D i_2^* i_{1*} \D i_1^* \xrightarrow{\sim} \D j_1^* \D i_2^* \D i_{1*} i_1^* \xrightarrow{adj} \D j_1^* \D i_2^* \D$$
which can easily be rewritten as 
$$j_2^* \D i_1^* \xrightarrow{adj} j_1^! j_{1*} j_2^* \D i_1^* \xrightarrow{\sim} j_1^! i_2^* i_{1*} \D i_1^* \xrightarrow{\sim} j_1^! i_2^* \D i_{1*} i_1^* \xrightarrow{adj} j_1^! i_2^* \D \xrightarrow{\theta} \D j_1^* \D i_2^* \D.$$
Now compose everything on the right with $\D$ and then precompose with 
$$j_2^* \D i_1^* \D \xrightarrow{\theta^{-1}} j_2^* i_1^! \xrightarrow{\theta} j_2^* \D i_1^* \D$$
We can rearrange this to give
\begin{align*}
j_2^* \D i_1^* \D \xrightarrow{\theta^{-1}} j_2^* i_1^! & \xrightarrow{adj} j_1^! j_{1*} j_2^* i_1^! \\
& \xrightarrow{\sim} j_1^! i_2^* i_{1*} i_1^! \xrightarrow{\theta} j_1^! i_2^* i_{1*} \D i_1^* \D \xrightarrow{\sim} j_1^! i_2^* \D i_{1*} i_1^* \D \xrightarrow{adj} j_1^! i_2^* \D \D \xrightarrow{\theta} \D j_1^* \D i_2^* \D \D.
\end{align*}
Using the right hand diagram in Corollary \ref{cor:A} this is equal to 
$$j_2^* \D i_1^* \D \xrightarrow{\theta^{-1}} j_2^* i_1^! \xrightarrow{adj} j_1^! j_{1*} j_2^* i_1^! \xrightarrow{\sim} j_1^! i_2^* i_{1*} i_1
^! \xrightarrow{adj} j_1^! i_2^* \xrightarrow{\theta} \D j_1^* \D i_2^* \rightarrow \D j_1^* \D i_2^* \D \D$$
which is just 
\begin{equation}\label{eq:15}
j_2^* \D i_1^* \D \xrightarrow{\theta^{-1}} j_2^* i_1^! \xrightarrow{\phi} j_1^! i_2^* \xrightarrow{\theta} \D j_1^* \D i_2^* \xrightarrow{\tau} \D j_1^* \D i_2^* \D \D.
\end{equation}
The result follows since this shows that (\ref{eq:15}) is equal to the composition
$$j_2^* \D i_1^* \D \xrightarrow{\tau} \D \D j_2^* \D i_1^* \D \xrightarrow{\D \rho \D} \D j_1^* \D i_2^* \D \D.$$
\end{proof}

\begin{Lemma}\label{lem:4}
Assuming the square in (\ref{eq:ind}) is independent the following diagram commutes
\begin{equation}\label{eq:16}
\begin{gathered}\xymatrix{
i_2^! i_{1*} \ar[d]^{\theta} \ar[rrr]^{\sim} & & & j_{1*} j_2^! \ar[d]^{\theta} \\
\D i_2^* \D i_{1*} \ar[r]^{\sim} & \D i_2^* i_{1*} \D \ar[r]^{\sim} & \D j_{1*} j_2^* \D \ar[r]^{\sim} & j_{1^*} \D j_2^* \D 
}
\end{gathered}
\end{equation}
\end{Lemma}
\begin{proof}
Notice that the more natural morphism is $j_{1*} j_2^! \xrightarrow{\sim} i_2^! i_{1*}$ while the one appearing above is its inverse. So what we need to show is that the following diagram commutes
\begin{equation}\label{eq:17}
\begin{gathered}\xymatrix{
j_{1*} j_2^! \ar[r]^{adj} \ar[d]^{\theta} & j_{1*} j_2^! i_1^! i_{1*} \ar[rrr]^= & & & j_{1*} j_1^! i_2^! i_{1*} \ar[r]^{adj} & i_2^! i_{1*} \ar[d]^\theta \\
j_{1*} \D j_2^* \D \ar[r]^{\sim} & \D j_{1*} j_2^* \D \ar[r]^{adj} & \D j_{1*} j_2^* i_1^* i_{1*} \D \ar[r]^{=} & \D j_{1*} j_1^* i_2^* i_{1*} \D \ar[r]^{adj} & \D i_2^* i_{1*} \D \ar[r]^{\sim} & \D i_2^* \D i_{1*} 
}
\end{gathered}
\end{equation}
We can rewrite this diagram as follows. 
\begin{align}\label{eq:18}
  \begin{gathered}
  \xymatrix{
j_{1*} j_2^! \ar[dd]^{\theta} \ar[rr]^{adj}& & j_{1*} j_2^! i_1^! i_{1*} \ar[d]^{\theta} \ar[r]^{=} & j_{1*} j_1^! i_2^! i_{1*} \ar[d]^{\theta} \ar[rr]^{adj} & & i_2^! i_{1*} \ar[dd]^{\theta} \\  
& & j_{1*} j_2^! \D i_1^* \D i_{1*} \ar[d]^{\eta^{-1}} & j_{1*} j_1^! \D i_2^* \D i_{1*} \ar[d]^{\eta^{-1}} & & \\
j_{1*} \D j_2^* \D \ar[r]^{adj} & j_{1*} \D j_2^* i_1^* i_{1*} \D \ar[r]^{\sim} & j_{1*} \D j_2^* i_1^* \D i_{1*} \ar[r]^{=} & j_{1*} \D j_1^* i_2^* \D i_{1*} \ar[r]^{\sim} & \D j_{1*} j_1^* i_2^* \D i_{1*} \ar[r]^{adj} & \D i_2^* \D i_{1*} 
}
\end{gathered}
\end{align}
It is not hard to check that the compositions along the bottom row in (\ref{eq:17}) and (\ref{eq:18}) agree. Now consider the following diagram.
$$\xymatrix{
j_2^! \ar[d]^{\tau} \ar[rr]^{adj} & & j_2^! i_1^! i_{1*} \ar[d]^{\theta} \\
j_2^! \D \D \ar[r]^{adj} \ar[d]^{\eta^{-1}} & j_2^! \D i_1^* i_{1*} \D \ar[d]^{\eta^{-1}} \ar[r]^{\sim} & j_2^! \D i_1^* \D i_{1*} \ar[d]^{\eta^{-1}} \\
\D j_2^* \D \ar[r]^{adj} & \D j_2^* i_1^* i_{1*} \D \ar[r]^{\sim} & \D j_2^* i_1^* \D i_{1*}
}$$
The top rectangle commutes by the left diagram in Corollary \ref{cor:A} while the other two squares commute for obvious reasons. This proves the commutativity of the left hand square in (\ref{eq:18}). Similarly, using the right diagram in Corollary \ref{cor:A}, one can prove the commutativity of the right hand square in (\ref{eq:18}). Finally, to see that the middle rectangle commutes note that the composition 
$$j_2^! i_1^! \xrightarrow{\theta} j_2^! \D i_1^* \D \xrightarrow{\eta^{-1}} \D j_2^* i_1^* \D$$
is by definition 
$$j_2^! i_1^! \xrightarrow{\tau} j_2^! i_1^! \D \D \xrightarrow{\eta^{-1}} j_2^! \D i_1^* \D \xrightarrow{\eta^{-1}} \D j_2^* i_1^* \D,$$
which is nothing but $(i_1 \circ j_2)^! \xrightarrow{\theta} \D (i_1 \circ j_2)^* \D$. Likewise the composition 
$$j_1^! i_2^! \xrightarrow{\theta} j_1^! \D i_2^* \D \xrightarrow{\eta^{-1}} \D j_1^* i_2^* \D$$ 
is just $(i_2 \circ j_1)^! \xrightarrow{\theta} \D (i_2 \circ j_1)^* \D$. The commutation of (\ref{eq:18}) (and hence of (\ref{eq:17})) follows. 
\end{proof}

\bibliographystyle{amsalpha}
\bibliography{CSbib}

\end{document}